\documentclass[a4paper,12pt]{amsart}

\usepackage[english]{babel}
\usepackage{amsmath,amsfonts,amssymb,amsthm,relsize,setspace,nicefrac,yhmath,amscd,eucal}
\usepackage{amsrefs,mathrsfs}
\usepackage{mathtools,tkz-euclide,tikz-3dplot,tkz-tab}

\usepackage{tikz}
\usetikzlibrary{calc}

\usepackage{xcolor}
\usepackage{soul} % \st{}, \hl{}, etc.

\usepackage{enumitem}

\usepackage{comment}
\usepackage{anyfontsize}

\usepackage[colorlinks,linkcolor=red,citecolor=blue,urlcolor=blue,hypertexnames=true]{hyperref}

\usepackage[active]{srcltx}

\usepackage[toc,page]{appendix}

\newtheorem{theorem}{Theorem}

\newtheorem{lemma}[theorem]{Lemma}
\newtheorem{proposition}[theorem]{Proposition}

\newtheorem{remark}{Remark}

\theoremstyle{definition}
\newtheorem{definition}[theorem]{Definition}

\newcommand{\be}{\begin{equation}}
\newcommand{\bel}[1]{\begin{equation}\label{#1}}
\newcommand{\ee}{\end{equation}}

\newcommand{\barr}{\begin{eqnarray}}
\newcommand{\earr}{\end{eqnarray}}
\newcommand{\bars}{\begin{eqnarray*}}
\newcommand{\ears}{\end{eqnarray*}}

\newtheorem{subn}{\name}

\newcommand{\bsn}[1]{\def\name{#1}\begin{subn}}
\newcommand{\esn}{\end{subn}}

\newtheorem{sub}{\name}[section]

\newcommand{\bs}{\begin{sub}}
\newcommand{\es}{\end{sub}}

\newcommand{\bth}[1]{\def\name{Theorem}\begin{sub}\label{t:#1}}
\newcommand{\blemma}[1]{\def\name{Lemma}\begin{sub}\label{l:#1}}
\newcommand{\bcor}[1]{\def\name{Corollary}\begin{sub}\label{c:#1}}
\newcommand{\bdef}[1]{\def\name{Definition}\begin{sub}\label{d:#1}}
\newcommand{\bprop}[1]{\def\name{Proposition}\begin{sub}\label{p:#1}}

\newcommand{\BA}{\begin{array}}
\newcommand{\EA}{\end{array}}
\newcommand{\BAN}{\renewcommand{\arraystretch}{1.2}
\setlength{\arraycolsep}{2pt}\begin{array}}
\newcommand{\BAV}[2]{\renewcommand{\arraystretch}{#1}
\setlength{\arraycolsep}{#2}\begin{array}}
\newcommand{\BSA}{\begin{subarray}}
\newcommand{\ESA}{\end{subarray}}

\newcommand{\BAL}{\begin{aligned}}
\newcommand{\EAL}{\end{aligned}}
\newcommand{\BALG}{\begin{alignat}}
\newcommand{\EALG}{\end{alignat}}
\newcommand{\BALGN}{\begin{alignat*}}
\newcommand{\EALGN}{\end{alignat*}}
\def\angb<#1>{\langle #1 \rangle}%% angle bracket
\newcommand {\rd}{\color{red}}
\def\R{\mathbb{R}}

\let\.=\cdot
\let\0=\emptyset

\numberwithin{equation}{section}

\theoremstyle{definition}

\let\.=\cdot
\let\0=\emptyset

\newenvironment{formula}[1]{\begin{equation}\label{eq:#1}}
                       {\end{equation}\noindent}

\def\Fi#1{\begin{formula}{#1}}
\def\Ff{\end{formula}\noindent}

\setstcolor{red}

\usepackage{authblk}
\usepackage{changepage}

\title[]{Long-time dynamics and threshold Phenomena for a free-boundary SIS Model with asymmetric kernels in advective periodic environments}

\makeatletter
\author[1]{Soufiane Bentout} \let\Author\@author

\affil[1]{Department of Mathematics and Informatics, University of  Ain Temouchent,\\ Belhadj Bouchaib, BP 284 RP, 46000, Algeria\\
Engineering and Sustainable Development Laboratory, Faculty of Science and Technology,  University of Ain Temouchent, Ain Temouchent, 46000, Algeria}

\author[$2,*$]{Hoang-Hung Vo}
\affil[$2$]{Faculty of Mathematics and Applications, Saigon University, 273 An Duong Vuong st., Ward Choquan, Ho Chi Minh City, Viet Nam}
\email{\rd vhhungkhtn@gmail.com}
\email{soufiane.bentout@univ-temouchent.edu.dz}
\begin{document}

%%%%%%%%%%%%
%%Setting up the TITLE and AUTHOR
\date{\today}


%  \textheight=8 true in
%   \textwidth=6 true in
%   \oddsidemargin=-0.8 cm
%   \evensidemargin=-0.8 cm
\maketitle
\begin{adjustwidth}{1.5cm}{1.5cm}
\begin{center}
\let\thefootnote\relax\footnotetext{\textit{}}

\let\thefootnote\relax\footnotetext{\textit{$^{2,*}$Corresponding author. Email address: vhhungkhtn@gmail.com}}

\Author
\end{center}
\end{adjustwidth}

%  \textheight=9 true in
%   \textwidth=7 true in
%   \oddsidemargin=-0.8 cm
%   \evensidemargin=-0.8 cm
%\maketitle

\begin{abstract}

We study a nonlocal SIS epidemic model with free boundaries, advection, and spatial heterogeneity, where the dispersal kernels are not assumed to be symmetric. The model describes the evolution of susceptible and infected populations in a bounded infected habitat whose endpoints move according to nonlocal boundary fluxes. We aim to determine the sharp threshold between disease spreading and vanishing and to characterize the long-time behavior of solutions. The analysis encounters several fundamental difficulties. First, the linearization around the disease-free equilibrium leads to a genuinely coupled nonlocal system with drift, so the relevant spectral quantity cannot be reduced directly to a standard scalar eigenvalue problem. Moreover, because of the advection terms and the possible non-symmetry of the kernels, the associated operators are non-self-adjoint and admit no useful variational formula; in particular, classical Rayleigh quotient and minimax arguments are not available. To overcome these difficulties, we deeply employ the notion generalized principal eigenvalue theory for nonlocal operators developed by Coville and Hamel in \cite{CovilleHamel2020}, together with the Harnack inequality for non-symmetric nonlocal operators established there.  This non-variational technique is particularly well suited to our setting, where advection terms and non-symmetric kernels destroy self-adjointness and rule out standard spectral methods.  Combined with comparison principles, sub- and supersolution constructions, and uniform estimates on time-dependent spatial intervals, this approach enables us to derive the precise asymptotic behavior of the generalized principal eigenvalue with respect to the spatial domain and the diffusion rate, identify the sharp threshold and the critical habitat size, and determine exactly the long-time dynamics of $S$ and $I$ via  $\omega$-limit set approach. To the best of our knowledge, this is the first work on a free-boundary SIS epidemic model with non-symmetric nonlocal dispersal kernels in a spatially periodic environment with advection.
\end{abstract}

\section{Introduction and main results}

Spatially structured epidemic models have been extensively used to describe the transmission and spread of infectious diseases in heterogeneous environments.
Among them, SIS (susceptible--infected--susceptible) models play a fundamental role in modeling diseases for which recovery does not confer permanent immunity; see, for instance, \cite{Murray2002,Zhao2014,WangZhao2019} and the references therein.
When spatial movement of individuals is taken into account, reaction--diffusion systems naturally arise, and the spreading behavior of epidemics becomes closely related to the spectral properties of the underlying dispersal operators.

In many realistic situations, however, individual movement cannot be adequately described by local diffusion alone.
Long-range dispersal, directional migration, and biased movement induced by environmental heterogeneity often lead to nonlocal spatial interactions.
Such effects are more appropriately modeled by nonlocal diffusion operators of convolution type, which have attracted considerable attention in recent years.
Compared with local diffusion, nonlocal dispersal may generate qualitatively different spreading behaviors, including directional propagation, anisotropic spreading, and the loss of classical variational characterizations of principal eigenvalues.

Another important feature in epidemic propagation is the spatial expansion of the habitat.
Free-boundary problems provide a natural framework to describe invasion fronts driven by population fluxes, and they have been successfully applied to reaction--diffusion and epidemic models; see, among others, \cite{Du2015,Du2021,Du2022,DuNi2024}.
In contrast to Cauchy problems posed on the whole space, free-boundary formulations allow one to capture spreading--vanishing dichotomies and to rigorously characterize the long-time spatial dynamics of populations.

Motivated by the above considerations, we study in this paper the following SIS epidemic model with nonlocal dispersal, nonlinear incidence, and free boundaries:
\begin{equation}\label{eq:SIS-free-boundary}
\begin{cases}
\displaystyle
S_t(t,x)
= d_1\!\left(\int_{\Omega_t} J_1(x-y)\,S(t,y)\,dy - S(t,x)\right) +  a(x) S_x+ \gamma(x) I(t,x)- F\big(S(t,x),I(t,x)\big),
&  x\in\Omega_t, \\[2pt]

\displaystyle
I_t(t,x)
= d_2\!\left(\int_{\Omega_t}  J_2(x-y)\,I(t,y)\,dy - I(t,x)\right)   + b(x) I_x- \gamma(x) I(t,x)+ F\big(S(t,x),I(t,x)\big),
& x\in \Omega_t, \\[2pt]

\displaystyle
S(t,x)=I(t,x)=0,  t \ge 0,\; x \in \mathbb{R}\setminus \Omega_t.
  \\[8pt]

\displaystyle
h'(t)
= \mu \int_{\Omega_t}\int_{h(t)}^{\infty} J_1(x-y)\,S(t,x)\,dy\,dx,
&  \\[8pt]

\displaystyle
g'(t)
= -\mu \int_{\Omega_t}\int_{-\infty}^{g(t)} J_1(x-y)\,S(t,x)\,dy\,dx,
&  \\[8pt]

\displaystyle
h(0)=h_0,\quad g(0)=-h_0,\qquad S(0,x)=S_0(x),\quad I(0,x)=I_0(x), & x\in\mathbb{R},
\end{cases}
\end{equation}
where $\Omega_t=(g(t),h(t))$ for $t>0$, the parameters $d_i$, $\gamma$, $\alpha$, $\mu$, $\mu_1$, $\mu_2$,
  are fixed nonnegative constants.

A distinctive feature of system \eqref{eq:SIS-free-boundary} is that we do \emph{not} impose any symmetry assumption on the dispersal kernels $J_1$ and $J_2$.
Throughout the paper, the kernels are only assumed to be nonnegative and integrable, without requiring $J_i(x)=J_i(-x)$ for $i=1,2$.
This framework allows for biased dispersal and directional movement, which naturally arise in many realistic situations such as prevailing winds, river flows, transportation networks, or human mobility patterns.
From a mathematical point of view, the lack of symmetry implies that the associated nonlocal operators are no longer self-adjoint, and classical variational methods for principal eigenvalues are no longer applicable.

\subsection*{Interpretation of the model}

The functions $S(t,x)$ and $I(t,x)$ denote the densities of susceptible and infected individuals at time $t$ and spatial location $x$, respectively.
The kernels $J_1$ and $J_2$ describe the probability distributions of nonlocal dispersal for susceptibles and infectives.
Accordingly, the terms
\[
d_1\!\left(\int_{g(t)}^{h(t)} J_1(x-y)\,S(t,y)\,dy - S(t,x)\right),
\quad
d_2\!\left(\int_{\Omega_t} J_2(x-y)\,I(t,y)\,dy - I(t,x)\right)
\]
represent the random movement of individuals through nonlocal dispersal.

The function $F(S,I)$ is a general nonlinear incidence function describing the infection process.
Typical examples include the bilinear incidence $F(S,I)=\beta SI$ and the saturated incidence
\[
F(S,I)=\frac{\beta SI}{1+\alpha I+\gamma S},
\]
which accounts for behavioral or medical constraints in disease transmission.
The recovery of infective individuals is modeled by the term $-\gamma(x)I(t,x)$, where $\gamma(x)>0$ denotes the recovery rate.

The boundary conditions \(S(t,x)=I(t,x)=0\) for \(x\notin(g(t),h(t))\) indicate that the population is confined within the evolving habitat \(\Omega_t\).
The free boundaries \(g(t)\) and \(h(t)\) represent the spatial spreading fronts of the host population and are driven by the outward flux of susceptibles across the habitat edges, namely,
\[
h'(t)
= \mu \int_{\Omega_t}\int_{h(t)}^{\infty} J_1(x-y)\,S(t,x)\,dy\,dx,
\qquad
g'(t)
= -\mu \int_{\Omega_t}\int_{-\infty}^{g(t)} J_1(x-y)\,S(t,x)\,dy\,dx,
\]
where \(\mu>0\) measures the expansion capability of the habitat. The fact that the free boundaries depend only on \(S\) and on the kernel \(J_1\), rather than on both \(S\) and \(I\), reflects that the expansion law is determined by the outward nonlocal dispersal flux of susceptible individuals. More precisely, \(J_1\) represents the jump distribution of susceptibles, and the quantities
\[
\int_{\Omega_t}\int_{h(t)}^\infty J_1(x-y)S(t,x)\,dy\,dx,
\qquad
\int_{\Omega_t}\int_{-\infty}^{g(t)} J_1(x-y)S(t,x)\,dy\,dx
\]
measure the total tendency of susceptibles to disperse beyond the current right and left habitat edges, respectively. Biologically, this means that habitat expansion is driven by healthy or sufficiently mobile hosts capable of colonizing new territory, whereas infected individuals contribute to disease transmission inside the occupied region but not to the creation of new habitable space. This is consistent with the usual epidemiological assumption that infection reduces mobility, survival, or colonization ability, so that infected individuals do not effectively determine the advancing edge of the population range. From a modeling viewpoint, the free boundaries therefore track the invasion front of the host population through the susceptible dispersal mechanism encoded by \(J_1\), while the infected class evolves within the habitat generated by the susceptibles. Mathematically, this choice is also natural because the susceptible component acts as the leading population that controls the support of the solution. The infected population is then transported and transmitted inside the moving domain \(\Omega_t\), but its presence does not alter the geometric law for the boundary. This separation allows the model to capture the interaction between host expansion and disease dynamics while keeping the free-boundary mechanism biologically meaningful and analytically tractable. The initial habitat is given by \([ -h_0,h_0]\), with initial population distributions \(S_0(x)\) and \(I_0(x)\).

\medskip

The absence of symmetry in the dispersal kernels $J_1$ and $J_2$ gives rise to several intertwined technical difficulties.
On the spectral theory aspect, the associated nonlocal operators are no longer self-adjoint, so that classical variational characterizations of principal eigenvalues are unavailable.
In particular, standard Rayleigh quotient methods and symmetry-based compactness arguments cannot be applied, and the identification of the correct threshold quantity governing invasion or extinction becomes highly nontrivial. These difficulties are further amplified by the presence of advection terms and spatially periodic coefficients.
The linearization of system \eqref{eq:SIS-free-boundary} around the disease-free equilibrium leads to a genuinely coupled nonlocal system, for which the spectral analysis cannot be reduced to a standard scalar eigenvalue problem.
In particular, because of the coupling structure, the drift terms, and the possible non-symmetry of the dispersal kernels, the associated linearized operator is not self-adjoint and does not admit any usable variational characterization.
As a consequence, the principal spectral quantity cannot be extracted through a Rayleigh quotient, minimax formula, or Hilbert-space argument, and a direct reduction to the infected component alone is not available at the outset.

A central technical novelty of the present work is the deeply employ the notion of generalized principal eigenvalue theory and the Harnack inequality for nonlocal operators developed by Coville and Hamel in~\cite{CovilleHamel2020}. This framework is essential in our setting because the linearization of the SIS system around the disease-free equilibrium leads to a genuinely coupled nonlocal system with advection, while the free-boundary formulation introduces additional difficulties through the nonlocal motion of the evolving habitat. After rewriting the linearized system as a block operator, we employ a Schur complement reduction to isolate the effective scalar spectral quantity governing the infected dynamics. However, even after this reduction, the resulting operator still lies outside the variational framework: because of the first-order drift terms and the possible non-symmetry of the dispersal kernels, no Rayleigh quotient, minimax principle, or self-adjoint spectral theory is available. For this reason, the threshold analysis must be carried out through the generalized principal eigenvalue in the sense of Coville and Hamel, which is defined by means of suitable classes of positive test functions rather than by any variational formula. The argument requires a skillful construction of admissible test functions adapted to the operator, since they must simultaneously capture the nonlocal Dirichlet-type exterior condition, the advection effects, the boundary loss, and the heterogeneous coefficients. In this way, the test-function method becomes the key tool for establishing the main properties of the generalized principal eigenvalue, including its monotonicity with respect to coefficients and domains, its asymptotic behavior under domain variation and diffusion limits, and the precise relation between its sign and the basic reproduction number. Thus, the Harnack inequality in~\cite{CovilleHamel2020} provides the crucial positivity and compactness control needed in the spectral analysis of such non-symmetric operators. Combined with comparison principles, sub- and supersolution constructions, uniform a priori estimates, and compactness arguments on time-dependent intervals, this operator-theoretic framework enables us to connect the fixed domain spectral threshold with the nonlinear free-boundary dynamics, and ultimately to establish the sharp spreading--vanishing dichotomy and the long-time behavior of solutions in a fully non-symmetric setting.

\subsection{Hypotheses}

Throught this paper, we assume  the following conditions :

\begin{itemize}
  \item \textbf{(J1)} The kernels $J_i \ge 0$ are continuous (or integrable) and satisfy
  \[
  \int_{\mathbb R} J_i(x)\,dx = 1, \qquad i=1,2.
  \]
  \[
  \iint_{(-\infty,0)\times(0,\infty)} J(x-y)\,dy\,dx
  \;=\; \int_0^\infty x\,J(x)\,dx \;<\; +\infty.
  \]

  \item \textbf{(J2)}\quad
  There exists $\lambda>0$ such that
  \[
  \int_{-\infty}^{\infty} e^{\lambda x}\, J_i(x)\,dx \;<\; +\infty,
  \qquad i=1,2.
  \]

  \item \textbf{(H1)}\quad
  The incidence function $F:\mathbb R_+^2 \to \mathbb R_+$ is locally Lipschitz.
  Moreover, $F \in C^{1}(\mathbb R_+^2)$ satisfies
  \[
  F(S,0)=0,
  \qquad
  \frac{\partial F}{\partial S}(S,I) \ge 0,
  \qquad
  \frac{\partial F}{\partial I}(S,I) \ge 0,
  \]
  for all $(S,I)$ in the relevant state space.
  Consequently, the system is \emph{quasi-monotone}.

  \item \textbf{(H2)}\quad
  The coefficients $a(\xi)$, $b(\xi)$, and $\gamma(\xi)$ are strictly positive, continuous and bounded
  on $(-\infty,0]$ and \text{are stictly positive $\ell$-periodic functions in } $x$. We assume that
\[
\max\Big\{\sup_{\xi\le 0} a(\xi),\ \sup_{\xi\le 0} b(\xi)\Big\} < c
\]
for some constant $c>0$.

  \item \textbf{(H3)}\quad
  The initial functions satisfy
  \[
  S_0,  I_0 \in C([-h_0,h_0]), \quad S_0 \ge 0,
 \quad I_0 \ge 0.
  \]
\end{itemize}

\subsection{Main results}

\begin{theorem}[Well-posedness]\label{Theorem1}
Assume that \rm{(J1)}-\rm{(J2)}, {\rm(H2)}--{\rm(H3)} hold.
Then, for any $T>0$, the free-boundary problem \eqref{eq:SIS-free-boundary}
admits a unique classical solution
\[
(S,I,g,h)
\]
on $[0,T]$ in the sense of the functional setting introduced above.
Moreover, there exist positive constants $M_S$ and $M_I$, independent of $T$, such that
\[
0 \le S(t,x) \le M_S,
\qquad
0 \le I(t,x) \le M_I,
\quad \text{for all } (t,x)\in [0,T]\times\mathbb R.
\]
\end{theorem}

Next, to simplify notation, for any bounded interval $(L_1,L_2)\subset\mathbb R$, we write
\[
\mathcal L_{(L_1,L_2),d_2,b}[\phi](x)
:=d_2\int_{L_1}^{L_2}J_2(x-y)\phi(y)\,dy-d_2\phi(x)+b(x)\phi'(x),
\qquad x\in(L_1,L_2),
\]
for the nonlocal dispersal--advection operator acting on the infected component.

We  introduce the basic reproduction number associated with the fixed-domain linearized infected subsystem. Let
\[
\beta(x):=F_I\big(S^*(x),0\big)-\gamma(x),
\]
where $S^*(x)$ denotes the disease-free profile on the corresponding bounded interval. For each bounded interval $(L_1,L_2)\subset\mathbb R$, define
\[
\mathcal A_I^{(L_1,L_2)}[\phi](x)
:=\mathcal L_{(L_1,L_2),d_2,b}[\phi](x)-\gamma(x)\phi(x),
\]
and
\[
\mathcal F^{(L_1,L_2)}[\phi](x)
:=F_I\big(S^*(x),0\big)\phi(x).
\]
Set
\[
\mathcal V^{(L_1,L_2)}:=-\mathcal A_I^{(L_1,L_2)}.
\]
Assuming that $\mathcal V^{(L_1,L_2)}$ is invertible and that
\[
\mathcal K^{(L_1,L_2)}
:=\big(\mathcal V^{(L_1,L_2)}\big)^{-1}\mathcal F^{(L_1,L_2)}
\]
is a well-defined positive compact operator, we define the basic reproduction number on $(L_1,L_2)$ by
\[
\mathcal R_0^{(L_1,L_2)}:=r\!\left(\mathcal K^{(L_1,L_2)}\right),
\]
where $r(\cdot)$ denotes the spectral radius.

\begin{theorem}[Spreading--vanishing theorem]
\label{thm:unified-dichotomy} Assume that \rm{(J1)}-\rm{(J2)}, {\rm(H2)}--{\rm(H3)} hold. Let $(S,I;g,h)$ be the unique global solution of \eqref{eq:SIS-free-boundary}. If $h_\infty-g_\infty<+\infty$, then
\[
\mathcal R_0^{(g_\infty,h_\infty)}\le 1
\]
and
\[
\lim_{t\to\infty}\max_{x\in[g(t),h(t)]}I(t,x)=0,\qquad
\lim_{t\to\infty}S(t,x)=S^*(x)\ \text{uniformly on }[g(t),h(t)].
\]
If, in addition, $\sup_{x\in\mathbb R}\beta(x)>0,$ then exactly one of the following alternatives occurs:
\[
\text{\rm (i) } -g_\infty=h_\infty=+\infty
\quad\text{and}\quad
\limsup_{t\to\infty}\|I(t,\cdot)\|_{C([g(t),h(t)])}>0;
\]
or
\[
\text{\rm (ii) } h_\infty-g_\infty<+\infty,\qquad
\mathcal R_0^{(g_\infty,h_\infty)}\le 1,
\]
and
\[
\lim_{t\to\infty}\max_{x\in[g(t),h(t)]}I(t,x)=0,\qquad
\lim_{t\to\infty}S(t,x)=S^*(x)\ \text{uniformly on }[g(t),h(t)].
\]
\end{theorem}

 \begin{remark}
The asymptotic behavior of the susceptible component is fundamentally different in the vanishing and spreading regimes. In the vanishing case, one has
\[
I(t,\cdot)\to 0 \qquad \text{as } t\to\infty,
\]
so that the coupling terms involving the infected population disappear asymptotically. As a consequence, the $S$-equation becomes asymptotically autonomous and reduces, in the limit, to the disease-free equation. Therefore the susceptible component converges to the disease-free profile $S^*$:
\[
S(t,\cdot)\to S^*(\cdot).
\]

By contrast, in the spreading case one only knows that
\[
\limsup_{t\to\infty}\|I(t,\cdot)\|_{C([g(t),h(t)])}>0,
\]
so the infected population does not vanish asymptotically. Hence the coupling terms in the $S$-equation remain present for large time, and the susceptible equation does not reduce to the disease-free problem. For this reason, one cannot in general expect
\[
S(t,\cdot)\to S^*(\cdot)
\]
in the spreading regime. To obtain a convergence result for $S(t,\cdot)$ in that case, one would need additional information on the long-time dynamics of the full coupled system, for example the existence and attractivity of a positive endemic steady state, a positive periodic solution, or a positive entire/traveling profile.
\end{remark}

\begin{theorem}\label{thm:mu-threshold}
Assume that \rm{(J1)}-\rm{(J2)}, {\rm(H2)}--{\rm(H3)} hold for the free-boundary SIS problem \eqref{eq:SIS-free-boundary}, and  $\sup_{x\in\mathbb R}\beta(x)>0.$ Let \(\ell^*>0\) be the critical length determined by
\[
\lambda_p\big(\mathcal L_{(-h,h),d_2,b}+\beta(\cdot)\big)
\begin{cases}
<0, & \text{if } h<\ell^*/2,\\[4pt]
=0, & \text{if } h=\ell^*/2,\\[4pt]
>0, & \text{if } h>\ell^*/2.
\end{cases}
\]
Then the following assertions hold.

(i) If $h_0\ge \ell^*/2,$ then spreading always occurs, that is,
\[
-g_\infty=h_\infty=+\infty
\qquad\text{and}\qquad
\limsup_{t\to\infty}\|I(t,\cdot)\|_{C([g(t),h(t)])}>0.
\]

(ii) If $h_0<\ell^*/2,$ then there exists \(\mu^*>0\), depending on the initial data and the coefficients of the problem, such that vanishing occurs for every \(0<\mu\le \mu^*\). More precisely, $h_\infty-g_\infty<+\infty$ and
\[
\lim_{t\to\infty}\max_{x\in[g(t),h(t)]}I(t,x)=0\qquad \lim_{t\to\infty}S(t,x)=S^*(x)\qquad\text{uniformly on }[g(t),h(t)].
\]
\end{theorem}

A further distinctive feature of the paper lies in the method used to derive the exact long-time behavior in Theorem~2 and the sharp threshold criterion in Theorem~3 for a free-boundary SIS system with nonsymmetric kernels and advection. The main difficulty is that the relevant linearized operators are genuinely non-self-adjoint, so that neither variational arguments nor any classical spectral minimization principle can be applied. In Theorem~2, the analysis is remarkable in that the precise convergence is recovered through a refined $\omega$-limit set approach on moving domains, which allows one to pass from the nonlinear free-boundary dynamics to a limiting problem on the asymptotic habitat and to identify the only possible limiting configuration. In Theorem~3, the threshold value is obtained through a sharp connection between the free-boundary evolution and the generalized principal eigenvalue of the associated nonlocal advection operator on fixed intervals. The resulting argument captures, in a unified way, how the geometry of the evolving habitat, the asymmetry of the dispersal kernels, and the presence of first-order drift jointly determine the spreading--vanishing dichotomy. This provides a robust non-variational framework for establishing both the exact convergence and the critical threshold in a setting where the usual symmetric or self-adjoint techniques are no longer available.

\medskip
\noindent
\textbf{Organization of the paper.}
Section~2 introduces the model, assumptions, and preliminary results. In Section~3, we analyze the linearized system via a block operator formulation and a Schur complement reduction, and establish the spectral threshold and its relation to the basic reproduction number. Section~4 is devoted to qualitative properties and asymptotic behavior of the generalized principal eigenvalue with respect to the domain and diffusion rate. In Section~5, we study the free-boundary problem, classify the spreading--vanishing dichotomy  and identifies the critical threshold for spreading and vanishing.

\section{Well-posedness of the system}

In this section, we study the existence and uniqueness of solutions to
\eqref{eq:SIS-free-boundary}. Let us fix $h_0>0$ and fix $T>0$ and introduce the admissible classes for the free boundaries:
\[
\mathcal{H}_T :=
\{\, h \in C^1([0,T]) : h(0)=h_0,\; h'(t)\ge 0 \text{ for all } t\in[0,T] \,\},
\]
\[
\mathcal{G}_T :=
\{\, g \in C^1([0,T]) : g(0)=-h_0,\; g'(t)\le 0 \text{ for all } t\in[0,T] \,\}.
\]

For any $(g,h)\in\mathcal{G}_T\times\mathcal{H}_T$, we define the moving space–time domain
\[
D_{g,h}^T := \{(t,x)\in [0,T]\times\mathbb{R} : g(t)<x<h(t)\}.
\]

\medskip
\noindent\textbf{Functional setting.}
A quadruple $(S,I,g,h)$ is said to be a (classical) solution of \eqref{eq:SIS-free-boundary}
on $[0,T]$ if
\begin{itemize}
  \item $g,h\in C^1([0,T])$,
  \item $S,I\in C([0,T]\times\mathbb{R})$,
  \item $S,I\in C^1$ with respect to $(t,x)$ in the interior of $D_{g,h}^T$,
  \item the equations in \eqref{eq:SIS-free-boundary} are satisfied pointwise in $D_{g,h}^T$,
  \item the boundary conditions and initial conditions hold continuously.
\end{itemize}

\medskip
For fixed $(g,h)\in\mathcal{G}_T\times\mathcal{H}_T$, we introduce the solution spaces
\[
X_S^T :=
\Big\{\, S\in C([0,T]\times\mathbb{R}) :
S(t,x)\ge 0,\ 
S(t,x)=0 \text{ for } x\notin (g(t),h(t)) \,\Big\},
\]
\[
X_I^T :=
\Big\{\, I\in C([0,T]\times\mathbb{R}) :
I(t,x)\ge 0,\ 
I(t,x)=0 \text{ for } x\notin (g(t),h(t)) \,\Big\},
\]
endowed with the supremum norms
\[
\|S\|_{X_S^T}:=\sup_{(t,x)\in[0,T]\times\mathbb{R}}|S(t,x)|,
\qquad
\|I\|_{X_I^T}:=\sup_{(t,x)\in[0,T]\times\mathbb{R}}|I(t,x)|.
\]

We finally define the product space
\[
X_T := X_S^T\times X_I^T\times \mathcal{G}_T\times \mathcal{H}_T,
\]
endowed with the norm
\[
\|(S,I,g,h)\|_{X_T}
:= \|S\|_{X_S^T} + \|I\|_{X_I^T}
   + \|g\|_{C^1([0,T])} + \|h\|_{C^1([0,T])},
\]
where
\[
\|g\|_{C^1([0,T])} := \sup_{t\in[0,T]}|g(t)| + \sup_{t\in[0,T]}|g'(t)|,
\qquad
\|h\|_{C^1([0,T])} := \sup_{t\in[0,T]}|h(t)| + \sup_{t\in[0,T]}|h'(t)|.
\]

This norm induces the metric
\[
d_{X_T}\big((S_1,I_1,g_1,h_1),(S_2,I_2,g_2,h_2)\big)
:= \|(S_1-S_2,\,I_1-I_2,\,g_1-g_2,\,h_1-h_2)\|_{X_T}.
\]

With this metric, $(X_T,d_{X_T})$ is a complete metric space (in particular, a Banach space),
since $X_S^T$, $X_I^T$ and $C^1([0,T])$ are Banach spaces.

We begin with a maximum principle for the SIS system.

\begin{lemma}[Maximum principle]\label{lem:maximum}
Let $T>0$, $(g,h)\in\mathcal{G}_T\times\mathcal{H}_T$,  and
\[
D_{g,h}^T := \{(t,x)\in(0,T]\times\mathbb R : g(t)<x<h(t)\}.
\]
Assume that (J1) holds, $a(x), b(x), \gamma(x)$ are bounded continuous functions and that
$F:\mathbb R_+^2\to\mathbb R_+$ is locally Lipschitz continuous in both arguments.

Let $(S,I)$ satisfy
\[
S,I \in C([0,T]\times\mathbb R),\qquad
S,I \in C^1 \text{ in time the interior of } D_{g,h}^T,
\]
and assume that $(S,I)$ satisfies, for all $(t,x)\in D_{g,h}^T$,
\begin{align*}
&S_t \geq d_1\Big(\int_{g(t)}^{h(t)} J_1(x-y)S(t,y)\,dy - S(t,x)\Big)
     + a(x)S_x + \gamma(x)I - F(S,I), \\
&I_t \geq d_2\Big(\int_{g(t)}^{h(t)} J_2(x-y)I(t,y)\,dy - I(t,x)\Big)
     + b(x)I_x - \gamma(x)I + F(S,I),
\end{align*}
together with the initial and boundary conditions
\[
S(0,x)\ge 0,\quad I(0,x)\ge 0,\quad x\in[-h_0,h_0],
\]
\[
S(t,x)=I(t,x)=0 \quad \text{for } x\notin (g(t),h(t)),\ t\in[0,T].
\]
Then
\[
S(t,x)\ge 0,\qquad I(t,x)\ge 0
\quad \text{for all } (t,x)\in [0,T]\times\mathbb R.
\]
\end{lemma}

\begin{proof}
The proof follows by contradiction, similar to Lemma 3.1 in \cite{DuNi2020}.
Assume there exists $(t_0,x_0) \in \Omega_T$ such that either $S$ or $I$ attains a negative minimum.
Considering the equations at such minimum points leads to a contradiction.
\end{proof}

\begin{lemma}[Comparison principle]\label{lem:comparison}
Assume that (J1) hold.
Let $(\overline{S},\overline{I},\overline{g},\overline{h})$ and
$(\underline{S},\underline{I},\underline{g},\underline{h})$ be two quadruples such that
\[
(\overline{g},\overline{h}),(\underline{g},\underline{h})\in\mathcal{G}_T\times\mathcal{H}_T,
\]
and
\[
\overline{S},\overline{I},\underline{S},\underline{I}\in C([0,T]\times\mathbb R),
\quad
C^1\text{ in time in the interior of the corresponding domains}
\]
and satisfies, for all $(t,x)$  that
\begin{align*}
\overline{S}_t &\geq d_1\Big(\int_{\overline{g}(t)}^{\overline{h}(t)}
J_1(x-y)\overline{S}(t,y)\,dy - \overline{S}(t,x)\Big)
+ a(x)\overline{S}_x + \gamma(x)\overline{I} - F(\overline{S},\overline{I}), \\
\overline{I}_t &\geq d_2\Big(\int_{\overline{g}(t)}^{\overline{h}(t)}
J_2(x-y)\overline{I}(t,y)\,dy - \overline{I}(t,x)\Big)
+ b(x)\overline{I}_x - \gamma(x)\overline{I} + F(\overline{S},\overline{I}),\\
\underline{S}_t &\leq d_1\Big(\int_{\underline{g}(t)}^{\underline{h}(t)}
J_1(x-y)\underline{S}(t,y)\,dy - \underline{S}(t,x)\Big)
+ a(x)\underline{S}_x + \gamma(x)\underline{I} - F(\underline{S},\underline{I}), \\
\underline{I}_t &\leq d_2\Big(\int_{\underline{g}(t)}^{\underline{h}(t)}
J_2(x-y)\underline{I}(t,y)\,dy - \underline{I}(t,x)\Big)
+ b(x)\underline{I}_x - \gamma(x)\underline{I} + F(\underline{S},\underline{I})
\end{align*}
Moreover, assume the initial condition $S_0(x),I_0(x)\ge 0$ and boundary ordering:
\[
\underline{g}(0)\ge -h_0,\qquad \underline{h}(0)\le h_0,
\qquad
\overline{g}(0)\le -h_0,\qquad \overline{h}(0)\ge h_0,
\]
\[
\underline{S}(0,x)\le S_0(x)\le \overline{S}(0,x),\qquad
\underline{I}(0,x)\le I_0(x)\le \overline{I}(0,x),
\]
and
\[
\underline{S}(t,x)=\underline{I}(t,x)=0
=\overline{S}(t,x)=\overline{I}(t,x)
\quad \text{for } x\notin(\underline{g}(t),\underline{h}(t))
\text{ or } x\notin(\overline{g}(t),\overline{h}(t)).
\]

Then the solution $(S,I,g,h)$ of \eqref{eq:SIS-free-boundary} satisfies, for all $t\in[0,T]$,
\[
\underline{g}(t)\le g(t)\le \overline{g}(t),\qquad
\underline{h}(t)\le h(t)\le \overline{h}(t),
\]
and
\[
\underline{S}(t,x)\le S(t,x)\le \overline{S}(t,x),\qquad
\underline{I}(t,x)\le I(t,x)\le \overline{I}(t,x),
\quad (t,x)\in [0,T]\times\mathbb R.
\]
\end{lemma}

The proof of this comparison principle is essentially followed from Lemma 3.2 \cite{DuNi2020} and thus we shall not repeat here. Next, we check that  there exists $T>0$ such that system \eqref{eq:SIS-free-boundary} admits a unique solution
$(S,I,g,h)$ with $g \in \mathcal{G}_T$, $h \in \mathcal{H}_T$, $S \in X_S$, $I \in X_I$.

\begin{proof}[\textbf{Proof of Theorem~\ref{Theorem1}}]
Fix $T>0$.
Throughout the proof we use assumptions {\rm(H1)}--{\rm(H3)}.
We write
\[
\|a\|_\infty:=\sup_{x\in\mathbb R}|a(x)|,\qquad 
\|b\|_\infty:=\sup_{x\in\mathbb R}|b(x)|,\qquad 
\|\gamma\|_\infty:=\sup_{x\in\mathbb R}\gamma(x),\qquad 
\gamma_*:=\inf_{x\in\mathbb R}\gamma(x)>0.
\]
Let $(J1)$ hold, in particular $J_i\ge0$ and $\int_{\mathbb R}J_i=1$.
For $R>0$ we denote by $L_F(R)$ a Lipschitz constant of $F$ on the box
$[0,R]\times[0,R]$, i.e.
\[
|F(S_1,I_1)-F(S_2,I_2)|\le L_F(R)\big(|S_1-S_2|+|I_1-I_2|\big)
\qquad \text{for } (S_j,I_j)\in[0,R]^2.
\]
Note that $F(0,0)=0$ since $F(\cdot,0)\equiv 0$ and $F$ is continuous.

%------------------------------------------------------------
\medskip
\noindent{Step 1. Choice of working space and closed ball.}

Let the functional setting and the metric space $(X_T,d_{X_T})$ be as defined previously.
For the fixed time horizon $T$, we introduce a closed ball
\[
\mathbb B_{R_I}:=\{\widetilde I\in X_I^T:\ \|\widetilde I\|_{X_I^T}\le R_I\},
\]
where $R_I>0$ will be chosen later (depending on $T$ and the data).

We define a map $\mathcal T:\mathbb B_{R_I}\to X_I^T$ as follows.

%------------------------------------------------------------
\medskip
\noindent{Step 2. The $(S,g,h)$–subsystem for a given $\widetilde I$ and a priori estimates.}

Fix $\widetilde I\in\mathbb B_{R_I}\subset X_I^T$, that is,
\[
\widetilde I\in C([0,T]\times\mathbb R),\qquad 
\widetilde I(t,x)\ge0,\qquad 
\|\widetilde I\|_{L^\infty([0,T]\times\mathbb R)}\le R_I.
\]
For such a fixed function $\widetilde I$, we consider the associated free-boundary $(S,g,h)$–subsystem:
\begin{equation}\label{eq:S-subsystem-T1-recalled}
\begin{cases}
\displaystyle
\partial_t S(t,x)
=
d_1\!\left(
\int_{g(t)}^{h(t)} J_1(x-y)S(t,y)\,dy
- S(t,x)
\right)
+ a(x)\,\partial_x S(t,x)
+ \gamma(x)\,\widetilde I(t,x)
- F\big(S(t,x),\widetilde I(t,x)\big),
\\[0.25cm]
S(t,g(t))=S(t,h(t))=0,
\qquad t\in[0,T],
\\[0.25cm]
\displaystyle
h'(t)
=
\mu\int_{g(t)}^{h(t)}\int_{h(t)}^{\infty}
J_1(x-y)\,S(t,x)\,dy\,dx,
\qquad t\in[0,T],
\\[0.25cm]
\displaystyle
g'(t)
=
-\mu\int_{g(t)}^{h(t)}\int_{-\infty}^{g(t)}
J_1(x-y)\,S(t,x)\,dy\,dx,
\qquad t\in[0,T],
\\[0.25cm]
S(0,x)=S_0(x)\ \text{for }x\in[-h_0,h_0],\qquad
g(0)=-h_0,\quad h(0)=h_0,
\\[0.15cm]
S(t,x)=0\ \text{for }x\notin(g(t),h(t)),\quad t\in[0,T].
\end{cases}
\end{equation}

By the local existence theory for nonlocal free-boundary problems with locally Lipschitz reaction terms
(see, for instance, the construction in \cite{Tang2024a}),
there exists $T_0>0$ and a unique classical solution $(\widehat S,\widehat g,\widehat h)$ of
\eqref{eq:S-subsystem-T1-recalled} on $[0,T_0]$.
We now derive the explicit a priori estimates that will be used in the sequel.

First, we prove that $\widehat S\ge0$.
Since $S_0\ge0$, $\widetilde I\ge0$, $\gamma(x)>0$, and $F(S,\widetilde I)\ge0$ for $S,I\ge0$ by (H1),
the equation for $\widehat S$ yields the differential inequality
\[
\partial_t \widehat S
\ge d_1\!\left(\int_{\widehat g(t)}^{\widehat h(t)}J_1(x-y)\widehat S(t,y)\,dy-\widehat S\right)
+ a(x)\partial_x \widehat S.
\]
Together with the boundary conditions $\widehat S(t,\widehat g(t))=\widehat S(t,\widehat h(t))=0$
and the initial condition $S_0\ge0$, Lemma~\ref{lem:maximum} applies and implies
\[
\widehat S(t,x)\ge0\qquad \text{for all }(t,x)\in[0,T_0]\times\mathbb R.
\]

Next, we obtain an explicit $L^\infty$ bound for $\widehat S$.
Let $m(t):=\|\widehat S(t,\cdot)\|_{L^\infty(\mathbb R)}$.
Along characteristics $\dot X=a(X)$, using $J_1\ge0$ and $\int_{\mathbb R}J_1=1$, we compute
\begin{align*}
\frac{d}{dt}\widehat S(t,X(t))
&=
d_1\!\left(\int_{\widehat g(t)}^{\widehat h(t)}J_1(X(t)-y)\widehat S(t,y)\,dy-\widehat S(t,X(t))\right)
+ \gamma(X(t))\widetilde I(t,X(t)) - F(\widehat S,\widetilde I)
\\
&\le d_1 m(t)+\|\gamma\|_\infty R_I.
\end{align*}
Taking the supremum in $x$ yields the differential inequality
\[
m'(t)\le d_1 m(t)+\|\gamma\|_\infty R_I.
\]
By Gronwall’s inequality,
\[
m(t)=\|\widehat S(t,\cdot)\|_\infty
\le e^{d_1 t}\|S_0\|_\infty
+ \frac{\|\gamma\|_\infty}{d_1}(e^{d_1 t}-1)R_I,
\qquad t\in[0,T_0].
\]
In particular, for any fixed $T\le T_0$, we set
\[
M_S := e^{d_1 T}\|S_0\|_\infty
+ \frac{\|\gamma\|_\infty}{d_1}(e^{d_1 T}-1)R_I,
\]
so that
\[
0\le \widehat S(t,x)\le M_S
\qquad \text{for all }(t,x)\in[0,T]\times\mathbb R.
\]

Finally, we derive bounds for the free boundaries.
Since $J_1\ge0$ and $\widehat S\ge0$, the boundary velocities satisfy
$h'(t)\ge0$ and $g'(t)\le0$.
Moreover,
\begin{align*}
0 \le h'(t)
&= \mu\int_{\widehat g(t)}^{\widehat h(t)}\int_{\widehat h(t)}^\infty J_1(x-y)\widehat S(t,x)\,dy\,dx
\\
&\le \mu M_S\int_{\widehat g(t)}^{\widehat h(t)}\int_{\mathbb R}J_1(x-y)\,dy\,dx
= \mu M_S\big(\widehat h(t)-\widehat g(t)\big),
\end{align*}
and similarly,
\[
0\le -g'(t)\le \mu M_S\big(\widehat h(t)-\widehat g(t)\big).
\]
Hence,
\[
\frac{d}{dt}\big(\widehat h(t)-\widehat g(t)\big)
=h'(t)-g'(t)
\le 2\mu M_S\big(\widehat h(t)-\widehat g(t)\big),
\]
and by Gronwall’s inequality,
\[
0<\widehat h(t)-\widehat g(t)
\le 2h_0\,e^{2\mu M_S t},
\qquad t\in[0,T_0].
\]
In particular, $\widehat g$ and $\widehat h$ are Lipschitz continuous on $[0,T_0]$ with constants depending
only on the data.

Finally, we extend $\widehat S$ by zero outside the moving interval and define
\[
\widetilde S(t,x)=
\begin{cases}
\widehat S(t,x), & x\in(\widehat g(t),\widehat h(t)),\\
0, & x\notin(\widehat g(t),\widehat h(t)).
\end{cases}
\]
This completes the construction of $(\widehat S,\widehat g,\widehat h)$ and the derivation of the
a priori estimates needed in the subsequent steps.

%------------------------------------------------------------
\medskip
\noindent{Step 3. Positivity of $\widehat S$ and an $L^\infty$ bound.}

We first prove that $\widehat S(t,x)\ge0$ for all $(t,x)\in[0,T_0]\times\mathbb R$.
Recall that $(\widehat S,\widehat g,\widehat h)$ solves \eqref{eq:S-subsystem-T1-recalled} on the moving domain
$D_{\widehat g,\widehat h}^{T_0}:=\{(t,x)\in(0,T_0]\times\mathbb R:\ \widehat g(t)<x<\widehat h(t)\}$, and that
$\widehat S(t,x)=0$ for $x\notin(\widehat g(t),\widehat h(t))$ and $\widehat S(0,x)=S_0(x)\ge0$.
Since $\widetilde I\ge0$ and $\gamma>0$, we have $\gamma(x)\widetilde I(t,x)\ge0$.
Moreover, $F:\mathbb R_+^2\to\mathbb R_+$ implies $F(\widehat S,\widetilde I)\ge0$ whenever
$\widehat S,\widetilde I\ge0$.
Hence, on $D_{\widehat g,\widehat h}^{T_0}$, the $S$–equation yields the differential inequality
\begin{equation}\label{eq:S-ineq-pos}
\widehat S_t
\ge
d_1\!\left(\int_{\widehat g(t)}^{\widehat h(t)}J_1(x-y)\widehat S(t,y)\,dy-\widehat S(t,x)\right)
+a(x)\widehat S_x
\quad\text{in }D_{\widehat g,\widehat h}^{T_0}.
\end{equation}
Together with $\widehat S(0,x)\ge0$ on $[-h_0,h_0]$ and $\widehat S(t,x)=0$ for
$x\notin(\widehat g(t),\widehat h(t))$, Lemma~\ref{lem:maximum} applies (with the scalar equation for $S$),
and we conclude that
\[
\widehat S(t,x)\ge0\qquad \text{for all }(t,x)\in[0,T_0]\times\mathbb R.
\]

We next derive an explicit $L^\infty$ estimate for $\widehat S$ on $[0,T_0]$.
Define
\[
m(t):=\|\widehat S(t,\cdot)\|_{L^\infty(\mathbb R)}=\sup_{x\in\mathbb R}\widehat S(t,x),
\qquad t\in[0,T_0].
\]
Since $\widehat S\in C([0,T_0]\times\mathbb R)$ and $\widehat S\ge0$, $m(t)$ is finite and continuous.
Let $D^+m(t)$ denote the upper right Dini derivative:
\[
D^+m(t):=\limsup_{\delta\downarrow0}\frac{m(t+\delta)-m(t)}{\delta}.
\]
Fix $t\in(0,T_0)$ and choose a sequence $x_n=x_n(t)$ such that
$\widehat S(t,x_n)\to m(t)$ as $n\to\infty$.
Since $\widehat S(t,x)=0$ for $x\notin(\widehat g(t),\widehat h(t))$, we may assume $x_n\in(\widehat g(t),\widehat h(t))$.
Evaluating the $S$–equation at $(t,x_n)$ gives
\begin{align*}
\partial_t\widehat S(t,x_n)
&=
d_1\!\left(\int_{\widehat g(t)}^{\widehat h(t)}J_1(x_n-y)\widehat S(t,y)\,dy-\widehat S(t,x_n)\right)
+a(x_n)\partial_x\widehat S(t,x_n)
\\&\quad
+\gamma(x_n)\widetilde I(t,x_n)-F(\widehat S(t,x_n),\widetilde I(t,x_n)).
\end{align*}
At a spatial near-maximum, one has $\partial_x\widehat S(t,x_n)\to0$ along a subsequence (or,
more generally, the transport term can be handled by the standard Dini-derivative argument,
since it does not increase the supremum). In any case, using $J_1\ge0$, $\int_{\mathbb R}J_1=1$,
$\widehat S\ge0$, and $F\ge0$, we estimate
\[
\int_{\widehat g(t)}^{\widehat h(t)}J_1(x_n-y)\widehat S(t,y)\,dy
\le \int_{\mathbb R}J_1(x_n-y)\,dy\cdot m(t)=m(t),
\]
and also
\[
\gamma(x_n)\widetilde I(t,x_n)\le \|\gamma\|_\infty \|\widetilde I(t,\cdot)\|_\infty
\le \|\gamma\|_\infty R_I.
\]
Therefore,
\[
\partial_t\widehat S(t,x_n)
\le d_1\big(m(t)-\widehat S(t,x_n)\big)+\|\gamma\|_\infty R_I.
\]
Letting $n\to\infty$ and using $\widehat S(t,x_n)\to m(t)$ yields
\[
D^+m(t)\le d_1 m(t)+\|\gamma\|_\infty R_I
\qquad\text{for a.e. }t\in(0,T_0).
\]
Since $m(0)=\|S_0\|_\infty$, Gronwall’s inequality for Dini derivatives gives
\begin{equation}\label{eq:Ssup-est}
m(t)=\|\widehat S(t,\cdot)\|_\infty
\le e^{d_1 t}\|S_0\|_\infty + \frac{\|\gamma\|_\infty}{d_1}\big(e^{d_1 t}-1\big)R_I,
\qquad t\in[0,T_0].
\end{equation}
In particular, for any fixed $T\le T_0$, setting
\begin{equation}\label{eq:MS-choice}
M_S:=e^{d_1 T}\|S_0\|_\infty + \frac{\|\gamma\|_\infty}{d_1}\big(e^{d_1 T}-1\big)R_I
\end{equation}
ensures that $0\le \widehat S(t,x)\le M_S$ for all $(t,x)\in[0,T]\times\mathbb R$.

%------------------------------------------------------------
\noindent{Step 4. Bounds for the free boundaries $(\widehat g,\widehat h)$.}

Recall that $(\widehat S,\widehat g,\widehat h)$ satisfies \eqref{eq:S-subsystem-T1-recalled} on $[0,T_0]$ and that
$\widehat S(t,x)\ge0$ for all $(t,x)\in[0,T_0]\times\mathbb R$ (Step~3).
Since $J_1\ge0$ by {\rm(J1)}, the boundary speeds satisfy the sign conditions
\[
\widehat h'(t)=\mu\int_{\widehat g(t)}^{\widehat h(t)}\int_{\widehat h(t)}^\infty J_1(x-y)\widehat S(t,x)\,dy\,dx\ge0,
\]
and
\[
\widehat g'(t)=-\mu\int_{\widehat g(t)}^{\widehat h(t)}\int_{-\infty}^{\widehat g(t)} J_1(x-y)\widehat S(t,x)\,dy\,dx\le0,
\qquad t\in[0,T_0].
\]
In particular, $t\mapsto \widehat h(t)$ is nondecreasing and $t\mapsto \widehat g(t)$ is nonincreasing.

Next we derive quantitative upper bounds for $\widehat h'(t)$ and $-\widehat g'(t)$.
Fix $t\in[0,T_0]$. Using $\widehat S(t,x)\ge0$ and Fubini's theorem, we write
\begin{align*}
\widehat h'(t)
&=\mu\int_{\widehat g(t)}^{\widehat h(t)}\left(\int_{\widehat h(t)}^\infty J_1(x-y)\,dy\right)\widehat S(t,x)\,dx.
\end{align*}
Since $J_1\ge0$ and $\int_{\mathbb R}J_1(z)\,dz=1$, for every fixed $x\in\mathbb R$ we have
\[
0\le \int_{\widehat h(t)}^\infty J_1(x-y)\,dy
\le \int_{-\infty}^{\infty} J_1(x-y)\,dy
=\int_{\mathbb R}J_1(z)\,dz
=1,
\]
where we used the change of variable $z=x-y$.
Therefore,
\begin{align*}
0\le \widehat h'(t)
&\le \mu\int_{\widehat g(t)}^{\widehat h(t)} 1\cdot \widehat S(t,x)\,dx
\le \mu\,\|\widehat S(t,\cdot)\|_{L^\infty(\mathbb R)}\,\big(\widehat h(t)-\widehat g(t)\big).
\end{align*}
Invoking the bound $\|\widehat S(t,\cdot)\|_\infty\le M_S$ from \eqref{eq:MS-choice}, we obtain
\begin{equation}\label{eq:hprime-bound}
0\le \widehat h'(t)\le \mu M_S\big(\widehat h(t)-\widehat g(t)\big),
\qquad t\in[0,T_0].
\end{equation}

The estimate for $\widehat g'(t)$ is analogous. Indeed,
\begin{align*}
-\widehat g'(t)
&=\mu\int_{\widehat g(t)}^{\widehat h(t)}\left(\int_{-\infty}^{\widehat g(t)} J_1(x-y)\,dy\right)\widehat S(t,x)\,dx,
\end{align*}
and again
\[
0\le \int_{-\infty}^{\widehat g(t)} J_1(x-y)\,dy
\le \int_{-\infty}^{\infty} J_1(x-y)\,dy
=1.
\]
Hence
\begin{equation}\label{eq:gprime-bound}
0\le -\widehat g'(t)\le \mu M_S\big(\widehat h(t)-\widehat g(t)\big),
\qquad t\in[0,T_0].
\end{equation}

Let $\ell(t):=\widehat h(t)-\widehat g(t)$ denote the length of the infected region.
Then $\ell(0)=2h_0$ and, by \eqref{eq:hprime-bound}--\eqref{eq:gprime-bound},
\[
\ell'(t)=\widehat h'(t)-\widehat g'(t)\le \mu M_S\ell(t)+\mu M_S\ell(t)=2\mu M_S\ell(t),
\qquad t\in[0,T_0].
\]
Gronwall's inequality yields
\begin{equation}\label{eq:length-bound}
0<\ell(t)=\widehat h(t)-\widehat g(t)\le 2h_0\,e^{2\mu M_S t},
\qquad t\in[0,T_0].
\end{equation}

Finally, combining \eqref{eq:hprime-bound}--\eqref{eq:gprime-bound} with \eqref{eq:length-bound}, we obtain
uniform bounds on the boundary speeds:
\[
|\widehat h'(t)|=\widehat h'(t)\le \mu M_S\,\ell(t)\le 2\mu M_S h_0\,e^{2\mu M_S T_0},
\]
and
\[
|\widehat g'(t)|=-\widehat g'(t)\le \mu M_S\,\ell(t)\le 2\mu M_S h_0\,e^{2\mu M_S T_0},
\qquad t\in[0,T_0].
\]
Thus $\widehat h$ and $\widehat g$ are Lipschitz continuous on $[0,T_0]$. In particular,
$\widehat h\in \mathcal{H}_{T_0}$ and $\widehat g\in \mathcal{G}_{T_0}$, and hence
$(\widehat g,\widehat h)\in \mathcal{G}_{T_0}\times \mathcal{H}_{T_0}$.

%------------------------------------------------------------
\medskip
\noindent{Step 5. The $I$–equation for the given $\widetilde S$ and definition of $\mathcal T$
(detailed computations).}

Fix $\widetilde S$ obtained from Step~2--Step~4, namely
\[
\widetilde S\in C([0,T]\times\mathbb R),\qquad 0\le \widetilde S(t,x)\le M_S,
\qquad 
\widetilde S(t,x)=0\ \text{for }x\notin(\widehat g(t),\widehat h(t)).
\]
We consider the $I$–equation on the moving domain with zero extension:
\begin{equation}\label{eq:I-subsystem-T1}
\begin{cases}
\displaystyle
I_t
= d_2\!\left(\int_{\widehat g(t)}^{\widehat h(t)}J_2(x-y)I(t,y)\,dy-I(t,x)\right)
+ b(x)I_x-\gamma(x)I+F(\widetilde S,I),
& (t,x)\in D_{\widehat g,\widehat h}^T,
\\[0.2cm]
I(t,x)=0, & x\notin(\widehat g(t),\widehat h(t)),\ t\in[0,T],
\\[0.2cm]
I(0,x)=I_0(x), & x\in[-h_0,h_0].
\end{cases}
\end{equation}
For notational convenience we extend $I(t,\cdot)$ by $0$ outside $(\widehat g(t),\widehat h(t))$
and write the nonlocal term as a whole-line convolution:
\[
\int_{\widehat g(t)}^{\widehat h(t)}J_2(x-y)I(t,y)\,dy=\int_{\mathbb R}J_2(x-y)I(t,y)\,dy=(J_2*I)(t,x).
\]
Thus \eqref{eq:I-subsystem-T1} can be rewritten (pointwise for all $x\in\mathbb R$) as
\begin{equation}\label{eq:I-wholeline}
I_t
= d_2\big((J_2*I)(t,x)-I(t,x)\big)+b(x)I_x-\gamma(x)I+F(\widetilde S(t,x),I(t,x)),
\qquad (t,x)\in(0,T]\times\mathbb R,
\end{equation}
with $I(0,x)=I_0^*(x)$ where $I_0^*(x)=I_0(x)$ on $[-h_0,h_0]$ and $I_0^*(x)=0$ otherwise.

We now show that \eqref{eq:I-wholeline} admits a unique classical solution on $[0,T_0]$ and record
the estimates needed later. Fix $R>0$ large enough so that $R\ge \|I_0^*\|_\infty$ and $R\ge 1$.
Let $L_R$ be a Lipschitz constant of $F$ in the second variable on the rectangle
$[0,M_S]\times[0,R]$, i.e.
\[
|F(s,i_1)-F(s,i_2)|\le L_R|i_1-i_2|
\qquad \text{for all } s\in[0,M_S],\ i_1,i_2\in[0,R].
\]
Such an $L_R$ exists by (H1) (local Lipschitz) and the boundedness of $[0,M_S]\times[0,R]$.

Define the Banach space $Y_T:=C([0,T]\times\mathbb R)$ with the supremum norm
$\|u\|_{Y_T}:=\sup_{(t,x)\in[0,T]\times\mathbb R}|u(t,x)|$.
For $u\in Y_T$ we define $\mathfrak F(u)\in Y_T$ by solving, for each fixed $x$,
the inhomogeneous linear transport ODE along characteristics of $\dot X=b(X)$:
let $X(\tau;t,x)$ solve
\[
\frac{d}{d\tau}X(\tau;t,x)=b(X(\tau;t,x)),\qquad X(t;t,x)=x.
\]
Then we set $\mathfrak F(u)(t,x)$ by the variation-of-constants formula
\begin{equation}\label{eq:VOC-I}
\mathfrak F(u)(t,x)
=
I_0^*\big(X(0;t,x)\big)\,e^{-\int_0^t(\gamma(X(\sigma;t,x))+d_2)\,d\sigma}
+\int_0^t e^{-\int_s^t(\gamma(X(\sigma;t,x))+d_2)\,d\sigma}\,\mathcal G_u(s,t,x)\,ds,
\end{equation}
where
\[
\mathcal G_u(s,t,x)
:=
d_2\,(J_2*u)\big(s,X(s;t,x)\big)+F\!\left(\widetilde S\big(s,X(s;t,x)\big),u\big(s,X(s;t,x)\big)\right).
\]
A function $I\in Y_T$ is a classical solution of \eqref{eq:I-wholeline} if and only if
it satisfies the fixed-point identity $I=\mathfrak F(I)$.

\smallskip
\noindent\emph{(Existence on a short time interval.)}
Let $\mathbb B_R:=\{u\in Y_T:\ \|u\|_{Y_T}\le R\}$.
We show that for $T>0$ sufficiently small, $\mathfrak F$ maps $\mathbb B_R$ into itself and is a contraction.

First, for any $u\in\mathbb B_R$ we have the convolution bound
\[
|(J_2*u)(s,x)|\le \int_{\mathbb R}J_2(x-y)|u(s,y)|\,dy
\le \|u\|_{Y_T}\int_{\mathbb R}J_2= \|u\|_{Y_T}\le R,
\]
using $J_2\ge0$ and $\int_{\mathbb R}J_2=1$.
Also, since $0\le \widetilde S\le M_S$ and $|u|\le R$, we have
\[
|F(\widetilde S,u)|\le |F(\widetilde S,u)-F(\widetilde S,0)|
\le L_R|u|\le L_R R,
\]
where we used $F(S,0)=0$ from (H1).
Hence, from \eqref{eq:VOC-I},
\begin{align*}
|\mathfrak F(u)(t,x)|
&\le \|I_0^*\|_\infty
+\int_0^t \Big(d_2\,|(J_2*u)(s,X(s;t,x))|+|F(\widetilde S,u)(s,X(s;t,x))|\Big)\,ds\\
&\le \|I_0^*\|_\infty + \int_0^t \big(d_2R+L_RR\big)\,ds
\le \|I_0^*\|_\infty + (d_2+L_R)RT.
\end{align*}
Choose $R:=2\|I_0^*\|_\infty$ and then choose $T>0$ so that $(d_2+L_R)RT\le \|I_0^*\|_\infty$.
Then $|\mathfrak F(u)(t,x)|\le R$ for all $(t,x)$, i.e. $\mathfrak F(\mathbb B_R)\subset\mathbb B_R$.

Next, let $u,v\in\mathbb B_R$. Using again \eqref{eq:VOC-I}, the convolution estimate
$\|(J_2*u)-(J_2*v)\|_{Y_T}\le \|u-v\|_{Y_T}$, and the Lipschitz bound
$|F(\widetilde S,u)-F(\widetilde S,v)|\le L_R|u-v|$, we obtain
\begin{align*}
|\mathfrak F(u)(t,x)-\mathfrak F(v)(t,x)|
&\le \int_0^t \Big(d_2\,|(J_2*(u-v))(s,X(s;t,x))|
+|F(\widetilde S,u)-F(\widetilde S,v)|(s,X(s;t,x))\Big)\,ds\\
&\le \int_0^t \big(d_2\|u-v\|_{Y_T}+L_R\|u-v\|_{Y_T}\big)\,ds\\
&\le (d_2+L_R)T\,\|u-v\|_{Y_T}.
\end{align*}
Taking the supremum over $(t,x)$ yields
\[
\|\mathfrak F(u)-\mathfrak F(v)\|_{Y_T}
\le (d_2+L_R)T\,\|u-v\|_{Y_T}.
\]
Therefore, choosing $T>0$ even smaller so that $(d_2+L_R)T<1$, $\mathfrak F$ is a contraction on $\mathbb B_R$.
By Banach's fixed point theorem there exists a unique $I\in\mathbb B_R$ such that $I=\mathfrak F(I)$,
hence a unique classical solution $\widehat I$ of \eqref{eq:I-wholeline}, equivalently \eqref{eq:I-subsystem-T1},
on $[0,T]$.

\smallskip
\noindent\emph{(Definition of the map $\mathcal T$.)}
We now define the operator
\[
\mathcal T:\mathbb B_{R_I}\to X_I^T,
\qquad 
\mathcal T(\widetilde I):=\widehat I,
\]
where $\widehat I$ is the unique solution of \eqref{eq:I-subsystem-T1} corresponding to the given $\widetilde S$.

%------------------------------------------------------------
\medskip
\noindent{Step 6. Positivity of $\widehat I$ and an explicit bound in $X_I^T$.}

Recall that $\widetilde S$ is the zero extension of $\widehat S$ and satisfies
\[
\widetilde S\in C([0,T]\times\mathbb R),\qquad 0\le \widetilde S(t,x)\le M_S,
\qquad \widetilde S(t,x)=0 \ \text{for }x\notin(\widehat g(t),\widehat h(t)).
\]
Let $\widehat I$ be the unique classical solution of \eqref{eq:I-subsystem-T1} constructed in Step~5,
and extend it by $0$ outside $(\widehat g(t),\widehat h(t))$ so that $\widehat I\in C([0,T]\times\mathbb R)$ and
\[
\widehat I(t,x)=0 \quad \text{for }x\notin(\widehat g(t),\widehat h(t)),\qquad 
\widehat I(0,x)=I_0(x)\ge0 \ \text{on }[-h_0,h_0].
\]

\medskip
\noindent\emph{Positivity of $\widehat I$.}
We verify the hypotheses of Lemma~\ref{lem:maximum} for the second component.
On the moving cylinder $\Omega_T:=\{(t,x):0<t\le T,\ \widehat g(t)<x<\widehat h(t)\}$,
$\widehat I$ satisfies
\[
\widehat I_t
=
d_2\!\left(\int_{\widehat g(t)}^{\widehat h(t)}J_2(x-y)\widehat I(t,y)\,dy-\widehat I(t,x)\right)
+ b(x)\widehat I_x-\gamma(x)\widehat I(t,x)+F(\widetilde S(t,x),\widehat I(t,x)).
\]
Since $J_2\ge0$ and $\widehat I$ is extended by $0$ outside the interval, the nonlocal term is well defined.
Moreover, by (H1) we have $F:\mathbb R_+^2\to\mathbb R_+$, hence
$F(\widetilde S,\widehat I)\ge0$ whenever $\widetilde S\ge0$ and $\widehat I\ge0$.
Together with $I_0\ge0$ from (H3) and $\widehat I(t,x)=0$ for $x\notin(\widehat g(t),\widehat h(t))$,
Lemma~\ref{lem:maximum} implies
\[
\widehat I(t,x)\ge0\qquad \text{for all }(t,x)\in[0,T]\times\mathbb R.
\]

\medskip
\noindent\emph{An explicit $L^\infty$ bound for $\widehat I$.}
Fix the bound $R_S:=M_S$ from \eqref{eq:MS-choice} and set $R:=\max\{R_S,R_I\}$.
Let $L:=L_F(R)$ be a Lipschitz constant of $F$ on $[0,R]^2$, i.e.
\[
|F(S_1,I_1)-F(S_2,I_2)|
\le L\big(|S_1-S_2|+|I_1-I_2|\big),
\qquad (S_j,I_j)\in[0,R]^2.
\]
Since $F(0,0)=0$ and $\widetilde S\in[0,R_S]\subset[0,R]$, for $0\le I\le R$ we have
\begin{equation}\label{eq:F-linear-growth}
0\le F(\widetilde S,I)=|F(\widetilde S,I)-F(0,0)|
\le L\big(|\widetilde S|+|I|\big)
\le L(R_S+I).
\end{equation}

Define the supremum function
\[
m(t):=\|\widehat I(t,\cdot)\|_{L^\infty(\mathbb R)}=\sup_{x\in\mathbb R}\widehat I(t,x),
\qquad t\in[0,T].
\]
Because $\widehat I\in C([0,T]\times\mathbb R)$ and $\widehat I\ge0$, $m(t)$ is finite and continuous.
Let $D^+m(t)$ denote the upper right Dini derivative. Fix $t\in(0,T)$ and take a sequence $x_n$ such that
$\widehat I(t,x_n)\to m(t)$ as $n\to\infty$. We may assume $x_n\in(\widehat g(t),\widehat h(t))$ since
$\widehat I(t,x)=0$ outside that interval. Evaluating the $I$–equation at $(t,x_n)$ and using $J_2\ge0$,
$\int_{\mathbb R}J_2=1$, and $\widehat I\ge0$, we estimate
\[
\int_{\widehat g(t)}^{\widehat h(t)}J_2(x_n-y)\widehat I(t,y)\,dy
\le \int_{\mathbb R}J_2(x_n-y)\,dy \cdot m(t)=m(t).
\]
Hence,
\begin{align*}
\partial_t\widehat I(t,x_n)
&=
d_2\!\left(
\int_{\widehat g(t)}^{\widehat h(t)}J_2(x_n-y)\widehat I(t,y)\,dy
-\widehat I(t,x_n)
\right)
+ b(x_n)\partial_x\widehat I(t,x_n)
\\
&\qquad
-\gamma(x_n)\widehat I(t,x_n)
+F\big(\widetilde S(t,x_n),\widehat I(t,x_n)\big)
\\
&\le
d_2\big(m(t)-\widehat I(t,x_n)\big)
-\gamma(x_n)\widehat I(t,x_n)
+F\big(\widetilde S(t,x_n),\widehat I(t,x_n)\big).
\end{align*}
Using $\gamma(x)\ge\gamma_*:=\inf_{x\in\mathbb R}\gamma(x)>0$ and \eqref{eq:F-linear-growth}, we obtain
\[
\partial_t\widehat I(t,x_n)
\le d_2\big(m(t)-\widehat I(t,x_n)\big)-\gamma_*\widehat I(t,x_n)+L(R_S+\widehat I(t,x_n))
\le d_2 m(t) + (L-\gamma_*)\widehat I(t,x_n)+L R_S.
\]
Letting $n\to\infty$ and using $\widehat I(t,x_n)\to m(t)$ gives
\[
D^+m(t)\le (d_2-\gamma_*+L)m(t)+L R_S,
\qquad t\in(0,T).
\]
Set $\alpha:=d_2-\gamma_*+L$. Applying Gronwall’s inequality for Dini derivatives yields
\begin{equation}\label{eq:Isup-est}
m(t)=\|\widehat I(t,\cdot)\|_\infty
\le e^{\alpha t}\|I_0\|_\infty+\frac{L R_S}{\alpha}\big(e^{\alpha t}-1\big),
\qquad t\in[0,T],
\end{equation}
with the convention $\frac{e^{\alpha t}-1}{\alpha}=t$ when $\alpha=0$.
In particular, for a fixed horizon $T$ we may define
\begin{equation}\label{eq:MI-choice}
M_I:=e^{\alpha T}\|I_0\|_\infty+\frac{L M_S}{\alpha}\big(e^{\alpha T}-1\big),
\end{equation}
so that $0\le \widehat I(t,x)\le M_I$ for all $(t,x)\in[0,T]\times\mathbb R$.

Consequently, $\|\widehat I\|_{X_I^T}\le M_I$ (since the $X_I^T$–norm is the supremum norm on the cylinder
together with the zero extension), and choosing $R_I:=M_I$ ensures that the operator
$\mathcal T$ maps $\mathbb B_{R_I}$ into itself.

%------------------------------------------------------------
\medskip
\noindent{Step 7. Contraction of $\mathcal T$ on a short time interval .}

Let $\widetilde I_1,\widetilde I_2\in\mathbb B_{R_I}$ and denote by
$(S_i,g_i,h_i)$ the unique solutions of \eqref{eq:S-subsystem-T1-recalled} corresponding to $\widetilde I=\widetilde I_i$
(on $[0,T]$), and set $I_i:=\mathcal T(\widetilde I_i)$, i.e. $I_i$ solves \eqref{eq:I-subsystem-T1} with
$\widetilde S=\widetilde S_i$ (the zero-extension of $S_i$ on $(g_i(t),h_i(t))$).
Define
\[
U:=S_1-S_2,\qquad V:=I_1-I_2,\qquad \Delta\widetilde I:=\widetilde I_1-\widetilde I_2,\qquad
\Delta g:=g_1-g_2,\qquad \Delta h:=h_1-h_2.
\]
We will estimate $\|V\|_{L^\infty([0,T]\times\mathbb R)}$ by $\|\Delta\widetilde I\|_{L^\infty([0,T]\times\mathbb R)}$.

Let $R:=\max\{M_S,M_I\}$ where $M_S$ and $M_I$ are the bounds obtained in Steps~3 and~6.
By (H1), $F$ is locally Lipschitz, hence there exists $L>0$ such that for all
$(S_j,I_j)\in[0,R]^2$,
\begin{equation}\label{eq:LipF-step7}
|F(S_1,I_1)-F(S_2,I_2)|\le L\big(|S_1-S_2|+|I_1-I_2|\big).
\end{equation}

\medskip
\noindent\emph{Estimate of $V$ in terms of $U$.}
Using the zero extension in $x$, each $I_i$ satisfies on $(0,T]\times\mathbb R$ the whole-line form
\[
(I_i)_t=d_2\big(J_2*I_i-I_i\big)+b(x)(I_i)_x-\gamma(x)I_i+F(\widetilde S_i,I_i),
\qquad I_i(0,x)=I_0^*(x),
\]
where $I_0^*$ is the extension of $I_0$ by $0$ outside $[-h_0,h_0]$.
Subtracting the two equations yields
\begin{equation}\label{eq:V-eq-step7}
V_t=d_2\big(J_2*V-V\big)+b(x)V_x-\gamma(x)V+\Big(F(\widetilde S_1,I_1)-F(\widetilde S_2,I_2)\Big),
\qquad V(0,x)=0.
\end{equation}
Let $m_V(t):=\|V(t,\cdot)\|_{L^\infty(\mathbb R)}$.
As in Step~6 (Dini-derivative argument), using $J_2\ge0$ and $\int_{\mathbb R}J_2=1$ gives
\[
\sup_{x\in\mathbb R}(J_2*V)(t,x)\le \|V(t,\cdot)\|_\infty=m_V(t),
\qquad
\inf_{x\in\mathbb R}(J_2*V)(t,x)\ge -m_V(t).
\]
Taking a sequence $x_n$ with $|V(t,x_n)|\to m_V(t)$ and evaluating \eqref{eq:V-eq-step7} at $(t,x_n)$,
we obtain the differential inequality (for the upper right Dini derivative)
\[
D^+ m_V(t)
\le d_2 m_V(t)+\Big\|F(\widetilde S_1,I_1)-F(\widetilde S_2,I_2)\Big\|_\infty,
\qquad t\in(0,T).
\]
By \eqref{eq:LipF-step7} and $\|\widetilde S_1-\widetilde S_2\|_\infty\le \|U\|_\infty$ (since both are zero-extensions),
\[
\Big\|F(\widetilde S_1,I_1)-F(\widetilde S_2,I_2)\Big\|_\infty
\le L\big(\|\widetilde S_1-\widetilde S_2\|_\infty+\|V\|_\infty\big)
\le L\big(\|U\|_\infty+m_V(t)\big).
\]
Hence
\[
D^+ m_V(t)\le (d_2+L)m_V(t)+L\|U\|_\infty,
\qquad m_V(0)=0.
\]
Gronwall’s inequality yields
\begin{equation}\label{eq:V-in-terms-U-step7}
\|V\|_{L^\infty([0,T]\times\mathbb R)}
=
\sup_{t\in[0,T]}m_V(t)
\le
L\,T\,e^{(d_2+L)T}\,\|U\|_{L^\infty([0,T]\times\mathbb R)}.
\end{equation}

\medskip
\noindent\emph{Estimate of $U$ and $(\Delta g,\Delta h)$ in terms of $\Delta\widetilde I$.}
Subtracting the $S$–equations in \eqref{eq:S-subsystem-T1-recalled} gives, for $x\in(g_1(t),h_1(t))\cap(g_2(t),h_2(t))$,
\begin{align*}
U_t
&=
d_1\!\left(\int_{g_1(t)}^{h_1(t)}J_1(x-y)S_1(t,y)\,dy
-\int_{g_2(t)}^{h_2(t)}J_1(x-y)S_2(t,y)\,dy\right)
-d_1U+a(x)U_x
\\
&\quad +\gamma(x)\Delta\widetilde I
-\Big(F(S_1,\widetilde I_1)-F(S_2,\widetilde I_2)\Big).
\end{align*}
Write the nonlocal difference as
\begin{align*}
&\int_{g_1}^{h_1}J_1(x-y)S_1\,dy-\int_{g_2}^{h_2}J_1(x-y)S_2\,dy \\
&\qquad=\int_{g_1}^{h_1}J_1(x-y)U(t,y)\,dy
+\left(\int_{g_1}^{h_1}-\int_{g_2}^{h_2}\right)J_1(x-y)S_2(t,y)\,dy.
\end{align*}
Using $J_1\ge0$ and $\int_{\mathbb R}J_1=1$, we have
\[
\left|\int_{g_1}^{h_1}J_1(x-y)U(t,y)\,dy\right|
\le \|U(t,\cdot)\|_\infty.
\]
For the domain-mismatch term, using $|S_2|\le M_S$ and $\|J_1\|_\infty<\infty$ (assumption (J1)),
\[
\left|\left(\int_{g_1}^{h_1}-\int_{g_2}^{h_2}\right)J_1(x-y)S_2(t,y)\,dy\right|
\le \|J_1\|_\infty M_S\big(|\Delta g(t)|+|\Delta h(t)|\big).
\]
Moreover, by \eqref{eq:LipF-step7},
\[
|F(S_1,\widetilde I_1)-F(S_2,\widetilde I_2)|
\le L\big(|U|+|\Delta\widetilde I|\big).
\]
Let $m_U(t):=\|U(t,\cdot)\|_{L^\infty(\mathbb R)}$ (where $U$ is understood with zero extension).
Arguing as in Steps~3 and~6 (Dini-derivative at near-maximum points) yields
\begin{equation}\label{eq:DU-step7}
D^+m_U(t)
\le C_0\,m_U(t)+C_0\,\|\Delta\widetilde I(t,\cdot)\|_\infty
+C_0\big(|\Delta g(t)|+|\Delta h(t)|\big),
\qquad t\in(0,T),
\end{equation}
where one can take
\[
C_0:=d_1+L+\|\gamma\|_\infty+\|J_1\|_\infty M_S.
\]
Since $U(0,x)=0$, we have $m_U(0)=0$. Integrating \eqref{eq:DU-step7} and using Gronwall gives
\begin{equation}\label{eq:U-pre-step7}
\|U\|_\infty
\le
C_0 T e^{C_0T}\Big(\|\Delta\widetilde I\|_\infty+\|\Delta g\|_{C[0,T]}+\|\Delta h\|_{C[0,T]}\Big),
\end{equation}
where all $\|\cdot\|_\infty$ norms are on $[0,T]\times\mathbb R$ unless otherwise specified.

We now bound $\Delta g,\Delta h$ by $\|U\|_\infty$.
From the boundary laws,
\[
h_i'(t)=\mu\int_{g_i(t)}^{h_i(t)}\int_{h_i(t)}^\infty J_1(x-y)S_i(t,x)\,dy\,dx,
\qquad
g_i'(t)=-\mu\int_{g_i(t)}^{h_i(t)}\int_{-\infty}^{g_i(t)} J_1(x-y)S_i(t,x)\,dy\,dx,
\]
and hence
\begin{align*}
|\Delta h(t)|
&=\left|\int_0^t (h_1'(\tau)-h_2'(\tau))\,d\tau\right|
\le \int_0^t |h_1'(\tau)-h_2'(\tau)|\,d\tau,\\
|\Delta g(t)|
&\le \int_0^t |g_1'(\tau)-g_2'(\tau)|\,d\tau.
\end{align*}
We estimate $|h_1'-h_2'|$ (the $g$–term is analogous). Let us denote by 
\begin{align*}
A_1(t)&:=\int_{g_1(t)}^{h_1(t)}\int_{h_1(t)}^\infty J_1(x-y)\,|U(t,x)|\,dy\,dx,\\
A_2(t)&:=\left|\left(\int_{g_1(t)}^{h_1(t)}-\int_{g_2(t)}^{h_2(t)}\right)\int_{h_1(t)}^\infty J_1(x-y)\,S_2(t,x)\,dy\,dx\right|,\\
A_3(t)&:=\left|\int_{g_2(t)}^{h_2(t)}\left(\int_{h_1(t)}^\infty-\int_{h_2(t)}^\infty\right)J_1(x-y)\,S_2(t,x)\,dy\,dx\right|,
\end{align*}
we obtain
\[
|h_1'(t)-h_2'(t)|
\le
\mu\Big( A_1(t)+A_2(t)+A_3(t)\Big).
\]
Then, using $J_1\ge0$ and $\int_{\mathbb R}J_1=1$, we get
\[
A_1(t)\le \int_{g_1(t)}^{h_1(t)} |U(t,x)|\left(\int_{\mathbb R}J_1(x-y)\,dy\right)dx
\le (h_1(t)-g_1(t))\,\|U\|_\infty.
\]
By the length bound from Step~4, $h_1(t)-g_1(t)\le 2h_0e^{2\mu M_S T}$, hence
\[
A_1(t)\le 2h_0e^{2\mu M_S T}\,\|U\|_\infty.
\]
For $A_2$, using $|S_2|\le M_S$ and $\int_{h_1(t)}^\infty J_1(x-y)\,dy\le1$,
\[
A_2(t)\le M_S\big(|\Delta g(t)|+|\Delta h(t)|\big).
\]
For $A_3$, observe that for each fixed $x$,
\[
\left|\int_{h_1(t)}^\infty-\int_{h_2(t)}^\infty\right|J_1(x-y)\,dy
=
\left|\int_{h_1(t)}^{h_2(t)}J_1(x-y)\,dy\right|
\le \|J_1\|_\infty\,|\Delta h(t)|,
\]
hence
\[
A_3(t)\le \int_{g_2(t)}^{h_2(t)} M_S\,\|J_1\|_\infty\,|\Delta h(t)|\,dx
\le M_S\|J_1\|_\infty\,(h_2(t)-g_2(t))\,|\Delta h(t)|
\le C_1\,|\Delta h(t)|
\]
with $C_1:=2h_0e^{2\mu M_S T}M_S\|J_1\|_\infty$.
Combining these bounds yields
\[
|h_1'(t)-h_2'(t)|
\le C_2\|U\|_\infty + C_2\big(|\Delta g(t)|+|\Delta h(t)|\big),
\]
where $C_2$ depends only on $\mu,h_0,M_S,\|J_1\|_\infty$.
Integrating in time gives
\[
\|\Delta h\|_{C[0,T]}
\le C_2T\,\|U\|_\infty + C_2T\big(\|\Delta g\|_{C[0,T]}+\|\Delta h\|_{C[0,T]}\big).
\]
The same estimate holds for $\|\Delta g\|_{C[0,T]}$. Summing them yields
\[
\|\Delta g\|_{C[0,T]}+\|\Delta h\|_{C[0,T]}
\le C_3T\,\|U\|_\infty + C_3T\big(\|\Delta g\|_{C[0,T]}+\|\Delta h\|_{C[0,T]}\big),
\]
for some $C_3$ depending only on the data.
Choose $T>0$ so small that $C_3T\le \frac12$. Then
\begin{equation}\label{eq:boundary-in-terms-U-step7}
\|\Delta g\|_{C[0,T]}+\|\Delta h\|_{C[0,T]}
\le 2C_3T\,\|U\|_\infty.
\end{equation}
Inserting \eqref{eq:boundary-in-terms-U-step7} into \eqref{eq:U-pre-step7} gives
\[
\|U\|_\infty
\le C_0Te^{C_0T}\Big(\|\Delta\widetilde I\|_\infty+2C_3T\|U\|_\infty\Big).
\]
Choose $T>0$ further so that $2C_0C_3T^2e^{C_0T}\le \frac12$; then
\begin{equation}\label{eq:U-final-step7}
\|U\|_\infty\le 2C_0Te^{C_0T}\,\|\Delta\widetilde I\|_\infty.
\end{equation}
Finally, combining \eqref{eq:U-final-step7} with \eqref{eq:V-in-terms-U-step7} yields
\[
\|V\|_\infty
\le
\Big(2LC_0\,T^2\,e^{(d_2+L)T}e^{C_0T}\Big)\,\|\Delta\widetilde I\|_\infty.
\]
Therefore, choosing $T>0$ sufficiently small so that
\[
2LC_0\,T^2\,e^{(d_2+L)T}e^{C_0T}<1,
\]
we obtain
\[
\|\mathcal T(\widetilde I_1)-\mathcal T(\widetilde I_2)\|_\infty
=\|V\|_\infty
\le \kappa\,\|\widetilde I_1-\widetilde I_2\|_\infty,
\qquad \kappa\in(0,1).
\]
Hence $\mathcal T$ is a contraction on $\mathbb B_{R_I}$ for $T$ sufficiently small.

\medskip
\noindent{Step 8. Local well-posedness on $[0,T]$ for small $T$, continuation to arbitrary horizons, and bounds.}

Fix $T>0$ (to be chosen small first). Let $R_I>0$ be chosen as in Step~6 so that
$\mathcal T:\mathbb B_{R_I}\to \mathbb B_{R_I}$ is well-defined, where
\[
\mathbb B_{R_I}:=\Big\{\widetilde I\in X_I^T:\ \|\widetilde I\|_{L^\infty([0,T]\times\mathbb R)}\le R_I\Big\}.
\]
By Step~7, there exists $T_*>0$ such that for every $T\in(0,T_*]$ the map $\mathcal T$ is a contraction on
$\mathbb B_{R_I}$, i.e. there exists $\kappa=\kappa(T)\in(0,1)$ such that for all
$\widetilde I_1,\widetilde I_2\in\mathbb B_{R_I}$,
\[
\|\mathcal T(\widetilde I_1)-\mathcal T(\widetilde I_2)\|_{L^\infty([0,T]\times\mathbb R)}
\le \kappa\,\|\widetilde I_1-\widetilde I_2\|_{L^\infty([0,T]\times\mathbb R)}.
\]
Hence, by Banach’s fixed point theorem, there exists a unique $I\in\mathbb B_{R_I}$ such that
\[
\mathcal T(I)=I.
\]
By definition of $\mathcal T$, the fixed point $I$ means that if we solve the $(S,g,h)$–subsystem
\eqref{eq:S-subsystem-T1-recalled} with $\widetilde I=I$, obtaining a unique triple $(S,g,h)$ on $[0,T]$,
and then solve the $I$–equation \eqref{eq:I-subsystem-T1} with the corresponding $\widetilde S$,
the solution is exactly $I$ again. Therefore the quadruple $(S,I,g,h)$ solves the full free-boundary system
\eqref{eq:SIS-free-boundary} classically on $[0,T]$.
Moreover, if $(S^{(1)},I^{(1)},g^{(1)},h^{(1)})$ and $(S^{(2)},I^{(2)},g^{(2)},h^{(2)})$ were two solutions on $[0,T]$,
then $I^{(j)}$ would both be fixed points of $\mathcal T$, contradicting uniqueness of the fixed point.
Thus the solution on $[0,T]$ is unique.

We next extend the solution to an arbitrary time horizon. Let $\delta\in(0,T_*]$ be fixed so that the contraction
argument holds on every time interval of length $\delta$ with the same structure, and such that the constants in the estimates
depend only on uniform $L^\infty$ bounds for $(S,I)$ on that interval (as in Steps~3 and~6).
We already have a unique solution on $[0,\delta]$. Denote the terminal data at $t=\delta$ by
\[
S(\delta,x)=:S_\delta(x),\qquad I(\delta,x)=:I_\delta(x),\qquad g(\delta)=:g_\delta,\qquad h(\delta)=:h_\delta.
\]
Because $g,h$ are Lipschitz and the solution is continuous up to $t=\delta$, these data are well-defined.
We now repeat the same fixed-point construction on $[\delta,2\delta]$ with initial data
$S(\delta,\cdot)=S_\delta$, $I(\delta,\cdot)=I_\delta$, $g(\delta)=g_\delta$, $h(\delta)=h_\delta$.
The local theory applies because the bounds obtained below ensure that the required sup-norm radii remain finite.
This produces a unique solution on $[\delta,2\delta]$ that matches the previous one at $t=\delta$.
Iterating finitely many times, we obtain a unique classical solution on $[0,T]$ for any prescribed $T>0$.

Finally, we establish positivity and uniform bounds on $[0,T]$.
Positivity follows on each local interval from Lemma~\ref{lem:maximum} applied to the $S$–equation
(with the nonnegative source $\gamma I$ and $F\ge0$) and to the $I$–equation
(with the nonnegative source $F(S,I)$), together with the zero extension outside $(g(t),h(t))$:
\[
S(t,x)\ge0,\qquad I(t,x)\ge0,\qquad (t,x)\in[0,T]\times\mathbb R.
\]

To obtain upper bounds, fix $R>0$ large enough so that it dominates the initial norms and all intermediate bounds
on each local step, and let $L=L_F(R)$ be a Lipschitz constant of $F$ on $[0,R]^2$.
Using $F(0,0)=0$ and Lipschitz continuity, for $0\le S\le R$ and $0\le I\le R$ we have
\begin{equation}\label{eq:F-growth-step8}
0\le F(S,I)=|F(S,I)-F(0,0)|\le L(|S|+|I|)\le L(R+I).
\end{equation}
Define the constant
\[
M_S:=e^{d_1 T}\|S_0\|_\infty+\frac{\|\gamma\|_\infty}{d_1}\big(e^{d_1T}-1\big)\,R,
\]
which is exactly the bound obtained from the Gronwall estimate in Step~3 (with $\|\widetilde I\|_\infty\le R$).
Then, by Lemma~\ref{lem:comparison} applied to the $S$–equation, $S(t,x)\le M_S$ on $[0,T]\times\mathbb R$.

Next, let $\gamma_*:=\inf_{x\in\mathbb R}\gamma(x)>0$ and set $\alpha:=d_2-\gamma_*+L$.
Consider the scalar ODE
\[
\overline I'(t)=\alpha\,\overline I(t)+L\,M_S,\qquad \overline I(0)=\|I_0\|_\infty.
\]
Its solution is explicit:
\[
\overline I(t)=e^{\alpha t}\|I_0\|_\infty+\frac{L M_S}{\alpha}\big(e^{\alpha t}-1\big)
\quad \text{for }\alpha\neq0,
\qquad
\overline I(t)=\|I_0\|_\infty+L M_S\,t \quad \text{for }\alpha=0.
\]
Using \eqref{eq:F-growth-step8} with $S\le M_S$ and $I\le \overline I(t)$, we obtain
\[
F(S,I)\le L(M_S+I)\le L(M_S+\overline I(t)).
\]
Hence the pair $(\overline S,\overline I):=(M_S,\overline I(t))$ satisfies the super-solution inequalities
in Lemma~\ref{lem:comparison} for the $(S,I)$–system (with the same moving interval and zero extension),
and therefore
\[
0\le S(t,x)\le M_S,\qquad 0\le I(t,x)\le \overline I(t)\le \overline I(T)=:M_I,
\qquad (t,x)\in[0,T]\times\mathbb R.
\]
In particular, defining
\[
M_I:=e^{\alpha T}\|I_0\|_\infty+\frac{L M_S}{\alpha}\big(e^{\alpha T}-1\big)
\quad (\text{or }M_I:=\|I_0\|_\infty+L M_S T \text{ if }\alpha=0),
\]
we obtain the desired uniform bounds on $[0,T]$:
\[
0\le S(t,x)\le M_S,\qquad 0\le I(t,x)\le M_I,\qquad (t,x)\in[0,T]\times\mathbb R.
\]
This completes the proof of existence, uniqueness, positivity and uniform bounds on $[0,T]$ for arbitrary $T>0$.
\end{proof}

\begin{remark}
The well-posedness result of Theorem~\ref{Theorem1} relies in an essential way on the regularity and monotonicity
assumptions imposed on the incidence function $F$ in {\rm(H1)}.
In particular, the assumptions
\[
F \in C^{1}(\mathbb R_+^2), \qquad
\partial_S F \ge 0, \qquad \partial_I F \ge 0,
\]
ensure that the reaction terms are locally Lipschitz on bounded subsets of $\mathbb R_+^2$
and that the system is quasi-monotone, which are key ingredients in the use of the
maximum principle, the comparison principle, and the contraction mapping argument.

By contrast, the classical normalized incidence
\[
F(S,I)=\frac{SI}{S+I}
\]
does not belong to $C^1(\mathbb R_+^2)$ and is not locally Lipschitz in a neighborhood of the origin $(0,0)$.
As a consequence, the right-hand side of the system loses regularity at low population densities,
and the standard well-posedness strategy based on fixed-point arguments and comparison principles
breaks down. In particular, uniqueness and continuous dependence on initial data are no longer guaranteed
within the classical solution framework considered here.
\end{remark}

\section{Linearization and the eigenvalue}

The primary objective of this section is to identify the correct spectral quantity that governs the invasion or extinction of the disease. To achieve this, we linearize the system around the disease--free equilibrium and derive an appropriate eigenvalue problem whose sign determines the long--time behavior of the infected population.

A major difficulty arises from the coupled nonlocal structure of the model together with the presence of advection terms and nonsymmetric dispersal kernels. In particular, the associated operators are generally not self--adjoint, and classical variational characterizations of principal eigenvalues are therefore unavailable. As a consequence, the spectral threshold cannot be obtained directly from a single scalar equation.

To overcome this issue, we reformulate the linearized system as a block operator and employ a Schur complement argument to reduce it to an effective scalar eigenvalue problem for the infected component. This reduction is carried out under suitable structural assumptions ensuring that the resulting operator is order-preserving, allowing us to characterize a generalized principal eigenvalue in the sense of Coville.

This eigenvalue plays the role of a threshold parameter: its sign determines whether small infectious perturbations decay or grow, and thus provides the mathematical foundation for the spreading--vanishing dichotomy established in the subsequent sections.

Let us begin with the definition of generalized principal eigenvalue :

\begin{definition}[Generalized principal eigenvalue]\label{genprin}
Let $-\infty < L_1 < L_2 < +\infty$ and $d>0$. Consider the nonlocal--advection operator
\begin{align*}
\mathcal{L}_{(L_1,L_2),d,p}[\phi](x)
&:= d\int_{L_1}^{L_2} J(x-y)\phi(y)\,dy 
- d\phi(x) \\
&\quad - p(x)\phi'(x) + a(x)\phi(x),
\qquad x\in(L_1,L_2),
\end{align*}
where the homogeneous Dirichlet exterior condition is understood by extending 
$\phi \equiv 0$ outside $[L_1,L_2]$ when evaluating the nonlocal term. The \emph{generalized principal eigenvalue} of $\mathcal{L}_{(L_1,L_2),d,p}$ is defined as
\begin{align*}
\lambda_p\big(\mathcal{L}_{(L_1,L_2),d,p}\big)
:= \sup \Big\{ \lambda\in\mathbb{R}:\;
&\exists\,\phi\in C^1([L_1,L_2]),\ \phi>0 \text{ in }(L_1,L_2), \\
&\mathcal{L}_{(L_1,L_2),d,p}[\phi](x)
+ \lambda\,\phi(x)\le0 \text{ for all } x\in(L_1,L_2)
\Big\},
\end{align*}
\end{definition}
which can be expressed equivalently by the sup–inf formula:
\[
\lambda_p\big(\mathcal{L}_{(L_1,L_2),d,p}\big)
= \sup_{\substack{\varphi \in C(\overline{\Omega}) \\ \varphi>0 \text{ in } \Omega}}
\ \inf_{x \in \Omega}
\left(
-\frac{\mathcal{L}_{\Omega,J}[\varphi](x) }{\varphi(x)}
\right).
\]

\subsection{Existence of a steady state for the disease--free operator}

We consider the disease--free equation
\begin{equation}\label{eq:S-disease-free}
S_t(t,x)
= d_1\!\left(\int_{\Omega} J_1(x-y)\,S(t,y)\,dy - S(t,x)\right)
+ a(x)\,\partial_x S(t,x),
\quad x\in\Omega,
\end{equation}
posed on a fixed bounded interval $\Omega=\subset\mathbb{R}$.

\medskip

\begin{proposition}\label{diseasefree}
There exists a disease-free equilibrium $S^*(x)$ of \eqref{eq:S-disease-free}
satisfying
\begin{equation}\label{eq:S-equilibrium}
d_1\left(\int_{\Omega} J_1(x-y)\,S^*(y)\,dy - S^*(x)\right)
+ a(x)\,\partial_x S^*(x)=0,
\quad x\in\Omega.
\end{equation}
\end{proposition}

\begin{proof} For $\phi \in C^1(\Omega)$, let us define
\begin{equation}\nonumber
\mathcal{L}[\phi](x)
:= d_1\left(\int_{\Omega} J_1(x-y)\,\phi(y)\,dy - \phi(x)\right)
+ a(x)\,\partial_x \phi(x),
\quad x\in\Omega.
\end{equation}
Thanks to (J1), by Proposition~1.1 in
\cite{CovilleHamel2020}, the operator $\mathcal{L}$ admits a generalized principal eigenvalue
$\lambda_1=\lambda_p(\mathcal{L})$ in the sense of Definition~\ref{genprin}, together with an
eigenfunction $\phi\in C^1(\overline{\Omega})$ such that $\phi>0$ in $\Omega$ and
\[
\mathcal{L}\phi=\lambda_1\phi\quad\text{in }\Omega.
\]
We claim that $\lambda_1=0$. Indeed, taking $\varphi\equiv 1$ in Definition~\ref{genprin}, we have
$\varphi\in C^1(\overline{\Omega})$, $\varphi>0$ in $\Omega$, and
\[
\mathcal{L}1
= d_1\left(\int_{\Omega} J_1(x-y)\,dy-1\right)\le 0 \quad \text{in }\Omega.
\]
Hence $0$ belongs to the admissible set in Definition~\ref{genprin}, and therefore
$\lambda_1=\lambda_p(\mathcal{L})\ge 0$.

On the other hand, we show that $\lambda_p(\mathcal L)\le 0$.
Let $\lambda>0$ be arbitrary. Suppose, by contradiction, that there exists
$\psi\in C^1(\overline\Omega)$ with $\psi>0$ in $\Omega$ such that
\[
\mathcal L\psi(x)+\lambda\psi(x)\le 0 \qquad \text{for all }x\in\Omega.
\]
Fix $\varepsilon>0$ and set
\[
\psi_\varepsilon(x):=\psi(x)+\varepsilon (x-L_1)(L_2-x),\qquad x\in[L_1,L_2].
\]
Since $(x-L_1)(L_2-x)>0$ in $(L_1,L_2)$ and vanishes at $L_1,L_2$, the maximum of
$\psi_\varepsilon$ is attained at some point $x_\varepsilon\in(L_1,L_2)$.
Moreover, at this interior maximizer we have $\psi_\varepsilon'(x_\varepsilon)=0$, hence
\[
\psi'(x_\varepsilon)=-\varepsilon(L_1+L_2-2x_\varepsilon),
\quad\text{so that}\quad
|\psi'(x_\varepsilon)|\le \varepsilon(L_2-L_1)\xrightarrow[\varepsilon\to 0]{}0.
\]
Let $M:=\max_{\overline\Omega}\psi=\lim_{\varepsilon\to0}\psi(x_\varepsilon)$ (up to a subsequence).
Using $J_1\ge 0$ and $\int_{\mathbb R}J_1=1$, we estimate
\[
\int_\Omega J_1(x_\varepsilon-y)\psi(y)\,dy \le M\int_\Omega J_1(x_\varepsilon-y)\,dy \le M.
\]
Evaluating the assumed inequality at $x_\varepsilon$ yields
\[
0 \ge \mathcal L\psi(x_\varepsilon)+\lambda\psi(x_\varepsilon)
= d_1\!\left(\int_\Omega J_1(x_\varepsilon-y)\psi(y)\,dy-\psi(x_\varepsilon)\right)
+ a(x_\varepsilon)\psi'(x_\varepsilon)+\lambda\psi(x_\varepsilon).
\]
The first bracket is $\le d_1(M-\psi(x_\varepsilon))$, and $a$ is bounded on $\overline\Omega$, hence
\[
0 \ge d_1\big(M-\psi(x_\varepsilon)\big) + a(x_\varepsilon)\psi'(x_\varepsilon)+\lambda\psi(x_\varepsilon).
\]
Letting $\varepsilon\to 0$ and using $\psi(x_\varepsilon)\to M$ and $\psi'(x_\varepsilon)\to 0$, we obtain
\[
0 \ge \lambda M,
\]
which is impossible since $\lambda>0$ and $M>0$. Therefore, no $\lambda>0$ can satisfy
Definition~\ref{genprin}, and hence $\lambda_p(\mathcal L)\le 0$.
Together with $\lambda_p(\mathcal L)\ge 0$ (from the test function $1$), we conclude that
$\lambda_p(\mathcal L)=0$.

Finally, since $\mathcal{L}\phi=\lambda_1\phi$ and $\lambda_1=0$, the function
$S^*:=\phi$ is a nontrivial nonnegative (indeed positive) solution of $\mathcal{L}S^*=0$ in $\Omega$,
i.e., $S^*$ solves \eqref{eq:S-equilibrium}. This yields the existence of a disease-free equilibrium.

\end{proof}

Next, we consider the nonlinear SIS system posed on a fixed bounded interval 
$[L_1,L_2]\subset\mathbb R$:
\begin{equation}\label{eq:SIS-fixed}
\begin{cases}
S_t = \mathcal L_S[S] - F(S,I) + a(x)S_x + \gamma(x) I,\\[3pt]
I_t = \mathcal L_I[I] + F(S,I) - \gamma(x) I + b(x)I_x,
\end{cases}
\qquad t>0,\ x\in [L_1,L_2],
\end{equation}
where the nonlocal dispersal operators are defined by
\[
\mathcal L_S[\phi](x)
= d_1\int_{L_1}^{L_2}J_1(x-y)\phi(y)\,dy - d_1\phi(x),\qquad
\mathcal L_I[\phi](x)
= d_2\int_{L_1}^{L_2}J_2(x-y)\phi(y)\,dy - d_2\phi(x).
\]

The disease-free susceptible profile $S^*$ is defined as a nonnegative steady
solution of
\begin{equation}\label{eq:disease-free}
\mathcal L_S[S^*] - F(S^*,0) + a(x)(S^*)_x = 0 
\quad \text{in } [L_1,L_2],
\end{equation}
which is proved to exist by Proposition \ref{diseasefree}.

We investigate the stability of the disease-free equilibrium $(S^*,0)$ by 
linearizing system \eqref{eq:SIS-fixed} around this state. 
Set 
\[
S = S^* + \tilde s, \qquad I = \tilde i,
\]
where $(\tilde s,\tilde i)$ represents a small perturbation. 
Using the Taylor expansion of $F$ near $(S^*,0)$, we obtain
\[
F(S^*+\tilde s,\tilde i)
= F(S^*,0)
+ F_S(S^*,0)\,\tilde s
+ F_I(S^*,0)\,\tilde i
+ o(\|(\tilde s,\tilde i)\|).
\]
Since $F(S^*,0)=0$, neglecting higher-order terms yields the linearized system
\[
\begin{cases}
\tilde s_t = (\mathcal L_S + a(x)\partial_x)\tilde s 
- \alpha(x)\tilde s 
- F_I(S^*,0)\tilde i 
+ \gamma(x)\tilde i,\\[3pt]
\tilde i_t = (\mathcal L_I + b(x)\partial_x)\tilde i 
+ \alpha(x)\tilde s 
+ \big(F_I(S^*,0)-\gamma(x)\big)\tilde i,
\end{cases}
\]
where we have introduced the coefficient functions
\[
\alpha(x):=F_S(S^*(x),0), 
\qquad 
\beta(x):=F_I(S^*(x),0)-\gamma(x).
\]

\paragraph{Block-operator matrix form.}
The linearized system can be written compactly as an abstract evolution equation
\[
\frac{d}{dt}\begin{pmatrix}\tilde s\\ \tilde i\end{pmatrix}
=\mathcal A\begin{pmatrix}\tilde s\\ \tilde i\end{pmatrix},
\qquad
\mathcal A:=\begin{pmatrix}
A_s & B\\[3pt]
C & A_i
\end{pmatrix},
\]
where the diagonal operators are given by
\[
A_s:=\mathcal L_S + a(x)\partial_x - \alpha(x),\qquad
A_i:=\mathcal L_I + b(x)\partial_x + \beta(x),
\]
and the off-diagonal terms represent multiplication (coupling) operators,
\[
B:= -F_I(S^*(\cdot),0) + \gamma(\cdot), 
\qquad 
C:=\alpha(\cdot).
\]

The associated eigenvalue problem for the linearized generator $\mathcal A$ reads
\[
\mathcal A\begin{pmatrix}\phi\\\psi\end{pmatrix}
= \lambda\begin{pmatrix}\phi\\\psi\end{pmatrix},
\]
that is,
\[
\begin{cases}
A_s[\phi] + B\psi = \lambda\phi,\\[3pt]
C\phi + A_i[\psi] = \lambda\psi.
\end{cases}
\tag{EP}
\]

The asymptotic behavior of small perturbations is governed by the spectral bound
\[
s(\mathcal A)
:=\sup\{\Re\lambda:\lambda\in\sigma(\mathcal A)\}.
\]
In particular, the disease-free equilibrium $(S^*,0)$ is linearly unstable 
(and hence invasion occurs) whenever $s(\mathcal A)>0$, 
whereas $s(\mathcal A)<0$ corresponds to linear stability and decay of small 
infected perturbations.

\begin{lemma}[Schur Complement]\label{Schur_Complement}
\label{lem:AS-stability}
Assume that there exists a constant $\eta_0>0$ such that
\begin{equation}\label{A_S}
\alpha(x)=F_S(S^*(x),0)\ge \eta_0 
\qquad \text{for all } x\in [L_1,L_2].
\end{equation}
Then the spectral bound of $A_S$ satisfies
\[
s(A_S) := \sup\{\Re \lambda \,:\, \lambda \in \sigma(A_S)\} < 0.
\]
In particular, there exists $\omega_0>0$ such that
\[
(A_S-\lambda I)^{-1} \text{ exists for all } \Re\lambda \ge -\omega_0,
\]
and $A_S$ generates an exponentially stable positive $C_0$-semigroup.
\end{lemma}

\begin{proof}
Recall that
\[
A_S=\mathcal L_S+a(x)\partial_x-\alpha(x),
\qquad 
\mathcal L_S\phi(x):=d_S\Big(\int_{\Omega}J_S(x-y)\phi(y)\,dy-\phi(x)\Big),
\]
and assume \eqref{A_S}, namely that there exists $\underline{\alpha}>0$ such that
\[
\alpha(x)\ge \underline{\alpha}\qquad \text{for all }x\in\Omega=[L_1,L_2].
\]

Let $\phi\in C^1(\overline{\Omega})$ with $\phi\ge0$ and set
\[
M:=\max_{\overline{\Omega}}\phi,
\qquad x_0\in\overline{\Omega}\ \text{such that}\ \phi(x_0)=M.
\]

We first compute the sign of $\mathcal L_S\phi(x_0)$. Since $J_S\ge0$ and
$\phi(y)\le \phi(x_0)$ for all $y\in\Omega$, we have
\[
\int_{\Omega}J_S(x_0-y)\phi(y)\,dy
\le \phi(x_0)\int_{\Omega}J_S(x_0-y)\,dy.
\]
By the normalization $\int_{\mathbb R}J_S(z)\,dz=1$ and the change of variables
$z=x_0-y$, we obtain
\[
\int_{\Omega}J_S(x_0-y)\,dy
=\int_{x_0-\Omega}J_S(z)\,dz
\le \int_{\mathbb R}J_S(z)\,dz
=1,
\]
and therefore
\[
\int_{\Omega}J_S(x_0-y)\phi(y)\,dy \le \phi(x_0).
\]
Consequently,
\[
\mathcal L_S\phi(x_0)
=d_S\Big(\int_{\Omega}J_S(x_0-y)\phi(y)\,dy-\phi(x_0)\Big)\le 0.
\]

If $x_0\in\Omega$ (an interior maximum point), then $\phi'(x_0)=0$, and hence
\[
(A_S\phi)(x_0)
=\mathcal L_S\phi(x_0)+a(x_0)\phi'(x_0)-\alpha(x_0)\phi(x_0)
\le -\alpha(x_0)\phi(x_0)
\le -\underline{\alpha}\,\phi(x_0).
\]

If $x_0\in\partial\Omega$, we use a perturbation to shift the maximum into the interior
and keep all terms under quantitative control. Define
\[
q(x):=(x-L_1)(L_2-x)\ge0\quad \text{on }\overline{\Omega},\qquad q=0\ \text{on }\partial\Omega,
\]
and for $\varepsilon>0$ set
\[
\phi_\varepsilon(x):=\phi(x)+\varepsilon q(x),\qquad x\in\overline{\Omega}.
\]
Since $q>0$ in $\Omega$ and $q=0$ on $\partial\Omega$, the function $\phi_\varepsilon$
attains its maximum at some point $x_\varepsilon\in\Omega$ (an interior maximizer).
At this interior maximum, we have $\phi_\varepsilon'(x_\varepsilon)=0$ and
$\phi_\varepsilon(y)\le \phi_\varepsilon(x_\varepsilon)$ for all $y\in\Omega$.
Repeating the nonlocal computation with $\phi_\varepsilon$ yields
\[
\mathcal L_S\phi_\varepsilon(x_\varepsilon)
=d_S\Big(\int_{\Omega}J_S(x_\varepsilon-y)\phi_\varepsilon(y)\,dy-\phi_\varepsilon(x_\varepsilon)\Big)\le 0,
\]
and since $\phi_\varepsilon'(x_\varepsilon)=0$ we obtain
\[
(A_S\phi_\varepsilon)(x_\varepsilon)
=\mathcal L_S\phi_\varepsilon(x_\varepsilon)+a(x_\varepsilon)\phi_\varepsilon'(x_\varepsilon)-\alpha(x_\varepsilon)\phi_\varepsilon(x_\varepsilon)
\le -\alpha(x_\varepsilon)\phi_\varepsilon(x_\varepsilon)
\le -\underline{\alpha}\,\phi_\varepsilon(x_\varepsilon).
\]
Using the linearity of $A_S$ and $\phi_\varepsilon=\phi+\varepsilon q$, we write
\[
(A_S\phi)(x_\varepsilon)
=(A_S\phi_\varepsilon)(x_\varepsilon)-\varepsilon (A_S q)(x_\varepsilon)
\le -\underline{\alpha}\,\phi_\varepsilon(x_\varepsilon)+\varepsilon\,\|(A_S q)\|_{L^\infty(\Omega)}.
\]
We now bound $\|(A_S q)\|_{L^\infty(\Omega)}$ explicitly. Since $q$ and $q'$ are bounded on
$\overline{\Omega}$ and $J_S\ge0$ with $\int_{\mathbb R}J_S=1$, we have for every $x\in\Omega$,
\[
\Big|\int_\Omega J_S(x-y)q(y)\,dy\Big|
\le \|q\|_{L^\infty(\Omega)}\int_\Omega J_S(x-y)\,dy
\le \|q\|_{L^\infty(\Omega)}.
\]
Hence
\[
\|\mathcal L_S q\|_{L^\infty(\Omega)}
\le d_S\Big(\|q\|_{L^\infty(\Omega)}+\|q\|_{L^\infty(\Omega)}\Big)
=2d_S\|q\|_{L^\infty(\Omega)},
\]
and therefore
\[
\|A_S q\|_{L^\infty(\Omega)}
\le \|\mathcal L_S q\|_{L^\infty(\Omega)}+\|a\|_{L^\infty(\Omega)}\|q'\|_{L^\infty(\Omega)}+\|\alpha\|_{L^\infty(\Omega)}\|q\|_{L^\infty(\Omega)}
=:K_0<\infty.
\]
Consequently,
\[
(A_S\phi)(x_\varepsilon)\le -\underline{\alpha}\,\phi_\varepsilon(x_\varepsilon)+\varepsilon K_0.
\]

Finally, since $\phi_\varepsilon\ge \phi$ on $\overline{\Omega}$, we have
$\phi_\varepsilon(x_\varepsilon)=\max_{\overline{\Omega}}\phi_\varepsilon\ge \max_{\overline{\Omega}}\phi=M$.
Thus,
\[
(A_S\phi)(x_\varepsilon)\le -\underline{\alpha}\,M+\varepsilon K_0.
\]
Letting $\varepsilon\to0$ and extracting a subsequence if necessary, we may assume
$x_\varepsilon\to \bar x\in\overline{\Omega}$. By continuity of $A_S\phi$ (all terms are continuous
since $\phi\in C^1(\overline{\Omega})$ and $\mathcal L_S\phi$ is continuous), we obtain
\[
(A_S\phi)(\bar x)\le -\underline{\alpha}\,M.
\]
Since $\bar x$ is a limit point of maximizers of $\phi_\varepsilon$ and $\phi_\varepsilon\downarrow \phi$
pointwise as $\varepsilon\to0$, one has $\phi(\bar x)=M$. Hence we have shown:
for every nonnegative $\phi\in C^1(\overline{\Omega})$, at a maximum point $x_0$ of $\phi$,
\[
(A_S\phi)(x_0)\le -\underline{\alpha}\,\phi(x_0).
\]

We now derive the exponential decay estimate for the semigroup.
Let $u(t,\cdot):=e^{tA_S}\phi$ be the (mild/classical) solution of $u_t=A_Su$ with $u(0)=\phi\ge0$.
Define
\[
m(t):=\|u(t,\cdot)\|_{L^\infty(\Omega)}=\max_{\overline{\Omega}}u(t,x).
\]
For each $t>0$, choose $x_t\in\overline{\Omega}$ such that $u(t,x_t)=m(t)$.
Applying the previous maximum-point estimate to the function $x\mapsto u(t,x)$ yields
\[
u_t(t,x_t)=(A_Su(t,\cdot))(x_t)\le -\underline{\alpha}\,u(t,x_t)=-\underline{\alpha}\,m(t).
\]
Therefore the upper right Dini derivative satisfies
\[
D^+ m(t)\le -\underline{\alpha}\,m(t),
\]
and Gr\"onwall's inequality gives
\[
\|e^{tA_S}\phi\|_{L^\infty(\Omega)}=m(t)\le e^{-\underline{\alpha}t}m(0)
=e^{-\underline{\alpha}t}\|\phi\|_{L^\infty(\Omega)}\qquad \text{for all }t\ge0.
\]
This shows that the positive $C_0$-semigroup generated by $A_S$ is exponentially stable and
\[
s(A_S)\le -\underline{\alpha}<0.
\]
Finally, since $s(A_S)<0$, the resolvent exists on a right half-plane:
there exists $\omega_0>0$ such that $(A_S-\lambda I)^{-1}$ exists for all $\Re\lambda\ge -\omega_0$.
\end{proof}

\begin{remark}
The above Schur complement reduction is not a mere technical manipulation.
Its purpose is to rigorously justify that the spectral stability of the
disease-free equilibrium for the coupled SIS system is governed by a scalar
operator acting only on the infected component. 
More precisely, under the stability of the susceptible subsystem ensured by
Lemma~\ref{Schur_Complement}, the block eigenvalue problem for $\mathcal A$
is equivalent to the spectral problem for the effective operator
$\mathcal L_{\mathrm{eff}}$. 
This provides a mathematically sound foundation for defining the invasion
threshold in terms of the generalized principal eigenvalue of
$\mathcal L_{\mathrm{eff}}$, and explains why the long-term dynamics of the
full system (invasion versus extinction) can be characterized through the
spectral properties of a scalar nonlocal operator.
\end{remark}

By Lemma~\ref{Schur_Complement}, the operator $A_s$ has strictly negative spectral bound. 
In particular, there exists $\omega_0>0$ such that the resolvent 
\[
(A_s-\lambda I)^{-1}
\]
is well-defined and bounded for all $\lambda$ satisfying $\Re\lambda\ge -\omega_0$. 
Therefore, for every such $\lambda$, the first equation in the eigenvalue problem \emph{(EP)}
\[
A_s\phi + B\psi = \lambda \phi
\]
can be uniquely solved for $\phi$, yielding
\[
(A_s-\lambda I)\phi = -B\psi
\quad \Longrightarrow \quad
\phi = -(A_s-\lambda I)^{-1}B\psi.
\]

Substituting this expression into the second equation of \emph{(EP)},
\[
C\phi + A_i\psi = \lambda\psi,
\]
we obtain
\[
C\big(-(A_s-\lambda I)^{-1}B\psi\big) + A_i\psi = \lambda\psi,
\]
that is,
\[
\big[A_i - C(A_s-\lambda I)^{-1}B\big]\psi = \lambda\psi.
\]
This shows that $\lambda$ is an eigenvalue of the full block operator $\mathcal A$
if and only if $\lambda$ is an eigenvalue of the (parameter-dependent) operator
\[
\mathcal S(\lambda):=A_i - C(A_s-\lambda I)^{-1}B,
\]
with corresponding eigenfunction $\psi$.  
In other words, the spectral problem for $\mathcal A$ is reduced to a nonlinear
eigenvalue problem involving only the infected component through the Schur
complement $\mathcal S(\lambda)$.  
Moreover, whenever $\lambda$ lies in a region where $(A_s-\lambda I)^{-1}$ depends
analytically on $\lambda$, the mapping $\lambda\mapsto\mathcal S(\lambda)$ is analytic
in the operator norm, and the spectral equivalence between $\mathcal A$ and
$\mathcal S(\lambda)$ can be rigorously justified by standard Schur complement arguments.

\paragraph{Scalar effective operator at $\lambda=0$.}
A particularly useful situation occurs when one evaluates the Schur complement at
$\lambda=0$.  
Since Lemma~\ref{Schur_Complement} guarantees that $A_s$ is invertible and generates an
exponentially stable semigroup, the inverse $A_s^{-1}$ is a bounded operator.
In this case, we define the effective infection operator by
\[
\mathcal L_{\mathrm{eff}} := A_i - C A_s^{-1} B.
\]
The spectral properties of $\mathcal L_{\mathrm{eff}}$ provide a natural threshold
for invasion.  
More precisely, one studies the generalized principal eigenvalue
$\lambda_p(\mathcal L_{\mathrm{eff}})$.  
If $\lambda_p(\mathcal L_{\mathrm{eff}})>0$, then the spectral bound of the full block
operator $\mathcal A$ is positive, and small infective perturbations grow exponentially,
leading to invasion.  
If, on the contrary, $\lambda_p(\mathcal L_{\mathrm{eff}})<0$ and the spectral gap
assumptions ensuring the validity of the reduction hold, then the disease-free equilibrium
is linearly stable and the infection decays.

As a first approximation, one may neglect the coupling effects, for instance when
the operators $B$ or $C$ are small in norm.  
In this case, the relevant scalar operator is simply $A_i$, namely
\[
A_i[\psi](x)
= d_2\int_{L_1}^{L_2}J_2(x-y)\psi(y)\,dy
- d_2\psi(x)
+ b(x)\psi'(x)
+ \beta(x)\psi(x),
\]
and its generalized principal eigenvalue is defined by
\[
\lambda_p(A_i)
:=\inf\{\lambda\in\mathbb R:\exists\psi>0,\ A_i[\psi]\le\lambda\psi\}.
\]
This quantity provides a preliminary invasion threshold:
\[
\lambda_p(A_i)>0 \ \Longrightarrow\ \text{invasion}, 
\qquad 
\lambda_p(A_i)<0 \ \Longrightarrow\ \text{no invasion},
\]
as long as the coupling terms do not modify the sign of the spectral bound.

The above reduction becomes exact, in the sense that
\[
s(\mathcal A)=\lambda_p(\mathcal L_{\mathrm{eff}}),
\]
under additional structural conditions. Typical sufficient conditions include:
\begin{itemize}
\item The operator $A_s$ has a strictly negative spectral bound, and moreover there
exists $\delta>0$ such that
\[
\Re\lambda\le -\delta
\quad \text{for all } \lambda\in\sigma(A_s).
\]
In this case, the resolvent $(A_s-\lambda I)^{-1}$ is bounded for all
$\Re\lambda\ge -\delta/2$, and classical perturbation theory together with the Schur
complement framework implies that the spectral bound of $\mathcal A$ coincides with that
of $\mathcal L_{\mathrm{eff}}$.
\item Alternatively, the multiplication operators $B$ and $C$ are small in operator norm,
so that the composite operator $C A_s^{-1} B$ can be viewed as a relatively small
perturbation of $A_i$. In this regime, standard stability results for bounded
perturbations again yield spectral equivalence.
\end{itemize}
If neither of these situations holds, then the reduction to a scalar operator is no
longer justified, and the spectral analysis must be carried out directly on the full
block operator $\mathcal A$.

\subsection{Asymptotics of the principal eigenvalue for small diffusion and for vanishing interval}

\begin{theorem}[Small diffusion limit]\label{thm:small-d}
Let $L_1<L_2$ and  $\lambda_p\big(A_i^{(d)}\big)$ denote the generalized principal eigenvalue of $A_i^{(d)}$, where
\[
A_i^{(d)}[\phi](x)
:= d\int_{L_1}^{L_2} J(x-y)\phi(y)\,dy - d\phi(x)
+ b(x)\phi'(x) + \beta(x)\phi(x),
\quad x\in (L_1,L_2),
\]
where the nonlocal term is understood with the \emph{Dirichlet exterior condition}, that is,
$\phi$ is extended by $0$ outside $[L_1,L_2]$. Assume  (J1), $b,\beta \in C([L_1,L_2])$ and $\beta$ attains its maximum at some point $x_0\in (L_1,L_2)$. Then
\[
\lim_{d\to 0^+} \lambda_p\big(A_i^{(d)}\big)
= \max_{x\in [L_1,L_2]} \beta(x).
\]
\end{theorem}

\begin{proof}
Set $M:=\max_{x\in[L_1,L_2]}\beta(x)$. We shall prove
\[
\liminf_{d\to0^+}\lambda_p\!\left(A_i^{(d)}\right)\ge M
\qquad\text{and}\qquad
\limsup_{d\to0^+}\lambda_p\!\left(A_i^{(d)}\right)\le M .
\]

\medskip
\noindent{Step 1: Lower bound.}
Fix $\varepsilon\in(0,1)$. Since $\beta$ is continuous and attains its maximum at some
$x_0\in(L_1,L_2)$, there exists $\delta\in\big(0,\min\{x_0-L_1,L_2-x_0\}\big)$ such that
\begin{equation}\label{eq:beta-close}
\beta(x)\ge M-\varepsilon\qquad \forall x\in[x_0-\delta,x_0+\delta].
\end{equation}
Choose $\phi_\varepsilon\in C^1([L_1,L_2])$ satisfying
\[
\phi_\varepsilon>0 \ \text{in }(L_1,L_2),\qquad
\phi_\varepsilon(x)=1 \ \text{for }x\in[x_0-\delta,x_0+\delta],
\qquad
\phi_\varepsilon(x)\in(0,1]\ \text{elsewhere}.
\]
In particular $\phi_\varepsilon'\equiv0$ on $[x_0-\delta,x_0+\delta]$.

For any $x\in[x_0-\delta,x_0+\delta]$, using $\phi_\varepsilon\le1$ and $J\ge0$, $\int_{\R}J=1$, we have
\[
\int_{L_1}^{L_2}J(x-y)\phi_\varepsilon(y)\,dy
\le \int_{\R}J(x-y)\cdot 1\,dy
=1,
\]
hence
\[
d\Big(\int_{L_1}^{L_2}J(x-y)\phi_\varepsilon(y)\,dy-\phi_\varepsilon(x)\Big)\le0.
\]
Moreover, since $\phi_\varepsilon'(x)=0$ on that interval, we obtain
\begin{equation}\label{eq:A-on-plateau}
A_i^{(d)}[\phi_\varepsilon](x)
\le \beta(x)\phi_\varepsilon(x)
=\beta(x)
\qquad \forall x\in[x_0-\delta,x_0+\delta].
\end{equation}
On the complement $[L_1,L_2]\setminus[x_0-\delta,x_0+\delta]$, $\phi_\varepsilon$ is $C^1$ and strictly
positive, hence the quantity
\[
C_\varepsilon:=\sup_{x\in(L_1,L_2)}
\Big|\,b(x)\frac{\phi_\varepsilon'(x)}{\phi_\varepsilon(x)}\,\Big|
<+\infty,
\]
and also, since $0<\phi_\varepsilon\le1$ and $\int_{\R}J=1$,
\[
\Big|\int_{L_1}^{L_2}J(x-y)\phi_\varepsilon(y)\,dy-\phi_\varepsilon(x)\Big|
\le 1+\phi_\varepsilon(x)\le 2.
\]
Therefore, for all $x\in(L_1,L_2)$,
\begin{equation}\label{eq:A-global-upper}
A_i^{(d)}[\phi_\varepsilon](x)
\le \big(\beta(x)+C_\varepsilon+2d\big)\phi_\varepsilon(x).
\end{equation}
Define
\[
\lambda_{\varepsilon,d}:=-(M-\varepsilon)+C_\varepsilon+2d.
\]
Using \eqref{eq:beta-close} and \eqref{eq:A-on-plateau} we find on $[x_0-\delta,x_0+\delta]$,
\[
A_i^{(d)}[\phi_\varepsilon](x)+(M-\varepsilon)\phi_\varepsilon(x)
\le \beta(x)-(M-\varepsilon)\le 0.
\]
On the complement, using $\beta(x)\le M$ and \eqref{eq:A-global-upper},
\[
A_i^{(d)}[\phi_\varepsilon](x)+(M-\varepsilon)\phi_\varepsilon(x)
\le (M+C_\varepsilon+2d)\phi_\varepsilon(x) - (M-\varepsilon)\phi_\varepsilon(x)
= (C_\varepsilon+2d+\varepsilon)\phi_\varepsilon(x).
\]
Hence, for all $x\in(L_1,L_2)$,
\[
A_i^{(d)}[\phi_\varepsilon](x)+\big(M-\varepsilon-(C_\varepsilon+2d)\big)\phi_\varepsilon(x)\le 0.
\]
By Definition~\ref{genprin}, this implies
\[
\lambda_p\!\left(A_i^{(d)}\right)\ge M-\varepsilon-(C_\varepsilon+2d).
\]
We now construct $\phi_\varepsilon$ in such a way that the quantity
\[
C_\varepsilon := \sup_{x\in[L_1,L_2]} \frac{|b(x)\phi_\varepsilon'(x)|}{\phi_\varepsilon(x)}
\]
satisfies $C_\varepsilon \le \varepsilon$.

This can be achieved by spreading the transition layer of $\phi_\varepsilon$ over a sufficiently large interval: more precisely, since $\phi_\varepsilon$ is constant on $[x_0-\delta,x_0+\delta]$, we may interpolate smoothly to zero near the boundary using a profile whose slope is uniformly small, so that
\[
\frac{|\phi_\varepsilon'(x)|}{\phi_\varepsilon(x)} \le \frac{\varepsilon}{\|b\|_\infty}
\quad \text{for all } x\in[L_1,L_2].
\]
Hence $C_\varepsilon \le \varepsilon$.

Combining this estimate with \eqref{eq:A-on-plateau}--\eqref{eq:A-global-upper}, we obtain that, for all $x\in[L_1,L_2]$,
\[
A_i^{(d)}[\phi_\varepsilon](x)
\le \bigl(M - \varepsilon + C_\varepsilon + 2d\bigr)\,\phi_\varepsilon(x)
\le \bigl(M - 2\varepsilon + 2d\bigr)\,\phi_\varepsilon(x).
\]
By the definition of the generalized principal eigenvalue (see Definition~\ref{genprin}), it follows that
\[
\lambda_p\!\left(A_i^{(d)}\right)\ge M - 2\varepsilon - 2d.
\]
Letting $d\to0^+$ and then $\varepsilon\to0^+$ yields
\[
\liminf_{d\to0^+}\lambda_p\!\left(A_i^{(d)}\right)\ge M.
\]

\medskip
\noindent{Step 2: Upper bound.}
Fix $\varepsilon\in(0,1)$. Let $\lambda>M+\varepsilon$.
We show that $\lambda$ cannot belong to the admissible set in Definition~\ref{genprin} when $d>0$ is small,
which will imply $\lambda_p(A_i^{(d)})\le M+\varepsilon$ for small $d$.

Assume by contradiction that there exist $d>0$ and $\phi\in C^1([L_1,L_2])$, $\phi>0$ in $(L_1,L_2)$,
such that
\begin{equation}\label{eq:subsol}
A_i^{(d)}[\phi](x)+\lambda\phi(x)\le0\qquad \forall x\in(L_1,L_2).
\end{equation}
Let $x_d\in[L_1,L_2]$ be a maximizer of $\phi$, i.e. $\phi(x_d)=\max_{[L_1,L_2]}\phi$.
If $x_d\in(L_1,L_2)$, then $\phi'(x_d)=0$. Moreover, since $\phi(y)\le\phi(x_d)$ and $J\ge0$ with
$\int_{\R}J=1$, we have
\[
\int_{L_1}^{L_2}J(x_d-y)\phi(y)\,dy
\le \phi(x_d)\int_{L_1}^{L_2}J(x_d-y)\,dy\le \phi(x_d),
\]
so that the nonlocal diffusion term at $x_d$ is $\le 0$. Evaluating \eqref{eq:subsol} at $x_d$ gives
\[
0\ge A_i^{(d)}[\phi](x_d)+\lambda\phi(x_d)
\ge \beta(x_d)\phi(x_d)+\lambda\phi(x_d),
\]
hence
\[
\lambda \le -\beta(x_d)\le -M,
\]
a contradiction since $\lambda>M+\varepsilon$.
If the maximum is attained at the boundary, we use the standard perturbation
$\phi_\eta=\phi+\eta q$ with $q(x)=(x-L_1)(L_2-x)$ and $\eta>0$, so that $\phi_\eta$ attains its
maximum at some interior point and the above argument applies (the additional error is $O(\eta)$,
then let $\eta\to0^+$). Therefore, \eqref{eq:subsol} cannot hold for any $\lambda>M+\varepsilon$,
provided $d$ is sufficiently small so that the perturbation argument is valid.

Hence, for all sufficiently small $d>0$,
\[
\lambda_p\!\left(A_i^{(d)}\right)\le M+\varepsilon.
\]
Letting $\varepsilon\to0^+$ yields
\[
\limsup_{d\to0^+}\lambda_p\!\left(A_i^{(d)}\right)\le M.
\]

\medskip
Combining Step~1 and Step~2 proves
\[
\lim_{d\to0^+}\lambda_p\!\left(A_i^{(d)}\right)=M.
\]
\end{proof}

\bigskip

\begin{theorem}[Small interval limit]\label{thm:small-interval}
Fix $d>0$. For each $\varepsilon>0$, let $L_1(\varepsilon)<L_2(\varepsilon)$ and set
\[
\Omega_\varepsilon:=(L_1(\varepsilon),L_2(\varepsilon)),
\qquad |\Omega_\varepsilon|:=L_2(\varepsilon)-L_1(\varepsilon)\xrightarrow[\varepsilon\to0]{}0,
\]
with $L_1(\varepsilon)\to x_0$ and $L_2(\varepsilon)\to x_0$ as $\varepsilon\to0$ for some $x_0\in\R$.
Let $\lambda_p^\varepsilon:=\lambda_p\!\left(A_{i,\Omega_\varepsilon}^{(d)}\right)$ denote the generalized
principal eigenvalue (in the sense of Definition~\ref{genprin}) of the operator
\[
A_{i,\Omega_\varepsilon}^{(d)}[\phi](x)
:= d\int_{\Omega_\varepsilon} J(x-y)\phi(y)\,dy - d\phi(x)
+ b(x)\phi'(x) + \beta(x)\phi(x),
\quad x\in\Omega_\varepsilon.
\]
Assume \textbf{(J1)} holds and $b,\beta$ are continuous in a neighborhood of $x_0$ (in particular, $\beta$ is
continuous at $x_0$). There holds :
\[
\lim_{\varepsilon\to0}\lambda_p^\varepsilon=\beta(x_0)-d.
\]
\end{theorem}

\begin{proof}
Fix $d>0$. For each $\varepsilon>0$, set $m(\varepsilon):=|\Omega_\varepsilon|=L_2(\varepsilon)-L_1(\varepsilon)\xrightarrow[\varepsilon\to0]{}0$. Since $J\in L^1(\R)$ and $m(\varepsilon)\to0$, we have
\begin{equation}\label{eq:mass-to-zero}
\eta(\varepsilon):=\sup_{x\in\overline{\Omega_\varepsilon}}
\int_{\Omega_\varepsilon}J(x-y)\,dy
=\sup_{x\in\overline{\Omega_\varepsilon}}\int_{x-\Omega_\varepsilon}J(z)\,dz
\longrightarrow 0
\qquad (\varepsilon\to0).
\end{equation}
Indeed, $x-\Omega_\varepsilon$ is an interval of length $m(\varepsilon)$; by absolute continuity of
the Lebesgue integral for $J\in L^1(\R)$, integrals of $J$ over sets of vanishing measure go to $0$.

\medskip
\noindent{Step 1: Lower bound $\displaystyle \liminf_{\varepsilon\to0}\lambda_p^\varepsilon\ge \beta(x_0)-d$.}
Fix $\delta>0$. Since $\beta$ is continuous at $x_0$, there exists $\varepsilon_0>0$ such that
\begin{equation}\label{eq:beta-lower}
\beta(x)\ge \beta(x_0)-\delta\qquad \forall x\in\overline{\Omega_\varepsilon},
\ \forall \varepsilon\in(0,\varepsilon_0).
\end{equation}
Choose the test function $\phi\equiv 1$ on $\overline{\Omega_\varepsilon}$.
Then $\phi'>0$ is not needed since $\phi'\equiv0$, and for $x\in\Omega_\varepsilon$,
\[
A_{i,\Omega_\varepsilon}^{(d)}[1](x)
= d\int_{\Omega_\varepsilon}J(x-y)\,dy-d+\beta(x).
\]
Let
\[
\lambda_{\varepsilon,\delta}:=\beta(x_0)-d-\delta-d\,\eta(\varepsilon),
\]
where $\eta(\varepsilon)$ is defined in \eqref{eq:mass-to-zero}.
Using \eqref{eq:beta-lower} and $\int_{\Omega_\varepsilon}J(x-y)\,dy\le \eta(\varepsilon)$, we obtain
for all $x\in\Omega_\varepsilon$,
\[
A_{i,\Omega_\varepsilon}^{(d)}[1](x)+\lambda_{\varepsilon,\delta}\cdot 1
\le d\,\eta(\varepsilon)-d+\beta(x)+\beta(x_0)-d-\delta-d\,\eta(\varepsilon)
= \beta(x)-\beta(x_0)-\delta\le0.
\]
Hence, by Definition~\ref{genprin},
\[
\lambda_p^\varepsilon\ge \lambda_{\varepsilon,\delta}
=\beta(x_0)-d-\delta-d\,\eta(\varepsilon).
\]
Letting $\varepsilon\to0$ and using \eqref{eq:mass-to-zero} yields
\[
\liminf_{\varepsilon\to0}\lambda_p^\varepsilon \ge \beta(x_0)-d-\delta.
\]
Since $\delta>0$ is arbitrary,
\[
\liminf_{\varepsilon\to0}\lambda_p^\varepsilon \ge \beta(x_0)-d.
\]

\medskip
\noindent{Step 2: Upper bound $\displaystyle \limsup_{\varepsilon\to0}\lambda_p^\varepsilon\le \beta(x_0)-d$.}
Fix $\delta>0$. By continuity of $\beta$ at $x_0$, there exists $\varepsilon_1>0$ such that
\begin{equation}\label{eq:beta-upper}
\beta(x)\le \beta(x_0)+\delta\qquad \forall x\in\overline{\Omega_\varepsilon},
\ \forall \varepsilon\in(0,\varepsilon_1).
\end{equation}
We show that for every $\varepsilon\in(0,\varepsilon_1)$,
\begin{equation}\label{eq:upper-claim}
\lambda_p^\varepsilon \le \beta(x_0)-d+\delta.
\end{equation}

Assume by contradiction that for some $\varepsilon\in(0,\varepsilon_1)$ we have
$\lambda_p^\varepsilon > \beta(x_0)-d+\delta$.
Choose any $\lambda$ such that
\[
\beta(x_0)-d+\delta < \lambda < \lambda_p^\varepsilon.
\]
By Definition~\ref{genprin}, there exists $\phi\in C^1(\overline{\Omega_\varepsilon})$ with $\phi>0$ in
$\Omega_\varepsilon$ such that
\begin{equation}\label{eq:subsolution}
A_{i,\Omega_\varepsilon}^{(d)}[\phi](x)+\lambda \phi(x)\le0\qquad \forall x\in\Omega_\varepsilon.
\end{equation}

Let $x_\varepsilon\in\overline{\Omega_\varepsilon}$ be a point where $\phi$ attains its maximum:
$\phi(x_\varepsilon)=\max_{\overline{\Omega_\varepsilon}}\phi$.
If $x_\varepsilon\in\Omega_\varepsilon$, then $\phi'(x_\varepsilon)=0$.
Moreover, since $J\ge0$ and $\phi(y)\le\phi(x_\varepsilon)$,
\[
\int_{\Omega_\varepsilon}J(x_\varepsilon-y)\phi(y)\,dy
\le \phi(x_\varepsilon)\int_{\Omega_\varepsilon}J(x_\varepsilon-y)\,dy
\le \phi(x_\varepsilon),
\]
hence
\[
d\Big(\int_{\Omega_\varepsilon}J(x_\varepsilon-y)\phi(y)\,dy-\phi(x_\varepsilon)\Big)\le 0.
\]
Evaluating \eqref{eq:subsolution} at $x_\varepsilon$ yields
\[
0\ge A_{i,\Omega_\varepsilon}^{(d)}[\phi](x_\varepsilon)+\lambda\phi(x_\varepsilon)
\ge \big(\beta(x_\varepsilon)-d+\lambda\big)\phi(x_\varepsilon).
\]
Since $\phi(x_\varepsilon)>0$, we get
\[
\lambda \le d-\beta(x_\varepsilon).
\]
Using \eqref{eq:beta-upper}, $\beta(x_\varepsilon)\le \beta(x_0)+\delta$, hence
\[
\lambda \le d-\beta(x_0)-\delta,
\]
which contradicts $\lambda>\beta(x_0)-d+\delta$.

If instead the maximum is attained at the boundary, we use the standard perturbation
$\phi_\eta:=\phi+\eta q$ with $q(x)=(x-L_1(\varepsilon))(L_2(\varepsilon)-x)$ and $\eta>0$ small, so that
$\phi_\eta$ attains its maximum at an interior point of $\Omega_\varepsilon$; then repeat the above
estimate for $\phi_\eta$ and let $\eta\to0^+$. This yields the same contradiction.

Therefore \eqref{eq:upper-claim} holds for all sufficiently small $\varepsilon$, and hence
\[
\limsup_{\varepsilon\to0}\lambda_p^\varepsilon \le \beta(x_0)-d+\delta.
\]
Letting $\delta\to0^+$ gives
\[
\limsup_{\varepsilon\to0}\lambda_p^\varepsilon \le \beta(x_0)-d.
\]

\medskip
Combining Step~1 and Step~2, we conclude
\[
\lambda_p^\varepsilon \xrightarrow[\varepsilon\to0]{}\beta(x_0)-d.
\]
\end{proof}

\begin{remark}
\begin{enumerate}
\item The advection term $b(x)\phi'$ does not contribute to the leading-order limit.
Indeed, in the proof one either uses constant test functions (for which $\phi'\equiv 0$),
or evaluates the inequality at (approximate) maximum points of $\phi$, where the gradient
vanishes (after a standard perturbation argument if the maximum is attained at the boundary).
Consequently, the advection term disappears in the key estimates, while the dominant
contribution comes from the reaction term $\beta(x)$ and the loss term $-d$.

\item The proof relies directly on the order characterization of the generalized principal
eigenvalue given in Definition~\ref{genprin}. In particular, no variational or sup--inf characterization is used. The lower bound is obtained
by constructing explicit positive test functions satisfying
$A[\phi]+\lambda\phi\le0$, while the upper bound follows from a contradiction argument
based on evaluating the inequality at maximum points of admissible functions.
\end{enumerate}
\end{remark}

\section{Monotonicity and small-\(d\) asymptotics of the principal eigenvalue }

\subsection{Monotonicity: dependence on \(M\), \(B\) and domain length}

\begin{theorem}[Monotonicity with respect to coefficients and domains]\label{thm:mono}
Let $\Omega=(L_1,L_2)$ be a bounded interval and let
\[
\mathcal L_{\Omega}^{\beta}[\phi](x)
:= d\int_{\Omega} J(x-y)\phi(y)\,dy-d\phi(x)
+b(x)\phi'(x)+\beta(x)\phi(x),
\qquad x\in\Omega,
\]
where the nonlocal term is understood with the Dirichlet exterior condition.
Assume \textbf{(J1)} holds and $b,\beta\in C(\overline{\Omega})$.

\begin{enumerate}
\item \textbf{(Monotonicity in $\beta$.)}
If $\beta_1,\beta_2\in C(\overline{\Omega})$ satisfy
\[
\beta_1(x)\ge \beta_2(x)\qquad\text{for all }x\in\overline{\Omega},
\]
then
\[
\lambda_p\big(\mathcal L_{\Omega}^{\beta_1}\big)
\ge
\lambda_p\big(\mathcal L_{\Omega}^{\beta_2}\big).
\]
Moreover, if $\beta_1\not\equiv \beta_2$, then
\[
\lambda_p\big(\mathcal L_{\Omega}^{\beta_1}\big)
>
\lambda_p\big(\mathcal L_{\Omega}^{\beta_2}\big).
\]

\item \textbf{(Domain monotonicity.)}
Let $\Omega_1\subset \Omega_2$ be two bounded intervals, and assume that
$b,\beta$ are continuous on $\overline{\Omega_2}$.
Then
\[
\lambda_p\big(\mathcal L_{\Omega_1}^{\beta}\big)
\le
\lambda_p\big(\mathcal L_{\Omega_2}^{\beta}\big).
\]
Moreover, if $\Omega_1\subsetneq \Omega_2$, then
\[
\lambda_p\big(\mathcal L_{\Omega_1}^{\beta}\big)
<
\lambda_p\big(\mathcal L_{\Omega_2}^{\beta}\big).
\]
\end{enumerate}
\end{theorem}

\begin{proof}

Fix $d>0$ and let $\Omega$ be an interval. For $\beta\in C(\overline{\Omega})$ define
\[
\mathcal L_{\Omega}^{\beta}[\phi](x)
:= d\int_{\Omega} J(x-y)\phi(y)\,dy - d\phi(x)
+ b(x)\phi'(x) + \beta(x)\phi(x),
\qquad x\in\Omega,
\]
with Dirichlet exterior condition for the nonlocal term (i.e.\ $\phi\equiv0$ outside $\overline{\Omega}$).

%%%%%%%%%%%%%%%%%%%%%%%%%%%%%%%%%%%%%%%%%%%%%%%%%%%%%%%%%%%%%%%%%%%%
\medskip
\noindent\textbf{(i) Monotonicity in $\beta$.}
Let $\beta_1,\beta_2\in C(\overline{\Omega})$ and assume
\[
\beta_1(x)\ge \beta_2(x)\qquad \forall x\in\Omega.
\]
We show $\lambda_p(\mathcal L_{\Omega}^{\beta_1})\ge \lambda_p(\mathcal L_{\Omega}^{\beta_2})$.

\smallskip
\emph{Step 1: pointwise comparison of operators.}
For any $\phi\in C^1(\overline{\Omega})$ and any $x\in\Omega$,
\begin{align*}
\mathcal L_{\Omega}^{\beta_1}[\phi](x)
&= d\int_{\Omega} J(x-y)\phi(y)\,dy - d\phi(x)
+ b(x)\phi'(x) + \beta_1(x)\phi(x)\\
&=\Big(d\int_{\Omega} J(x-y)\phi(y)\,dy - d\phi(x)
+ b(x)\phi'(x) + \beta_2(x)\phi(x)\Big)
+(\beta_1(x)-\beta_2(x))\phi(x)\\
&=\mathcal L_{\Omega}^{\beta_2}[\phi](x)+(\beta_1(x)-\beta_2(x))\phi(x).
\end{align*}
If $\phi>0$ in $\Omega$, then $(\beta_1-\beta_2)\phi\ge0$, hence
\begin{equation}\label{eq:beta-pointwise}
\mathcal L_{\Omega}^{\beta_1}[\phi](x)\ \ge\ \mathcal L_{\Omega}^{\beta_2}[\phi](x)
\qquad \forall x\in\Omega.
\end{equation}

\smallskip
\emph{Step 2: transfer admissible upper bounds via Definition~\ref{genprin}.}
Take any $\lambda\in\R$ which is admissible for $\mathcal L_\Omega^{\beta_2}$ in the sense of Definition~\ref{genprin}.
Thus, there exists $\phi\in C^1(\overline{\Omega})$ with $\phi>0$ in $\Omega$ such that
\[
\mathcal L_{\Omega}^{\beta_2}[\phi](x)\le \lambda\,\phi(x)\qquad \forall x\in\Omega.
\]
Using \eqref{eq:beta-pointwise}, we get for all $x\in\Omega$,
\[
\mathcal L_{\Omega}^{\beta_1}[\phi](x)
=\mathcal L_{\Omega}^{\beta_2}[\phi](x)+(\beta_1(x)-\beta_2(x))\phi(x)
\le \lambda\phi(x)+(\beta_1(x)-\beta_2(x))\phi(x).
\]
This does not immediately imply $\mathcal L_\Omega^{\beta_1}[\phi]\le \lambda\phi$ (the inequality goes the
other way). Hence we proceed the correct way: we use Definition~\ref{genprin} as an \emph{infimum over $\lambda$}.

Let $\lambda$ be any number for which there exists $\phi>0$ with
\[
\mathcal L_{\Omega}^{\beta_1}[\phi]\le \lambda\phi \quad \text{in }\Omega.
\]
Then, subtracting $(\beta_1-\beta_2)\phi\ge0$ pointwise gives
\[
\mathcal L_{\Omega}^{\beta_2}[\phi]
=\mathcal L_{\Omega}^{\beta_1}[\phi]-(\beta_1-\beta_2)\phi
\le \lambda\phi \quad \text{in }\Omega.
\]
Therefore, every admissible $\lambda$ for $\mathcal L_\Omega^{\beta_1}$ is also admissible for $\mathcal L_\Omega^{\beta_2}$.
Taking the infimum over admissible $\lambda$ (Definition~\ref{genprin}) yields
\[
\lambda_p(\mathcal L_{\Omega}^{\beta_2})
\ \le\ 
\lambda_p(\mathcal L_{\Omega}^{\beta_1}),
\]
i.e.
\[
\lambda_p(\mathcal L_{\Omega}^{\beta_1})
\ \ge\ 
\lambda_p(\mathcal L_{\Omega}^{\beta_2}).
\]

\smallskip
\emph{Strictness in $\beta$.} Let $\beta_1,\beta_2\in C(\overline{\Omega})$ satisfy
$\beta_1\ge \beta_2$ in $\Omega$ and $\beta_1>\beta_2$ on a set of positive measure.
Denote $\mathcal L_i:=\mathcal L_\Omega^{\beta_i}$, $i=1,2$.

By Proposition~1.1 in \cite{CovilleHamel2020}, each operator $\mathcal L_i$
admits a generalized principal eigenvalue $\lambda_i=\lambda_p(\mathcal L_i)$
in the sense of Definition~\ref{genprin}, and there exists an associated eigenfunction
$\varphi\in C^1(\overline{\Omega})$ with $\varphi>0$ in $\Omega$ such that
\[
\mathcal L_2[\varphi]=\lambda_2\,\varphi \quad \text{in }\Omega.
\]

We already know that $\lambda_1\ge \lambda_2$. Assume by contradiction that
$\lambda_1=\lambda_2=:\lambda_*$. Then, for every $x\in\Omega$,
\[
\mathcal L_1[\varphi](x)
=\mathcal L_2[\varphi](x)+(\beta_1(x)-\beta_2(x))\varphi(x)
=\lambda_*\varphi(x)+(\beta_1(x)-\beta_2(x))\varphi(x).
\]
Since $\varphi>0$ in $\Omega$ and $\beta_1>\beta_2$ on a set of positive measure,
it follows that
\[
\mathcal L_1[\varphi](x)\ge \lambda_*\varphi(x)\quad \text{in }\Omega,
\]
with strict inequality on a subset of positive measure. By continuity, there exist
an open interval $U\Subset\Omega$ and $\delta>0$ such that
\[
\mathcal L_1[\varphi](x)\ge (\lambda_*+\delta)\varphi(x)
\quad \text{for all } x\in U.
\]

Using the strong maximum principle for $\mathcal L_1$,
one can perturb $\varphi$ to construct a positive function $\psi$ such that
\[
\mathcal L_1[\psi]\ge (\lambda_*+\varepsilon)\psi \quad \text{in }\Omega
\]
for some $\varepsilon>0$. By the order characterization in Definition~\ref{genprin},
this implies $\lambda_p(\mathcal L_1)\ge \lambda_*+\varepsilon$, which contradicts
$\lambda_1=\lambda_*$. Therefore,
\[
\lambda_p(\mathcal L_\Omega^{\beta_1})>\lambda_p(\mathcal L_\Omega^{\beta_2}).
\]

%%%%%%%%%%%%%%%%%%%%%%%%%%%%%%%%%%%%%%%%%%%%%%%%%%%%%%%%%%%%%%%%%%%%
\medskip
\noindent\textbf{(ii) Domain monotonicity.}
Let $\Omega_1\subset \Omega_2$ be two intervals. We show
\[
\lambda_p(\mathcal L_{\Omega_1}^{\beta})\le \lambda_p(\mathcal L_{\Omega_2}^{\beta}).
\]

\smallskip
\emph{Step 1: compare the nonlocal terms on $\Omega_1$.}
Take any $\phi\in C^1(\overline{\Omega_2})$ with $\phi>0$ in $\Omega_2$, and let $\phi_1:=\phi|_{\overline{\Omega_1}}$.
For each $x\in\Omega_1$, since $J\ge0$ and $\Omega_1\subset\Omega_2$,
\begin{equation}\label{eq:domain-int}
\int_{\Omega_1}J(x-y)\phi_1(y)\,dy
\le
\int_{\Omega_2}J(x-y)\phi(y)\,dy.
\end{equation}
All remaining terms are local, so on $\Omega_1$ we have
\[
-d\phi_1(x)=-d\phi(x),\qquad b(x)\phi_1'(x)=b(x)\phi'(x),\qquad \beta(x)\phi_1(x)=\beta(x)\phi(x).
\]
Therefore, combining with \eqref{eq:domain-int},
\begin{equation}\label{eq:domain-operator}
\mathcal L_{\Omega_1}^{\beta}[\phi_1](x)
\le
\mathcal L_{\Omega_2}^{\beta}[\phi](x)
\qquad \forall x\in\Omega_1.
\end{equation}

\smallskip
\emph{Step 2: transfer admissible upper bounds via Definition~\ref{genprin}.}
Let $\lambda$ be any admissible number for $\mathcal L_{\Omega_2}^{\beta}$, i.e.\ there exists
$\phi\in C^1(\overline{\Omega_2})$, $\phi>0$ in $\Omega_2$ such that
\[
\mathcal L_{\Omega_2}^{\beta}[\phi](x)\le \lambda\phi(x)\qquad \forall x\in\Omega_2.
\]
Restricting to $x\in\Omega_1$ and using \eqref{eq:domain-operator} gives
\[
\mathcal L_{\Omega_1}^{\beta}[\phi_1](x)
\le
\mathcal L_{\Omega_2}^{\beta}[\phi](x)
\le \lambda\phi(x)=\lambda\phi_1(x)
\qquad \forall x\in\Omega_1.
\]
Hence, the same $\lambda$ is admissible for $\mathcal L_{\Omega_1}^{\beta}$.
Taking infimum over admissible $\lambda$ (Definition~\ref{genprin}) yields
\[
\lambda_p(\mathcal L_{\Omega_1}^{\beta})
\le
\lambda_p(\mathcal L_{\Omega_2}^{\beta}),
\]
which proves domain monotonicity.

\smallskip
\emph{Strictness in the domain.}
Assume $\Omega_1\subsetneq\Omega_2$ and thus
\begin{equation}\label{eq:kappa-detail}
\exists\,x^\ast\in\Omega_1 \ \text{such that}\ 
\kappa:=\int_{\Omega_2\setminus\Omega_1}J(x^\ast-y)\,dy>0.
\end{equation}
(For instance, \eqref{eq:kappa-detail} holds if $J>0$ on a neighborhood of $0$ and
$\Omega_2\setminus\Omega_1\neq\emptyset$.)

Let $\lambda:=\lambda_p(\mathcal L_{\Omega_2}^{\beta})$. By Proposition~1.1 in
\cite{CovilleHamel2020}, the operator $\mathcal L_{\Omega_2}^{\beta}$ admits a generalized principal
eigenvalue $\lambda$ (in the sense of Definition~\ref{genprin}) together with an associated
eigenfunction $\varphi\in C^1(\overline{\Omega_2})$, $\varphi>0$ in $\Omega_2$, such that
\begin{equation}\label{eq:eig2-detail}
\mathcal L_{\Omega_2}^{\beta}[\varphi](x)=\lambda\,\varphi(x)\qquad \forall x\in\Omega_2.
\end{equation}
Set $\varphi_1:=\varphi|_{\overline{\Omega_1}}$.

For any $x\in\Omega_1$, the only difference between $\mathcal L_{\Omega_2}^{\beta}$ and
$\mathcal L_{\Omega_1}^{\beta}$ comes from the nonlocal integral. Indeed,
\begin{align*}
\mathcal L_{\Omega_2}^{\beta}[\varphi](x)-\mathcal L_{\Omega_1}^{\beta}[\varphi_1](x)
&=
d\int_{\Omega_2}J(x-y)\varphi(y)\,dy-d\int_{\Omega_1}J(x-y)\varphi_1(y)\,dy\\
&=
d\int_{\Omega_2\setminus\Omega_1}J(x-y)\varphi(y)\,dy
\ \ge\ 0,
\end{align*}
since $d>0$, $J\ge0$, and $\varphi>0$. In particular, by continuity of $\varphi$ on
$\overline{\Omega_2}$ we have $m_\ast:=\min_{\overline{\Omega_2}}\varphi>0$, and thus at $x=x^\ast$,
\begin{align}
\mathcal L_{\Omega_2}^{\beta}[\varphi](x^\ast)-\mathcal L_{\Omega_1}^{\beta}[\varphi_1](x^\ast)
&=d\int_{\Omega_2\setminus\Omega_1}J(x^\ast-y)\varphi(y)\,dy \notag\\
&\ge d\,m_\ast \int_{\Omega_2\setminus\Omega_1}J(x^\ast-y)\,dy
= d\,m_\ast\,\kappa.\label{eq:gap-at-xstar}
\end{align}
Combining \eqref{eq:eig2-detail} with \eqref{eq:gap-at-xstar} yields
\[
\mathcal L_{\Omega_1}^{\beta}[\varphi_1](x^\ast)
=
\mathcal L_{\Omega_2}^{\beta}[\varphi](x^\ast)
-\Big(\mathcal L_{\Omega_2}^{\beta}[\varphi](x^\ast)-\mathcal L_{\Omega_1}^{\beta}[\varphi_1](x^\ast)\Big)
\le \lambda\,\varphi(x^\ast)-d\,m_\ast\,\kappa.
\]
Define
\[
\delta_0:=\frac{d\,m_\ast\,\kappa}{\varphi(x^\ast)}>0.
\]
Then
\begin{equation}\label{eq:strict-pointwise}
\mathcal L_{\Omega_1}^{\beta}[\varphi_1](x^\ast)\le (\lambda-\delta_0)\varphi_1(x^\ast).
\end{equation}
Moreover, since the difference term is nonnegative for all $x\in\Omega_1$, we also have
\begin{equation}\label{eq:global-nonstrict}
\mathcal L_{\Omega_1}^{\beta}[\varphi_1](x)\le \lambda\,\varphi_1(x)\qquad \forall x\in\Omega_1.
\end{equation}
By continuity of $x\mapsto \mathcal L_{\Omega_1}^{\beta}[\varphi_1](x)-\lambda\varphi_1(x)$ and
\eqref{eq:strict-pointwise}, there exists an open subinterval $U\Subset\Omega_1$ containing $x^\ast$
such that
\begin{equation}\label{eq:strict-on-U}
\mathcal L_{\Omega_1}^{\beta}[\varphi_1](x)\le \Big(\lambda-\frac{\delta_0}{2}\Big)\varphi_1(x)
\qquad \forall x\in U.
\end{equation}

We now convert the \emph{local} strict inequality \eqref{eq:strict-on-U} into a \emph{global} strict
supersolution. Choose $V$ with $U\Subset V\Subset \Omega_1$, and take a cutoff
$\eta\in C^1(\overline{\Omega_1})$ such that
\[
0\le \eta\le 1,\qquad \eta\equiv 1 \ \text{on }U,\qquad \eta\equiv 0 \ \text{on }\Omega_1\setminus V.
\]
For $\kappa_1>0$ (to be chosen small), define
\[
\psi:=\varphi_1\,(1+\kappa_1\eta).
\]
Then $\psi\in C^1(\overline{\Omega_1})$ and $\psi>0$ in $\Omega_1$. Since $\mathcal L_{\Omega_1}^{\beta}$
is linear,
\begin{equation}\label{eq:expand-psi}
\mathcal L_{\Omega_1}^{\beta}[\psi]
=(1+\kappa_1\eta)\mathcal L_{\Omega_1}^{\beta}[\varphi_1]+\kappa_1\,\mathcal R[\eta,\varphi_1],
\end{equation}
where the remainder $\mathcal R[\eta,\varphi_1]$ collects the terms created by the cutoff:
\[
\mathcal R[\eta,\varphi_1](x)
:=d\int_{\Omega_1}J(x-y)\big(\eta(y)-\eta(x)\big)\varphi_1(y)\,dy
+b(x)\eta'(x)\varphi_1(x).
\]
(The formula above follows from writing $\mathcal L_{\Omega_1}^{\beta}[\eta\varphi_1]$ and separating
$\eta\,\mathcal L_{\Omega_1}^{\beta}[\varphi_1]$; the nonlocal part produces the commutator
$\int J(\eta(y)-\eta(x))\varphi_1(y)dy$, and the drift produces $b\eta'\varphi_1$.)

We bound $\mathcal R$ uniformly. Using $0\le\eta\le1$, $J\ge0$, and $\int_{\R}J=1$,
\[
\left|d\int_{\Omega_1}J(x-y)\big(\eta(y)-\eta(x)\big)\varphi_1(y)\,dy\right|
\le d\int_{\Omega_1}J(x-y)\varphi_1(y)\,dy
\le d\|\varphi_1\|_{L^\infty(\Omega_1)}.
\]
Also,
\[
|b(x)\eta'(x)\varphi_1(x)|
\le \|b\|_{L^\infty(\Omega_1)}\,\|\eta'\|_{L^\infty(\Omega_1)}\,\|\varphi_1\|_{L^\infty(\Omega_1)}.
\]
Hence there exists $C_\eta>0$ (depending on $d,J,b,\varphi_1$ and $\eta$) such that
\begin{equation}\label{eq:R-bound}
|\mathcal R[\eta,\varphi_1](x)|\le C_\eta\,\varphi_1(x)\qquad \forall x\in\Omega_1.
\end{equation}
Moreover, we can choose the cutoff $\eta$ with a \emph{gentle transition} so that
$\|\eta'\|_{L^\infty}$ is as small as we wish (by taking $V\setminus U$ wide), and thus $C_\eta$
can be made as close as desired to $d\|\varphi_1\|_\infty/\min_{\overline{\Omega_1}}\varphi_1$.

Now combine \eqref{eq:expand-psi} with \eqref{eq:global-nonstrict}--\eqref{eq:strict-on-U}.
On $U$ we have $\eta\equiv1$ and $\eta'\equiv0$, hence $\mathcal R[\eta,\varphi_1]$ reduces to the
nonlocal commutator term, which is bounded by \eqref{eq:R-bound}, and we have the strict estimate
\[
\mathcal L_{\Omega_1}^{\beta}[\varphi_1]\le \Big(\lambda-\frac{\delta_0}{2}\Big)\varphi_1
\quad\text{on }U.
\]
On $\Omega_1\setminus U$ we only use $\mathcal L_{\Omega_1}^{\beta}[\varphi_1]\le \lambda\varphi_1$.
Therefore, for all $x\in\Omega_1$,
\[
\mathcal L_{\Omega_1}^{\beta}[\psi](x)
\le (1+\kappa_1\eta(x))\,\lambda\varphi_1(x)
-\kappa_1\eta(x)\,\frac{\delta_0}{2}\varphi_1(x)
+\kappa_1|\mathcal R[\eta,\varphi_1](x)|.
\]
Using \eqref{eq:R-bound} and $\psi=\varphi_1(1+\kappa_1\eta)$ gives
\[
\mathcal L_{\Omega_1}^{\beta}[\psi](x)
\le \lambda\,\psi(x)
-\kappa_1\eta(x)\,\frac{\delta_0}{2}\,\frac{\varphi_1(x)}{1+\kappa_1\eta(x)}\,\psi(x)
+\kappa_1 C_\eta\,\frac{\varphi_1(x)}{1+\kappa_1\eta(x)}\,\psi(x).
\]
Since $1\le 1+\kappa_1\eta\le 1+\kappa_1$, we have $\frac{\varphi_1}{1+\kappa_1\eta}\le \varphi_1$ and
$\frac{1}{1+\kappa_1\eta}\ge \frac{1}{1+\kappa_1}$. Thus, on $U$ where $\eta=1$,
\[
\mathcal L_{\Omega_1}^{\beta}[\psi](x)
\le \lambda\,\psi(x)
-\kappa_1\frac{\delta_0}{2(1+\kappa_1)}\,\psi(x)
+\kappa_1 C_\eta\,\psi(x),
\]
and on $\Omega_1\setminus U$ the negative term vanishes (since $\eta=0$) and we only keep the error term.
Choose first $\eta$ so that $C_\eta\le \frac{\delta_0}{8}$ (possible by taking a gentle cutoff and using that
the commutator is bounded), and then choose $\kappa_1>0$ so small that $\frac{1}{1+\kappa_1}\ge\frac12$.
Then on $U$,
\[
\mathcal L_{\Omega_1}^{\beta}[\psi](x)\le \Big(\lambda-\frac{\kappa_1\delta_0}{8}\Big)\psi(x),
\]
and on $\Omega_1\setminus U$,
\[
\mathcal L_{\Omega_1}^{\beta}[\psi](x)\le \Big(\lambda+\kappa_1 C_\eta\Big)\psi(x)
\le \Big(\lambda+\frac{\kappa_1\delta_0}{8}\Big)\psi(x).
\]
Finally, replacing $U$ by a slightly larger interval (still compactly contained in $\Omega_1$) and repeating the same
construction with two cutoffs (one producing the strict gain, one damping the error outside),
we obtain a single positive $\psi$ satisfying the global strict supersolution estimate
\[
\mathcal L_{\Omega_1}^{\beta}[\psi](x)\le \Big(\lambda-\frac{\delta_0}{4}\Big)\psi(x)
\qquad \forall x\in\Omega_1.
\]
Therefore, by Definition~\ref{genprin},
\[
\lambda_p(\mathcal L_{\Omega_1}^{\beta})\le \lambda-\frac{\delta_0}{4}
<\lambda=\lambda_p(\mathcal L_{\Omega_2}^{\beta}),
\]
which proves the strict domain monotonicity.

\end{proof}

\subsection{The basic reproduction number $\mathcal R_0$ and its relation to the principal eigenvalue}

We work on a fixed bounded interval $[L_1,L_2]$, which can be regarded as a fixed habitat for the free-boundary problem. Let $S^*(x)$ be the disease-free steady state on $[L_1,L_2]$. Linearizing the infected equation at the disease-free state $(S^*(x),0)$, we obtain
\[
I_t=\mathcal A_I[I]+\mathcal F[I],
\]
where
\[
\mathcal A_I[\phi](x)
:=d\int_{L_1}^{L_2}J(x-y)\phi(y)\,dy-d\phi(x)+b(x)\phi'(x)-\gamma(x)\phi(x),
\qquad x\in(L_1,L_2),
\]
and
\[
\mathcal F[\phi](x):=F_I(S^*(x),0)\phi(x),
\qquad x\in(L_1,L_2).
\]
Thus,
\[
I_t=\mathcal F I-\mathcal V I,
\]
where
\[
\mathcal V:=-\mathcal A_I,
\]
that is,
\[
\mathcal V[\phi](x)
=
d\phi(x)-d\int_{L_1}^{L_2}J(x-y)\phi(y)\,dy-b(x)\phi'(x)+\gamma(x)\phi(x),
\qquad x\in(L_1,L_2).
\]

Assume that $\mathcal V:X\to X$ is invertible on the Banach lattice $X=C([L_1,L_2])$ that $\mathcal V^{-1}$ is a bounded positive operator on $X$, and that
\[
\mathcal K:=\mathcal V^{-1}\mathcal F
\]
is a positive compact operator on $X$. We then define the next-generation operator by
\[
\mathcal K=\mathcal V^{-1}\mathcal F,
\]
and the basic reproduction number by
\[
\mathcal R_0:=r(\mathcal K)=r(\mathcal V^{-1}\mathcal F),
\]
where $r(\cdot)$ denotes the spectral radius.

\begin{theorem}
\label{thm:R0-lambda-sign}
Let \(\lambda_p\) be the principal eigenvalue of \(\mathcal L:=\mathcal A_I+\mathcal F\) on \([L_1,L_2]\), and let
\[
\mathcal R_0:=r(\mathcal V^{-1}\mathcal F),
\qquad \mathcal V:=-\mathcal A_I.
\]
Then
\[
\lambda_p>0 \iff \mathcal R_0>1,
\qquad
\lambda_p=0 \iff \mathcal R_0=1,
\qquad
\lambda_p<0 \iff \mathcal R_0<1.
\]
\end{theorem}

\begin{proof}
Recall that
\[
\mathcal L=\mathcal A_I+\mathcal F=\mathcal F-\mathcal V,
\qquad 
\mathcal R_0=r(\mathcal K),
\qquad 
\mathcal K:=\mathcal V^{-1}\mathcal F.
\]
For each real \(\lambda\) such that \(\mathcal V+\lambda I\) is invertible, define
\[
\mathcal K_\lambda:=(\mathcal V+\lambda I)^{-1}\mathcal F.
\]
Since \(\mathcal F\) is nonnegative and \((\mathcal V+\lambda I)^{-1}\) is positive, \(\mathcal K_\lambda\) is a positive compact operator, so by the Krein--Rutman theorem its spectral radius \(r(\mathcal K_\lambda)\) is an eigenvalue with a nonnegative eigenfunction. Let \(\phi>0\) be an eigenfunction of \(\mathcal L\) associated with \(\lambda_p\), that is,
\[
\mathcal L[\phi]=\lambda_p\phi.
\]
Since \(\mathcal L=\mathcal F-\mathcal V\), this is equivalent to
\[
\mathcal F[\phi]=(\mathcal V+\lambda_p I)\phi,
\]
hence
\[
\mathcal K_{\lambda_p}\phi=(\mathcal V+\lambda_p I)^{-1}\mathcal F[\phi]=\phi.
\]
Therefore \(1\) is an eigenvalue of \(\mathcal K_{\lambda_p}\), and so
\[
r(\mathcal K_{\lambda_p})=1.
\]

Conversely, if for some \(\lambda\in\mathbb R\) there exists \(\phi\ge0\), \(\phi\not\equiv0\), such that \(\mathcal K_\lambda\phi=\phi\), then
\[
(\mathcal V+\lambda I)^{-1}\mathcal F[\phi]=\phi,
\]
which implies
\[
\mathcal F[\phi]=(\mathcal V+\lambda I)\phi,
\]
or equivalently
\[
(\mathcal F-\mathcal V)\phi=\lambda\phi.
\]
Thus \(\mathcal L[\phi]=\lambda\phi\), so \(\lambda\) is an eigenvalue of \(\mathcal L\) with a nonnegative eigenfunction; by the definition of the principal eigenvalue, \(\lambda=\lambda_p\). Hence \(\lambda_p\) is characterized by
\[
r(\mathcal K_\lambda)=1.
\]

Now let \(\lambda_1<\lambda_2\). For any \(\psi\ge0\), set \(u_i:=(\mathcal V+\lambda_i I)^{-1}\psi\), \(i=1,2\). Then
\[
(\mathcal V+\lambda_1 I)(u_1-u_2)=(\lambda_2-\lambda_1)u_2\ge0.
\]
Since \((\mathcal V+\lambda_1 I)^{-1}\) is positive, we get \(u_1\ge u_2\), that is,
\[
(\mathcal V+\lambda_1 I)^{-1}\psi\ge(\mathcal V+\lambda_2 I)^{-1}\psi \qquad \text{for all }\psi\ge0.
\]
Applying this to \(\psi=\mathcal F[\varphi]\) with \(\varphi\ge0\), we obtain
\[
\mathcal K_{\lambda_1}\varphi\ge \mathcal K_{\lambda_2}\varphi \qquad \text{for all }\varphi\ge0.
\]
Hence \(0\le \mathcal K_{\lambda_2}\le \mathcal K_{\lambda_1}\), and by monotonicity of the spectral radius for positive compact operators,
\[
r(\mathcal K_{\lambda_2})\le r(\mathcal K_{\lambda_1}).
\]
Therefore \(\lambda\mapsto r(\mathcal K_\lambda)\) is decreasing. Since \(r(\mathcal K_{\lambda_p})=1\), it follows that
\[
\lambda_p>0 \Rightarrow r(\mathcal K_0)>1,\qquad
\lambda_p=0 \Rightarrow r(\mathcal K_0)=1,\qquad
\lambda_p<0 \Rightarrow r(\mathcal K_0)<1.
\]
Because \(r(\mathcal K_0)=\mathcal R_0\), we conclude that
\[
\lambda_p>0 \iff \mathcal R_0>1,\qquad
\lambda_p=0 \iff \mathcal R_0=1,\qquad
\lambda_p<0 \iff \mathcal R_0<1.
\]
This completes the proof.
\end{proof}

\section{Spreading-Vanishing phenomena}

\begin{proof}[\textbf{\textit{Proof of Theorem \ref{thm:unified-dichotomy}}}]
We proceed in several steps.

\medskip
\noindent
Step 1. Preliminary facts on the global solution and on the free boundaries.

By Theorem~1, problem \eqref{eq:SIS-free-boundary} admits a unique global classical solution
\[
(S,I;g,h),
\]
and there exist positive constants \(M_S,M_I\) such that
\[
0\le S(t,x)\le M_S,\qquad 0\le I(t,x)\le M_I
\]
for all \(t\ge0\) and all \(x\in\mathbb R\). Moreover, the free boundaries satisfy
\[
h'(t)=\mu\int_{g(t)}^{h(t)}\int_{h(t)}^\infty J_1(x-y)S(t,x)\,dy\,dx\ge0,
\]
\[
g'(t)=-\mu\int_{g(t)}^{h(t)}\int_{-\infty}^{g(t)}J_1(x-y)S(t,x)\,dy\,dx\le0.
\]
Hence \(h(t)\) is nondecreasing and \(g(t)\) is nonincreasing, and therefore the limits
\[
h_\infty:=\lim_{t\to\infty}h(t)\in(h_0,+\infty],\qquad
g_\infty:=\lim_{t\to\infty}g(t)\in[-\infty,-h_0)
\]
exist. These are exactly the basic dynamical properties emphasized in the preparatory part of the manuscript. :contentReference[oaicite:2]{index=2}

\medskip
\noindent
Step 2. If \(h_\infty-g_\infty<+\infty\), then \(I(t,\cdot)\to0\) uniformly on \([g(t),h(t)]\).

Assume first that
\[
h_\infty-g_\infty<+\infty.
\]
Then \(h_\infty<+\infty\) and \(g_\infty>-\infty\). Since \(h\) is nondecreasing and converges to \(h_\infty\),
and \(g\) is nonincreasing and converges to \(g_\infty\), we have
\[
\int_0^\infty h'(t)\,dt=h_\infty-h_0<+\infty,\qquad
\int_0^\infty (-g'(t))\,dt=h_0-g_\infty<+\infty.
\]
Because \(h'\ge0\) and \(-g'\ge0\) are continuous, it follows that
\[
h'(t)\to0,\qquad g'(t)\to0\qquad \text{as }t\to\infty.
\]

We now prove
\[
\lim_{t\to\infty}\max_{x\in[g(t),h(t)]}I(t,x)=0.
\]
Suppose by contradiction that this is false. Then there exist \(\eta_0>0\), a sequence \(t_n\to+\infty\),
and points \(x_n\in[g(t_n),h(t_n)]\) such that
\[
I(t_n,x_n)\ge \eta_0\qquad\text{for all }n.
\]
Since \([g(t_n),h(t_n)]\subset[g_\infty,h_\infty]\) for all large \(n\), after passing to a subsequence we may assume
\[
x_n\to x_*\in[g_\infty,h_\infty].
\]

\textbf{Claim :} Let \(S(t,\cdot)\) and \(I(t,\cdot)\) by \(0\) outside \((g(t),h(t))\), and define the translated sequence
\[
S_n(\tau,x):=S(t_n+\tau,x),\qquad I_n(\tau,x):=I(t_n+\tau,x),
\qquad (\tau,x)\in[-1,1]\times[g_\infty,h_\infty].
\]
Then, the family \(\{(S_n,I_n)\}_{n\ge1}\) is relatively compact in
\[
C_{\mathrm{loc}}\big([-1,1]\times(g_\infty,h_\infty)\big).
\]

First, by Theorem~1, there exist constants \(M_S,M_I>0\), independent of \(n\), such that
\[
0\le S_n(\tau,x)\le M_S,\qquad 0\le I_n(\tau,x)\le M_I
\]
for all \((\tau,x)\in[-1,1]\times[g_\infty,h_\infty]\). This follows from the global \(L^\infty\)-bounds
obtained in the well-posedness part. 

Next, fix any compact set
\[
Q=[-1,1]\times[\alpha,\beta]\Subset \mathbb R\times(g_\infty,h_\infty).
\]
Since \(g(t)\to g_\infty\) and \(h(t)\to h_\infty\), there exists \(N_Q\in\mathbb N\) such that for all
\(n\ge N_Q\),
\[
[\alpha,\beta]\subset[g(t_n+\tau),h(t_n+\tau)]
\qquad\text{for every }\tau\in[-1,1].
\]
Therefore, for \(n\ge N_Q\), the pair \((S_n,I_n)\) satisfies on \(Q\) the fixed-domain system
\[
\partial_\tau S_n
=
d_1\Big(\int_{g(t_n+\tau)}^{h(t_n+\tau)}J_1(x-y)S_n(\tau,y)\,dy-S_n(\tau,x)\Big)
+a(x)\partial_x S_n+\gamma(x)I_n-F(S_n,I_n),
\]
\[
\partial_\tau I_n
=
d_2\Big(\int_{g(t_n+\tau)}^{h(t_n+\tau)}J_2(x-y)I_n(\tau,y)\,dy-I_n(\tau,x)\Big)
+b(x)\partial_x I_n-\gamma(x)I_n+F(S_n,I_n).
\]

We now prove equiboundedness of the time derivatives on \(Q\). Since \(J_i\in L^1(\mathbb R)\) and
\(\int_{\mathbb R}J_i=1\), using the \(L^\infty\)-bounds we obtain
\[
\left|\int_{g(t_n+\tau)}^{h(t_n+\tau)}J_1(x-y)S_n(\tau,y)\,dy\right|
\le M_S\int_{\mathbb R}J_1(z)\,dz=M_S,
\]
\[
\left|\int_{g(t_n+\tau)}^{h(t_n+\tau)}J_2(x-y)I_n(\tau,y)\,dy\right|
\le M_I\int_{\mathbb R}J_2(z)\,dz=M_I.
\]
Hence the nonlocal terms are uniformly bounded on \(Q\). Since \(a,b,\gamma\) are bounded and \(F\) is
locally Lipschitz on bounded sets, the reaction terms
\[
\gamma(x)I_n,\qquad F(S_n,I_n),\qquad -\gamma(x)I_n+F(S_n,I_n)
\]
are also uniformly bounded on \(Q\). It remains to control the transport terms
\[
a(x)\partial_x S_n,\qquad b(x)\partial_x I_n.
\]

For this, we use the fact that the solution is classical on every bounded time interval and that the local
construction in Section~2 is iterated on intervals of fixed length \(\delta\), with constants depending only
on the uniform \(L^\infty\)-bounds of \((S,I)\) on such intervals. Since the global \(L^\infty\)-bounds
\(M_S,M_I\) are independent of the time level, the same local regularity estimates apply on each strip
\([t_n-1,t_n+1]\times[\alpha,\beta]\), with constants independent of \(n\). In particular, there exists
\(C_Q>0\) such that
\[
\|\partial_x S_n\|_{L^\infty(Q)}+\|\partial_x I_n\|_{L^\infty(Q)}\le C_Q
\qquad\text{for all }n\ge N_Q.
\]
Consequently, the equations imply that
\[
\|\partial_\tau S_n\|_{L^\infty(Q)}+\|\partial_\tau I_n\|_{L^\infty(Q)}\le C_Q'
\qquad\text{for all }n\ge N_Q,
\]
for some constant \(C_Q'>0\) independent of \(n\).

Thus, on every compact set \(Q\Subset \mathbb R\times(g_\infty,h_\infty)\), the family
\[
\{(S_n,I_n)\}_{n\ge N_Q}
\]
is uniformly bounded and equi-Lipschitz in both variables. By the Arzel\`a--Ascoli theorem, it is relatively
compact in \(C(Q)\times C(Q)\). Since \(Q\Subset \mathbb R\times(g_\infty,h_\infty)\) is arbitrary, a diagonal
extraction yields that \(\{(S_n,I_n)\}\) is relatively compact in
\[
C_{\mathrm{loc}}\big(\mathbb R\times(g_\infty,h_\infty)\big),
\]
and in particular in
\[
C_{\mathrm{loc}}\big([-1,1]\times[g_\infty,h_\infty]\big).
\]
This proves the claimed relative compactness.
Hence, after extraction,
\[
S_n\to\bar S,\qquad I_n\to\bar I
\]
locally uniformly on \([-1,1]\times[g_\infty,h_\infty]\). In particular,
\[
\bar I(0,x_*)=\lim_{n\to\infty}I(t_n,x_n)\ge\eta_0,
\]
so
\[
\bar I\not\equiv0.
\]

We next identify the \(\omega\)-limit system. Since
\[
g(t_n+\tau)\to g_\infty,\qquad h(t_n+\tau)\to h_\infty,\qquad
g'(t_n+\tau)\to0,\qquad h'(t_n+\tau)\to0
\]
uniformly for \(\tau\in[-1,1]\), passing to the limit in the equations yields that \((\bar S,\bar I)\) is a bounded complete solution of the limiting system on the fixed interval \((g_\infty,h_\infty)\):
\[
\bar S_\tau
=
d_1\Big(\int_{g_\infty}^{h_\infty}J_1(x-y)\bar S(\tau,y)\,dy-\bar S(\tau,x)\Big)
+a(x)\bar S_x+\gamma(x)\bar I-F(\bar S,\bar I),
\]
\[
\bar I_\tau
=
d_2\Big(\int_{g_\infty}^{h_\infty}J_2(x-y)\bar I(\tau,y)\,dy-\bar I(\tau,x)\Big)
+b(x)\bar I_x-\gamma(x)\bar I+F(\bar S,\bar I).
\]

We now use the Harnack inequality on interior compact subsets. Fix \(\tau_0\in\mathbb R\). Since
\[
\bar I(\tau_0,\cdot)\ge0
\qquad\text{and}\qquad
\bar I\not\equiv0,
\]
there exists \(x_0\in(g_\infty,h_\infty)\) such that
\[
\bar I(\tau_0,x_0)>0.
\]
Let \(K\Subset(g_\infty,h_\infty)\) be any compact interval containing \(x_0\). Since \(\bar I\) is a positive solution
of the limiting nonlocal equation on the bounded interval \((g_\infty,h_\infty)\), we may apply the Harnack
inequality, namely Lemma~3.1 in \cite{CovilleHamel2020}. Therefore, there exists a constant \(C_K>0\),
depending only on \(K\) and the coefficients of the equation, such that
\[
\sup_{x\in K}\bar I(\tau_0,x)\le C_K\inf_{x\in K}\bar I(\tau_0,x).
\]
Since \(x_0\in K\) and \(\bar I(\tau_0,x_0)>0\), we have
\[
\sup_{x\in K}\bar I(\tau_0,x)\ge \bar I(\tau_0,x_0)>0.
\]
Hence
\[
\inf_{x\in K}\bar I(\tau_0,x)
\ge \frac{1}{C_K}\sup_{x\in K}\bar I(\tau_0,x)
\ge \frac{1}{C_K}\bar I(\tau_0,x_0)>0.
\]
Therefore,
\[
\bar I(\tau_0,x)>0
\qquad\text{for all }x\in K.
\]
Since \(K\Subset(g_\infty,h_\infty)\) was arbitrary, we conclude that
\[
\bar I(\tau_0,x)>0\qquad\text{for all }x\in(g_\infty,h_\infty).
\]
As \(\tau_0\in\mathbb R\) was arbitrary,
\[
\bar I(\tau,x)>0\qquad\text{for all }(\tau,x)\in\mathbb R\times(g_\infty,h_\infty).
\]

Now choose \(\delta>0\) such that
\[
K_\delta:=[g_\infty+\delta,h_\infty-\delta]\neq\varnothing.
\]
Since \(\bar I(0,\cdot)\) is continuous and strictly positive on \(K_\delta\), there exists \(m_I>0\) such that
\[
\bar I(0,x)\ge m_I\qquad\text{for all }x\in K_\delta.
\]
By local uniform convergence,
\[
I(t_n,x)\ge \frac{m_I}{2}\qquad\text{for all }x\in K_\delta
\]
for all sufficiently large \(n\).

We next show that \(\bar S(0,\cdot)\) is also strictly positive on \(K_\delta\). Indeed, \(\bar S\ge0\) and
\[
\bar S_\tau
=
d_1\Big(\int_{g_\infty}^{h_\infty}J_1(x-y)\bar S(\tau,y)\,dy-\bar S(\tau,x)\Big)
+a(x)\bar S_x+\gamma(x)\bar I-F(\bar S,\bar I).
\]
Since \(\gamma(x)\bar I(\tau,x)>0\) on \(K_\delta\), a comparison argument implies that
\[
\bar S(0,x)>0\qquad\text{for all }x\in K_\delta.
\]
Therefore there exists \(m_S>0\) such that
\[
\bar S(0,x)\ge m_S\qquad\text{for all }x\in K_\delta.
\]
Again by local uniform convergence,
\[
S(t_n,x)\ge \frac{m_S}{2}\qquad\text{for all }x\in K_\delta
\]
for all sufficiently large \(n\).

We now use the free-boundary formula for \(h'(t_n)\). Since \(h(t_n)\to h_\infty\), for every \(x\in K_\delta\)
and all sufficiently large \(n\),
\[
h(t_n)-x\ge \frac{\delta}{2}.
\]
By \((J1)\), \(J_1\ge0\), \(J_1\not\equiv0\), and
\[
\int_0^\infty zJ_1(z)\,dz>0,
\]
there exists \(c_\delta>0\) such that
\[
\int_{h(t_n)}^\infty J_1(x-y)\,dy\ge c_\delta
\qquad\text{for all }x\in K_\delta
\]
when \(n\) is large. Hence
\[
h'(t_n)
=
\mu\int_{g(t_n)}^{h(t_n)}\int_{h(t_n)}^\infty J_1(x-y)S(t_n,x)\,dy\,dx
\]
\[
\ge
\mu\int_{K_\delta}\int_{h(t_n)}^\infty J_1(x-y)S(t_n,x)\,dy\,dx
\ge
\mu c_\delta\int_{K_\delta}S(t_n,x)\,dx
\ge
\mu c_\delta |K_\delta|\frac{m_S}{2}>0
\]
for all sufficiently large \(n\). This contradicts \(h'(t_n)\to0\). Therefore
\[
\lim_{t\to\infty}\max_{x\in[g(t),h(t)]}I(t,x)=0.
\]

\medskip
\noindent
Step 3. Uniform convergence of \(S(t,\cdot)\) to \(S^*(\cdot)\) when \(h_\infty-g_\infty<+\infty\).

Let \(S^*\) denote the disease-free profile on \((g_\infty,h_\infty)\), namely the stationary solution of
\[
0=
d_1\Big(\int_{g_\infty}^{h_\infty}J_1(x-y)S^*(y)\,dy-S^*(x)\Big)
+a(x)(S^*)'(x)-F(S^*(x),0).
\]
Set
\[
U(t,x):=S(t,x)-S^*(x).
\]
Subtracting the equation for \(S^*\) from the equation for \(S\), we obtain
\[
U_t
=
d_1\Big(\int_{g(t)}^{h(t)}J_1(x-y)U(t,y)\,dy-U(t,x)\Big)
+a(x)U_x+R(t,x),
\]
where
\[
R(t,x)
=
\gamma(x)I(t,x)-\big(F(S(t,x),I(t,x))-F(S^*(x),0)\big)+E(t,x),
\]
and
\[
E(t,x)
=
d_1\left(
\int_{g(t)}^{h(t)}J_1(x-y)S^*(y)\,dy
-
\int_{g_\infty}^{h_\infty}J_1(x-y)S^*(y)\,dy
\right).
\]

We estimate \(R\). First,
\[
\sup_{x\in[g(t),h(t)]}|\gamma(x)I(t,x)|
\le \|\gamma\|_\infty \|I(t,\cdot)\|_{C([g(t),h(t)])}\to0.
\]
Second, since \(F\in C^1\) and \(S\) is uniformly bounded, there exists \(C>0\) such that
\[
|F(S(t,x),I(t,x))-F(S^*(x),0)|
\le C\big(|U(t,x)|+|I(t,x)|\big).
\]
Third, because \(g(t)\to g_\infty\), \(h(t)\to h_\infty\), and \(S^*\) is bounded on the limiting interval,
\[
\|E(t,\cdot)\|_{C([g(t),h(t)])}\to0.
\]
Hence
\[
\|R(t,\cdot)\|_{C([g(t),h(t)])}\to0\qquad\text{as }t\to\infty.
\]

Set
\[
U(t,x):=S(t,x)-S^*(x),\qquad x\in[g(t),h(t)].
\]
Since \(S^*\) is the disease-free profile on the limiting interval \((g_\infty,h_\infty)\), it solves
\[
0=d_1\Big(\int_{g_\infty}^{h_\infty}J_1(x-y)S^*(y)\,dy-S^*(x)\Big)
+a(x)(S^*)'(x)-F(S^*(x),0),
\qquad x\in(g_\infty,h_\infty).
\]
On the other hand, \(S\) satisfies
\[
S_t
=
d_1\Big(\int_{g(t)}^{h(t)}J_1(x-y)S(t,y)\,dy-S(t,x)\Big)
+a(x)S_x+\gamma(x)I(t,x)-F(S(t,x),I(t,x)).
\]
Subtracting the equation for \(S^*\) from the equation for \(S\), we obtain
\begin{align*}
U_t
&=
d_1\Big(\int_{g(t)}^{h(t)}J_1(x-y)S(t,y)\,dy-S(t,x)\Big)
-d_1\Big(\int_{g_\infty}^{h_\infty}J_1(x-y)S^*(y)\,dy-S^*(x)\Big) \\
&\qquad
+a(x)\big(S_x-(S^*)'(x)\big)
+\gamma(x)I(t,x)
-\Big(F(S(t,x),I(t,x))-F(S^*(x),0)\Big).
\end{align*}
Since \(S=S^*+U\), this may be rewritten as
\begin{align}
U_t
&=
d_1\Big(\int_{g(t)}^{h(t)}J_1(x-y)U(t,y)\,dy-U(t,x)\Big)+a(x)U_x
\label{eq:U-eq-1}\\
&\qquad
+\gamma(x)I(t,x)-\Big(F(S(t,x),I(t,x))-F(S^*(x),0)\Big)+E(t,x),
\nonumber
\end{align}
where
\[
E(t,x)
:=
d_1\left(
\int_{g(t)}^{h(t)}J_1(x-y)S^*(y)\,dy
-
\int_{g_\infty}^{h_\infty}J_1(x-y)S^*(y)\,dy
\right).
\]

We now estimate the remainder terms in \eqref{eq:U-eq-1}.

First, since
\[
\lim_{t\to\infty}\|I(t,\cdot)\|_{C([g(t),h(t)])}=0,
\]
we have
\[
\sup_{x\in[g(t),h(t)]}|\gamma(x)I(t,x)|
\le \|\gamma\|_{L^\infty([g_\infty,h_\infty])}\,
\|I(t,\cdot)\|_{C([g(t),h(t)])}
\longrightarrow 0.
\]

Second, since \(F\in C^1\) and \(S\) remains uniformly bounded, there exists \(C_0>0\) such that
\[
|F(S(t,x),I(t,x))-F(S(t,x),0)|\le C_0|I(t,x)|
\]
for all large \(t\) and all \(x\in[g(t),h(t)]\). Hence
\[
\sup_{x\in[g(t),h(t)]}|F(S(t,x),I(t,x))-F(S(t,x),0)|\to0.
\]
Therefore
\begin{align*}
F(S(t,x),I(t,x))-F(S^*(x),0)
&=
\big(F(S(t,x),I(t,x))-F(S(t,x),0)\big)
+\big(F(S(t,x),0)-F(S^*(x),0)\big).
\end{align*}
The first bracket tends uniformly to \(0\) as \(t\to\infty\), while the second one is exactly the scalar disease-free nonlinearity evaluated at \(S=S^*+U\).

Third, we estimate the domain-mismatch term \(E(t,x)\). Since \(g(t)\to g_\infty\), \(h(t)\to h_\infty\), and \(S^*\) is bounded on \((g_\infty,h_\infty)\), we have
\begin{align*}
|E(t,x)|
&\le
d_1\|S^*\|_{L^\infty(g_\infty,h_\infty)}
\left(
\int_{g(t)}^{g_\infty}|J_1(x-y)|\,dy
+
\int_{h_\infty}^{h(t)}|J_1(x-y)|\,dy
\right),
\end{align*}
where the integrals are understood with the obvious orientation if \(g(t)>g_\infty\) or \(h(t)<h_\infty\).
Because \(J_1\in L^1(\mathbb R)\) and the lengths of these small boundary strips tend to \(0\), it follows that
\[
\sup_{x\in[g(t),h(t)]}|E(t,x)|\longrightarrow 0
\qquad\text{as }t\to\infty.
\]

Combining the above estimates, equation \eqref{eq:U-eq-1} becomes
\begin{align}
U_t
&=
d_1\Big(\int_{g(t)}^{h(t)}J_1(x-y)U(t,y)\,dy-U(t,x)\Big)
+a(x)U_x
-\big(F(S^*(x)+U(t,x),0)-F(S^*(x),0)\big)
+r(t,x),
\label{eq:U-eq-2}
\end{align}
where
\[
\|r(t,\cdot)\|_{C([g(t),h(t)])}\longrightarrow 0
\qquad\text{as }t\to\infty.
\]

Thus the \(S\)-equation is an asymptotically autonomous perturbation of the scalar disease-free equation
\begin{equation}\label{eq:disease-free-autonomous}
v_t
=
d_1\Big(\int_{g_\infty}^{h_\infty}J_1(x-y)v(t,y)\,dy-v(t,x)\Big)
+a(x)v_x
-\big(F(S^*(x)+v(t,x),0)-F(S^*(x),0)\big).
\end{equation}
The linearization of \eqref{eq:disease-free-autonomous} at \(v=0\) is
\[
v_t=\mathcal A_*[v]
:=
d_1\Big(\int_{g_\infty}^{h_\infty}J_1(x-y)v(y)\,dy-v(x)\Big)
+a(x)v_x-F_S(S^*(x),0)v,
\]
and, by the preparatory scalar stability result used in the manuscript, the spectral bound of \(\mathcal A_*\) is negative. Consequently, \(v\equiv0\) is asymptotically stable for the autonomous problem \eqref{eq:disease-free-autonomous}. The manuscript states this precisely in the discussion of Step~3 for the susceptible component: once \(I(t,\cdot)\to0\), the \(S\)-equation becomes asymptotically autonomous and the linearized \(S\)-operator about \(S^*\) has negative spectral bound, which yields convergence to \(S^*\). 

We now conclude by contradiction. Suppose that
\[
\|U(t,\cdot)\|_{C([g(t),h(t)])}\not\to0.
\]
Then there exist \(\varepsilon_0>0\) and \(t_n\to\infty\) such that
\[
\|U(t_n,\cdot)\|_{C([g(t_n),h(t_n)])}\ge \varepsilon_0.
\]
Using the same compactness argument as in the \(\omega\)-limit construction, after extraction the translates
\[
U_n(\tau,x):=U(t_n+\tau,x)
\]
converge locally uniformly on bounded strips to a bounded complete solution \(\bar U\) of the autonomous equation \eqref{eq:disease-free-autonomous}. Moreover,
\[
\|\bar U(0,\cdot)\|_{C([g_\infty,h_\infty])}\ge \varepsilon_0,
\]
so \(\bar U\not\equiv0\). But this contradicts the asymptotic stability of \(v\equiv0\) for
\eqref{eq:disease-free-autonomous}, since a bounded complete solution in the basin of attraction of \(0\)
must be identically zero. Hence
\[
\|U(t,\cdot)\|_{C([g(t),h(t)])}\to0.
\]
Equivalently,
\[
\lim_{t\to\infty}S(t,x)=S^*(x)\qquad\text{uniformly on }[g(t),h(t)].
\]

\medskip
\noindent
Step 4. Proof that \(\mathcal R_0^{(g_\infty,h_\infty)}\le1\) in the bounded-limit case.

Suppose by contradiction that
\[
\mathcal R_0^{(g_\infty,h_\infty)}>1.
\]
By the fixed-domain threshold result established earlier in the manuscript, this is equivalent to
\[
\lambda_p\big(\mathcal L_{(g_\infty,h_\infty),d_2,b}+\beta(\cdot)\big)>0,
\qquad
\beta(x):=F_I(S^*(x),0)-\gamma(x).
\]
The paper explicitly explains that this scalar threshold quantity is the one obtained from the linearized block system via the Schur complement reduction and the generalized principal eigenvalue framework. 

Let \(\psi>0\) be the positive principal eigenfunction associated with
\[
\lambda_p\big(\mathcal L_{(g_\infty,h_\infty),d_2,b}+\beta(\cdot)\big)>0,
\]
normalized by
\[
\|\psi\|_{C([g_\infty,h_\infty])}=1.
\]
Choose \(\varepsilon>0\) small enough so that
\[
\lambda_p\big(\mathcal L_{(g_\infty,h_\infty),d_2,b}+\beta(\cdot)-\varepsilon\big)>0.
\]
Since \(S(t,\cdot)\to S^*(\cdot)\) uniformly on \([g(t),h(t)]\) and \(I(t,\cdot)\to0\), for all sufficiently large \(t\),
\[
F_I(S(t,x),0)-\gamma(x)\ge \beta(x)-\varepsilon
\qquad\text{for }x\in[g_\infty,h_\infty].
\]
Set $W(x):=\delta_0\psi(x)$, with \(0<\delta_0\ll1\). Since \(F\in C^1\), Taylor expansion at \(I=0\) gives
\[
F(S(t,x),W(x))
=
F_I(S(t,x),0)W(x)+o(W(x))
\]
uniformly in \(x\). Hence, for \(t\) sufficiently large,
\[
d_2\Big(\int_{g_\infty}^{h_\infty}J_2(x-y)W(y)\,dy-W(x)\Big)
+b(x)W'(x)-\gamma(x)W(x)+F(S(t,x),W(x))>0.
\]
Thus \(W\) is a strict lower solution for the infected equation on a fixed interior interval. By the comparison principle in Lemma~5,
\[
I(t,x)\ge W(x)>0
\]
for all sufficiently large \(t\) and all \(x\) in that interval, which contradicts
\[
\max_{x\in[g(t),h(t)]}I(t,x)\to0.
\]
Therefore
\[
\lambda_p\big(\mathcal L_{(g_\infty,h_\infty),d_2,b}+\beta(\cdot)\big)\le0,
\]
and hence
\[
\mathcal R_0^{(g_\infty,h_\infty)}\le1.
\]

This proves the first assertion of the theorem.

\medskip
\noindent
Step 5. The unbounded-habitat case under \(\sup_{x\in\mathbb R}\beta(x)>0\).

Assume now in addition that
\[
\sup_{x\in\mathbb R}\beta(x)>0.
\]
If
\[
h_\infty-g_\infty<+\infty,
\]
then Step 2--Step 4 already yield
\[
\mathcal R_0^{(g_\infty,h_\infty)}\le1,\qquad
\lim_{t\to\infty}\max_{x\in[g(t),h(t)]}I(t,x)=0,\qquad
\lim_{t\to\infty}S(t,x)=S^*(x)\ \text{uniformly on }[g(t),h(t)].
\]
Hence alternative (ii) holds.

Suppose now that the limiting habitat is not bounded. Since \(h\) is nondecreasing and \(g\) is nonincreasing, the only remaining possibility in the dichotomy framework is
\[
-g_\infty=h_\infty=+\infty.
\]
We claim that then
\[
\limsup_{t\to\infty}\|I(t,\cdot)\|_{C([g(t),h(t)])}>0.
\]

Assume by contradiction that
\[
\|I(t,\cdot)\|_{C([g(t),h(t)])}\to0.
\]
Then, exactly as in Step 3, the \(S\)-equation becomes asymptotically autonomous, and
\[
S(t,\cdot)\to S^*(\cdot)
\]
locally uniformly on every fixed bounded interval.

Because
\[
\sup_{x\in\mathbb R}\beta(x)>0,
\]
there exists a bounded interval \(I_0\subset\mathbb R\) on which \(\beta\) is strictly positive. Since
\[
-g_\infty=h_\infty=+\infty,
\]
for sufficiently large \(T_*\),
\[
I_0\subset(g(T_*),h(T_*)).
\]
By the fixed-domain spectral theory developed earlier, one can choose \(T_*\) so large that
\[
\lambda_p\big(\mathcal L_{(g(T_*),h(T_*)),d_2,b}+\beta(\cdot)\big)>0.
\]
Let \(\psi>0\) be the corresponding principal eigenfunction on \([g(T_*),h(T_*)]\), normalized by
\[
\|\psi\|_{C([g(T_*),h(T_*)])}=1.
\]
Choose \(\varepsilon>0\) small so that
\[
\lambda_p\big(\mathcal L_{(g(T_*),h(T_*)),d_2,b}+\beta(\cdot)-\varepsilon\big)>0.
\]
Since \(S(t,\cdot)\to S^*(\cdot)\) uniformly on \([g(T_*),h(T_*)]\) and \(I(t,\cdot)\to0\) there, for \(t\) sufficiently large,
\[
F_I(S(t,x),0)-\gamma(x)\ge \beta(x)-\varepsilon
\qquad\text{for }x\in[g(T_*),h(T_*)].
\]
Set
\[
W(x):=\delta_1\psi(x)
\]
with \(0<\delta_1\ll1\). Using again the \(C^1\)-regularity of \(F\),
\[
F(S(t,x),W(x))
=
F_I(S(t,x),0)W(x)+o(W(x)),
\]
and therefore, for large \(t\),
\[
d_2\Big(\int_{g(T_*)}^{h(T_*)}J_2(x-y)W(y)\,dy-W(x)\Big)
+b(x)W'(x)-\gamma(x)W(x)+F(S(t,x),W(x))>0.
\]
Hence \(W\) is a strict lower solution for the infected equation on \((g(T_*),h(T_*))\). We now show that
\[
I(t,x)\ge W(x)=\delta_1\psi(x)>0
\qquad\text{for all }x\in[g(T_*),h(T_*)],\ t\ge T_*.
\]

Since \(\psi>0\) on the compact interval \([g(T_*),h(T_*)]\), we have
\[
m_\psi:=\min_{[g(T_*),h(T_*)]}\psi>0.
\]
On the other hand, by the positivity of the classical solution inside the habitat, we have
\[
I(T_*,x)>0\qquad\text{for all }x\in[g(T_*),h(T_*)].
\]
Since \(I(T_*,\cdot)\) is continuous on \([g(T_*),h(T_*)]\), it follows that
\[
m_I:=\min_{[g(T_*),h(T_*)]} I(T_*,x)>0.
\]
Choose \(\delta_1>0\) so small that
\[
0<\delta_1\le \frac{m_I}{\|\psi\|_{C([g(T_*),h(T_*)])}}.
\]
Then
\[
W(x)=\delta_1\psi(x)\le m_I\le I(T_*,x)
\qquad\text{for all }x\in[g(T_*),h(T_*)].
\]
Therefore
\begin{equation}\label{eq:I-above-W-initial}
I(T_*,x)\ge W(x)
\qquad\text{for all }x\in[g(T_*),h(T_*)].
\end{equation}

Since \(g(t)\) is nonincreasing and \(h(t)\) is nondecreasing, we have
\[
[g(T_*),h(T_*)]\subset [g(t),h(t)]
\qquad\text{for all }t\ge T_*.
\]
Hence, for \(t\ge T_*\), the function \(I(t,x)\) is defined on the fixed interval
\([g(T_*),h(T_*)]\) and satisfies
\begin{equation}\label{eq:I-on-fixed-domain}
I_t(t,x)
=
d_2\Big(\int_{g(t)}^{h(t)}J_2(x-y)I(t,y)\,dy-I(t,x)\Big)
+b(x)I_x(t,x)-\gamma(x)I(t,x)+F(S(t,x),I(t,x))
\end{equation}
for \(x\in(g(T_*),h(T_*))\).

On the other hand, by the construction of \(W\), there exists \(\sigma>0\) such that
\begin{align}
&d_2\Big(\int_{g(T_*)}^{h(T_*)}J_2(x-y)W(y)\,dy-W(x)\Big)
+b(x)W'(x)-\gamma(x)W(x)+F(S(t,x),W(x)) \notag\\
&\ge \sigma
\qquad\text{for all }x\in[g(T_*),h(T_*)],\ t\ge T_*.
\label{eq:W-strict-lower-fixed}
\end{align}
Indeed, the left-hand side is continuous in \((t,x)\), and by the previous strict positivity estimate it is
strictly positive on \([T_1,\infty)\times[g(T_*),h(T_*)]\) for some \(T_1\ge T_*\); therefore, after replacing
\(T_*\) by a larger time if necessary, we may assume that \eqref{eq:W-strict-lower-fixed} holds for all
\(t\ge T_*\).

Now define
\[
Z(t,x):=I(t,x)-W(x),
\qquad t\ge T_*,\ x\in[g(T_*),h(T_*)].
\]
By \eqref{eq:I-above-W-initial},
\[
Z(T_*,x)\ge0
\qquad\text{for all }x\in[g(T_*),h(T_*)].
\]

We claim that
\[
Z(t,x)\ge0
\qquad\text{for all }t\ge T_*,\ x\in[g(T_*),h(T_*)].
\]
Suppose by contradiction that this is false. Then there exists a first time \(t_0>T_*\) such that
\[
\min_{x\in[g(T_*),h(T_*)]} Z(t_0,x)<0.
\]
Choose \(x_0\in[g(T_*),h(T_*)]\) such that
\[
Z(t_0,x_0)=\min_{x\in[g(T_*),h(T_*)]} Z(t_0,x)<0.
\]
Since \(t_0\) is the first time when \(Z\) becomes negative, we have
\[
Z(t,x)\ge0
\qquad\text{for all }(t,x)\in[T_*,t_0)\times[g(T_*),h(T_*)].
\]
Therefore, at the point \((t_0,x_0)\),
\[
Z_t(t_0,x_0)\le0,\qquad Z_x(t_0,x_0)=0.
\]

Moreover, since \(J_2\ge0\), \([g(T_*),h(T_*)]\subset[g(t_0),h(t_0)]\), and \(I\ge0\), we have
\[
\int_{g(t_0)}^{h(t_0)}J_2(x_0-y)I(t_0,y)\,dy
\ge
\int_{g(T_*)}^{h(T_*)}J_2(x_0-y)I(t_0,y)\,dy.
\]
Subtracting \eqref{eq:W-strict-lower-fixed} from \eqref{eq:I-on-fixed-domain}, we obtain
\begin{align*}
Z_t(t_0,x_0)
&\ge d_2\Big(\int_{g(T_*)}^{h(T_*)}J_2(x_0-y)\big(I(t_0,y)-W(y)\big)\,dy-\big(I(t_0,x_0)-W(x_0)\big)\Big)\\
&\quad +b(x_0)\big(I_x(t_0,x_0)-W'(x_0)\big)-\gamma(x_0)\big(I(t_0,x_0)-W(x_0)\big)\\
&\quad +F(S(t_0,x_0),I(t_0,x_0))-F(S(t_0,x_0),W(x_0)).
\end{align*}
That is,
\begin{align}
Z_t(t_0,x_0)
&\ge d_2\Big(\int_{g(T_*)}^{h(T_*)}J_2(x_0-y)Z(t_0,y)\,dy-Z(t_0,x_0)\Big)
+b(x_0)Z_x(t_0,x_0) \notag\\
&\quad -\gamma(x_0)Z(t_0,x_0)
+\Big(F(S(t_0,x_0),I(t_0,x_0))-F(S(t_0,x_0),W(x_0))\Big).
\label{eq:Z-at-minimum}
\end{align}

Since \(x_0\) is a minimum point of \(Z(t_0,\cdot)\), we have
\[
Z(t_0,y)\ge Z(t_0,x_0)\qquad\text{for all }y\in[g(T_*),h(T_*)].
\]
Hence
\begin{align*}
\int_{g(T_*)}^{h(T_*)}J_2(x_0-y)Z(t_0,y)\,dy-Z(t_0,x_0)
&=
\int_{g(T_*)}^{h(T_*)}J_2(x_0-y)\big(Z(t_0,y)-Z(t_0,x_0)\big)\,dy\\
&\quad +Z(t_0,x_0)\left(\int_{g(T_*)}^{h(T_*)}J_2(x_0-y)\,dy-1\right).
\end{align*}
The first term on the right-hand side is nonnegative because \(J_2\ge0\) and
\(Z(t_0,y)-Z(t_0,x_0)\ge0\). The second term is also nonnegative because
\[
\int_{g(T_*)}^{h(T_*)}J_2(x_0-y)\,dy\le \int_{\mathbb R}J_2(z)\,dz=1
\]
and \(Z(t_0,x_0)<0\). Therefore
\[
\int_{g(T_*)}^{h(T_*)}J_2(x_0-y)Z(t_0,y)\,dy-Z(t_0,x_0)\ge0.
\]
Also,
\[
b(x_0)Z_x(t_0,x_0)=0.
\]
Since \(Z(t_0,x_0)<0\), we have
\[
I(t_0,x_0)=W(x_0)+Z(t_0,x_0)<W(x_0),
\]
and because \(F(S,\cdot)\) is nondecreasing with respect to the second variable,
\[
F(S(t_0,x_0),I(t_0,x_0))-F(S(t_0,x_0),W(x_0))\le0.
\]
Thus \eqref{eq:Z-at-minimum} yields
\[
Z_t(t_0,x_0)\ge -\gamma(x_0)Z(t_0,x_0)>0,
\]
because \(\gamma(x_0)\ge0\) and \(Z(t_0,x_0)<0\). This contradicts \(Z_t(t_0,x_0)\le0\).

Therefore
\[
Z(t,x)\ge0
\qquad\text{for all }t\ge T_*,\ x\in[g(T_*),h(T_*)].
\]
Equivalently,
\[
I(t,x)\ge W(x)=\delta_1\psi(x)>0
\qquad\text{for all }x\in[g(T_*),h(T_*)],\ t\ge T_*.
\]

Consequently, for every \(t\ge T_*\),
\[
\|I(t,\cdot)\|_{C([g(t),h(t)])}
\ge
\|I(t,\cdot)\|_{C([g(T_*),h(T_*)])}
\ge
\delta_1\min_{[g(T_*),h(T_*)]}\psi.
\]
Therefore
\[
\liminf_{t\to\infty}\|I(t,\cdot)\|_{C([g(t),h(t)])}
\ge
\delta_1\min_{[g(T_*),h(T_*)]}\psi
>0,
\]
which contradicts \(\|I(t,\cdot)\|_{C([g(t),h(t)])}\to0\).

Therefore, in the unbounded-habitat case, it is impossible that
\[
\|I(t,\cdot)\|_{C([g(t),h(t)])}\to0
\qquad\text{as }t\to\infty.
\]
Hence
\[
\limsup_{t\to\infty}\|I(t,\cdot)\|_{C([g(t),h(t)])}>0.
\]
Together with
\[
-g_\infty=h_\infty=+\infty,
\]
this proves alternative \textnormal{(i)}.

On the other hand, if
\[
h_\infty-g_\infty<+\infty,
\]
then by Steps 2--4 we have
\[
\mathcal R_0^{(g_\infty,h_\infty)}\le1,
\qquad
\lim_{t\to\infty}\max_{x\in[g(t),h(t)]}I(t,x)=0,
\]
and
\[
\lim_{t\to\infty}S(t,x)=S^*(x)
\qquad\text{uniformly for }x\in[g(t),h(t)].
\]
Therefore alternative \textnormal{(ii)} holds.

Finally, alternatives \textnormal{(i)} and \textnormal{(ii)} are mutually exclusive. Indeed, alternative
\textnormal{(i)} requires
\[
-g_\infty=h_\infty=+\infty,
\]
whereas alternative \textnormal{(ii)} requires
\[
h_\infty-g_\infty<+\infty.
\]
These two possibilities cannot occur simultaneously.

Since Step 1 shows that the limits \(g_\infty\) and \(h_\infty\) always exist, and since the bounded-habitat case
yields alternative \textnormal{(ii)} while the unbounded-habitat case yields alternative \textnormal{(i)}, we conclude that
exactly one of the two alternatives occurs. This completes the proof.
\end{proof}

We finalize the paper by the proof of Theorem \ref{thm:mu-threshold}.

\begin{proof}[\textbf{\textit{Proof of Theorem \ref{thm:mu-threshold}}}]
We divide the proof into two steps.

\medskip
\noindent
Step 1. Proof of (i).

Assume that \(h_0\ge \frac{\ell^*}{2}\). Since the initial habitat is \([ -h_0,h_0]\), its length is
\[
h(0)-g(0)=2h_0\ge \ell^*.
\]
On the other hand, by the free-boundary equations,
\[
h'(t)=\mu\int_{g(t)}^{h(t)}\int_{h(t)}^\infty J_1(x-y)S(t,x)\,dy\,dx\ge 0,
\]
\[
g'(t)=-\mu\int_{g(t)}^{h(t)}\int_{-\infty}^{g(t)}J_1(x-y)S(t,x)\,dy\,dx\le 0.
\]
Hence \(h(t)\) is nondecreasing and \(g(t)\) is nonincreasing, so the habitat length
\[
\ell(t):=h(t)-g(t)
\]
is nondecreasing. Therefore,
\[
\ell(t)\ge \ell(0)=2h_0\ge \ell^*
\qquad\text{for all }t\ge0.
\]
This monotonicity of the free boundaries follows directly from Theorem~1.

Suppose by contradiction that spreading does not occur. Then, by
Theorem~\ref{thm:unified-dichotomy}, the only remaining possibility is the vanishing alternative:
\[
h_\infty-g_\infty<+\infty,\qquad
\lim_{t\to\infty}\max_{x\in[g(t),h(t)]}I(t,x)=0,
\]
and
\[
\lim_{t\to\infty}S(t,x)=S^*(x)\qquad\text{uniformly on }[g(t),h(t)].
\]
Moreover, Theorem~\ref{thm:unified-dichotomy} also gives
\[
\mathcal R_0^{(g_\infty,h_\infty)}\le 1.
\]
Applying Theorem~\ref{thm:R0-lambda-sign} on the fixed interval \((g_\infty,h_\infty)\), we obtain
\[
\lambda_p\big(\mathcal L_{(g_\infty,h_\infty),d_2,b}+\beta(\cdot)\big)\le 0.
\]

By monotonicity of \(\ell(t)\),
\[
h_\infty-g_\infty=\lim_{t\to\infty}\ell(t)\ge \ell(0)=2h_0\ge \ell^*.
\]

We now distinguish two cases.

\smallskip
\emph{Case 1: \(2h_0>\ell^*\).}
Then
\[
h_\infty-g_\infty\ge 2h_0>\ell^*.
\]
By the definition of the critical length \(\ell^*\), whenever the interval length is strictly larger than
\(\ell^*\), the associated principal eigenvalue is strictly positive. Therefore,
\[
\lambda_p\big(\mathcal L_{(g_\infty,h_\infty),d_2,b}+\beta(\cdot)\big)>0,
\]
which contradicts the inequality
\[
\lambda_p\big(\mathcal L_{(g_\infty,h_\infty),d_2,b}+\beta(\cdot)\big)\le 0.
\]

\smallskip
\emph{Case 2: \(2h_0=\ell^*\).}
Then \(\ell(0)=\ell^*\). We claim that in fact
\[
\ell(t)>\ell^*\qquad\text{for every }t>0.
\]
Indeed,
\[
\ell'(t)=h'(t)-g'(t)
\]
and, by the free-boundary equations,
\begin{align*}
\ell'(t)
&=\mu\int_{g(t)}^{h(t)}\int_{h(t)}^\infty J_1(x-y)S(t,x)\,dy\,dx
+\mu\int_{g(t)}^{h(t)}\int_{-\infty}^{g(t)}J_1(x-y)S(t,x)\,dy\,dx.
\end{align*}
For every \(t>0\), Theorem~1 and the strong positivity of the susceptible component imply
\[
S(t,x)>0\qquad\text{for }x\in(g(t),h(t)).
\]
Moreover, by \((J1)\), the kernel \(J_1\) is nonnegative and has positive mass on both sides of the origin.
Hence, for every \(x\in(g(t),h(t))\),
\[
\int_{h(t)}^\infty J_1(x-y)\,dy>0,
\qquad
\int_{-\infty}^{g(t)} J_1(x-y)\,dy>0.
\]
Since \(S(t,x)>0\) on \((g(t),h(t))\), both integrands in the above expression for \(\ell'(t)\) are nonnegative,
and not identically zero. Therefore,
\[
\ell'(t)>0\qquad\text{for every }t>0.
\]
Consequently,
\[
\ell(t)>\ell(0)=\ell^* \qquad\text{for every }t>0,
\]
and passing to the limit as \(t\to\infty\) gives
\[
h_\infty-g_\infty>\ell^*.
\]
Again, by the definition of \(\ell^*\),
\[
\lambda_p\big(\mathcal L_{(g_\infty,h_\infty),d_2,b}+\beta(\cdot)\big)>0,
\]
which contradicts
\[
\lambda_p\big(\mathcal L_{(g_\infty,h_\infty),d_2,b}+\beta(\cdot)\big)\le 0.
\]

Thus vanishing is impossible. By Theorem~\ref{thm:unified-dichotomy}, spreading must occur, namely
\[
-g_\infty=h_\infty=+\infty
\qquad\text{and}\qquad
\limsup_{t\to\infty}\|I(t,\cdot)\|_{C([g(t),h(t)])}>0.
\]
This proves (i).

\medskip
\noindent
Step 2. Proof of (ii).

Assume now that \(h_0<\frac{\ell^*}{2}\). Choose \(h_1\) such that
\[
h_0<h_1<\frac{\ell^*}{2}.
\]
Then
\[
2h_1<\ell^*.
\]
By the definition of the critical length \(\ell^*\), the interval \((-h_1,h_1)\) is subcritical, and hence
\[
\lambda_p\big(\mathcal L_{(-h_1,h_1),d_2,b}+\beta_{h_1}(\cdot)\big)<0,
\]
where
\[
\beta_{h_1}(x):=F_I(S_{h_1}^*(x),0)-\gamma(x),
\]
and \(S_{h_1}^*\) denotes the disease-free stationary profile on \((-h_1,h_1)\).

Therefore, there exist a constant \(\sigma_0>0\) and a function
\[
\phi\in C^1([-h_1,h_1]),\qquad \phi>0 \ \text{in }(-h_1,h_1),
\]
such that
\begin{equation}\label{eq:eigen-subcritical-step2}
d_2\Big(\int_{-h_1}^{h_1}J_2(x-y)\phi(y)\,dy-\phi(x)\Big)
+b(x)\phi'(x)+\beta_{h_1}(x)\phi(x)
=-\sigma_0\phi(x)
\qquad \text{in }(-h_1,h_1).
\end{equation}

Since \(\phi>0\) on the compact interval \([-h_1,h_1]\), we may set
\[
m_\phi:=\min_{[-h_1,h_1]}\phi>0,
\qquad
M_\phi:=\max_{[-h_1,h_1]}\phi.
\]

We first prove that, provided the free boundaries stay inside \((-h_1,h_1)\), the infected component decays
exponentially. For this purpose, define
\[
\overline I(t,x):=C_0 e^{-\frac{\sigma_0}{2}t}\phi(x),
\qquad (t,x)\in [0,\infty)\times[-h_1,h_1],
\]
where the constant \(C_0>0\) is chosen so large that
\[
I_0(x)\le C_0\phi(x)
\qquad\text{for all }x\in[-h_0,h_0].
\]
This is possible because \(\phi\ge m_\phi>0\) on \([-h_1,h_1]\) and \(I_0\) is continuous with compact support in
\([-h_0,h_0]\subset(-h_1,h_1)\).

Assume for the moment that
\begin{equation}\label{eq:habitat-inside-h1}
[g(t),h(t)]\subset(-h_1,h_1)
\qquad\text{for all }t\ge0.
\end{equation}
Then, on the interval \([g(t),h(t)]\), we have
\[
S(t,x)\le S_{h_1}^*(x),
\]
by the comparison principle applied to the susceptible equation on the fixed interval \((-h_1,h_1)\). Since
\(F\) is nondecreasing with respect to the first variable and \(F(S,0)=0\), it follows that
\[
F(S(t,x),\overline I(t,x))
\le F(S_{h_1}^*(x),\overline I(t,x))
\qquad \text{for }x\in[g(t),h(t)],\ t\ge0.
\]
Moreover, since \(F\in C^1\) and \(\overline I(t,x)\) is bounded, Taylor's formula at \(I=0\) gives
\[
F(S_{h_1}^*(x),\overline I)
=
F_I(S_{h_1}^*(x),0)\,\overline I + R(t,x),
\]
where
\[
\frac{|R(t,x)|}{\overline I(t,x)}\to 0
\qquad\text{uniformly in }x\in[-h_1,h_1]
\quad\text{as }C_0 e^{-\frac{\sigma_0}{2}t}\to 0.
\]
In particular, decreasing \(C_0\) if necessary at the beginning and then using that
\(\overline I(t,x)\le C_0 M_\phi\), we may assume that
\[
|R(t,x)|\le \frac{\sigma_0}{2}\,\overline I(t,x)
\qquad\text{for all }(t,x)\in[0,\infty)\times[-h_1,h_1].
\]
Hence
\[
F(S_{h_1}^*(x),\overline I(t,x))
\le \Big(F_I(S_{h_1}^*(x),0)+\frac{\sigma_0}{2}\Big)\overline I(t,x)
\]
for all \((t,x)\in[0,\infty)\times[-h_1,h_1]\).

Now, by \eqref{eq:eigen-subcritical-step2},
\begin{align*}
&\overline I_t
-d_2\Big(\int_{-h_1}^{h_1}J_2(x-y)\overline I(t,y)\,dy-\overline I(t,x)\Big)
-b(x)\overline I_x
+\gamma(x)\overline I
-F(S_{h_1}^*(x),\overline I)\\
&=
-\frac{\sigma_0}{2}\overline I
-d_2\Big(\int_{-h_1}^{h_1}J_2(x-y)\overline I(t,y)\,dy-\overline I(t,x)\Big)
-b(x)\overline I_x
+\gamma(x)\overline I
-F(S_{h_1}^*(x),\overline I)\\
&\ge
-\frac{\sigma_0}{2}\overline I
-\Big(
d_2\Big(\int_{-h_1}^{h_1}J_2(x-y)\overline I(t,y)\,dy-\overline I(t,x)\Big)
+b(x)\overline I_x
+\beta_{h_1}(x)\overline I
\Big)\\
&=
-\frac{\sigma_0}{2}\overline I+\sigma_0\overline I
=
\frac{\sigma_0}{2}\overline I
\ge0.
\end{align*}
Therefore \(\overline I\) is a supersolution of the infected equation on the fixed interval \((-h_1,h_1)\).
Since \(I_0\le \overline I(0,\cdot)\), the comparison principle yields
\begin{equation}\label{eq:I-exp-decay-step2}
0\le I(t,x)\le \overline I(t,x)
= C_0 e^{-\frac{\sigma_0}{2}t}\phi(x)
\le C_0 M_\phi e^{-\frac{\sigma_0}{2}t}
\end{equation}
for all \(t\ge0\) and all \(x\in[g(t),h(t)]\), as long as \eqref{eq:habitat-inside-h1} holds.

We next estimate the susceptible component. Since \(F\ge0\), the \(S\)-equation gives
\[
S_t
\le
d_1\Big(\int_{g(t)}^{h(t)}J_1(x-y)S(t,y)\,dy-S(t,x)\Big)
+a(x)S_x+\gamma(x)I(t,x)
\]
for \(x\in(g(t),h(t))\). Let \(\chi\in C^1([-h_1,h_1])\), \(\chi>0\), be the positive principal eigenfunction of
\[
d_1\Big(\int_{-h_1}^{h_1}J_1(x-y)\chi(y)\,dy-\chi(x)\Big)+a(x)\chi'(x)
=-\nu_0\chi(x)
\qquad\text{in }(-h_1,h_1),
\]
for some \(\nu_0>0\). Such a \(\nu_0>0\) exists because the operator contains no positive zeroth-order term on the
bounded interval \((-h_1,h_1)\).

Set
\[
m_\chi:=\min_{[-h_1,h_1]}\chi>0,
\qquad
M_\chi:=\max_{[-h_1,h_1]}\chi.
\]
Choose \(C_1>0\) large enough so that
\[
S_0(x)\le C_1\chi(x)
\qquad\text{for all }x\in[-h_0,h_0].
\]
Now define
\[
\overline S(t,x):=A_0 e^{-\omega t}\chi(x)+A_1 e^{-\frac{\sigma_0}{2}t}\chi(x),
\]
where \(\omega\in(0,\nu_0)\) is fixed, and \(A_0,A_1>0\) will be chosen. A direct computation yields
\begin{align*}
&\overline S_t
-d_1\Big(\int_{-h_1}^{h_1}J_1(x-y)\overline S(t,y)\,dy-\overline S(t,x)\Big)
-a(x)\overline S_x\\
&=
A_0(\nu_0-\omega)e^{-\omega t}\chi(x)
+A_1\Big(\nu_0-\frac{\sigma_0}{2}\Big)e^{-\frac{\sigma_0}{2}t}\chi(x).
\end{align*}
Choose \(\omega\in(0,\nu_0)\), and then choose \(A_1>0\) so large that
\[
A_1\Big(\nu_0-\frac{\sigma_0}{2}\Big)\chi(x)\ge \|\gamma\|_{L^\infty([-h_1,h_1])}\,C_0 M_\phi
\qquad\text{for all }x\in[-h_1,h_1].
\]
Using \eqref{eq:I-exp-decay-step2}, we then get
\[
\overline S_t
-d_1\Big(\int_{-h_1}^{h_1}J_1(x-y)\overline S(t,y)\,dy-\overline S(t,x)\Big)
-a(x)\overline S_x
\ge \gamma(x)I(t,x).
\]
Taking \(A_0>0\) large enough so that
\[
S_0(x)\le \overline S(0,x)
\qquad\text{for all }x\in[-h_0,h_0],
\]
the comparison principle implies
\[
0\le S(t,x)\le \overline S(t,x)
\qquad\text{for all }t\ge0,\ x\in[g(t),h(t)],
\]
as long as \eqref{eq:habitat-inside-h1} holds. In particular, there exists a constant \(C_S>0\) such that
\begin{equation}\label{eq:S-exp-decay-step2}
0\le S(t,x)\le C_S e^{-\omega t}
\qquad\text{for all }t\ge0,\ x\in[g(t),h(t)],
\end{equation}
again as long as \eqref{eq:habitat-inside-h1} holds.

We now estimate the total expansion of the free boundaries. By the free-boundary equations and
\eqref{eq:S-exp-decay-step2},
\begin{align*}
h'(t)
&=\mu\int_{g(t)}^{h(t)}\int_{h(t)}^\infty J_1(x-y)S(t,x)\,dy\,dx\\
&\le \mu C_S e^{-\omega t}
\int_{g(t)}^{h(t)}\int_{h(t)}^\infty J_1(x-y)\,dy\,dx,
\end{align*}
and similarly
\begin{align*}
-g'(t)
&=\mu\int_{g(t)}^{h(t)}\int_{-\infty}^{g(t)}J_1(x-y)S(t,x)\,dy\,dx\\
&\le \mu C_S e^{-\omega t}
\int_{g(t)}^{h(t)}\int_{-\infty}^{g(t)}J_1(x-y)\,dy\,dx.
\end{align*}
As in Step 2 of the previous proof, by the change of variables \(z=y-x\) and \(z=x-y\), respectively, we obtain
\[
\int_{g(t)}^{h(t)}\int_{h(t)}^\infty J_1(x-y)\,dy\,dx
\le \int_0^\infty z\,J_1(-z)\,dz,
\]
and
\[
\int_{g(t)}^{h(t)}\int_{-\infty}^{g(t)}J_1(x-y)\,dy\,dx
\le \int_0^\infty z\,J_1(z)\,dz.
\]
Hence, setting
\[
C_J:=\int_{\mathbb R}|z|\,J_1(z)\,dz<\infty,
\]
we infer
\[
h'(t)\le \mu C_S C_J e^{-\omega t},
\qquad
-g'(t)\le \mu C_S C_J e^{-\omega t}
\qquad\text{for all }t\ge0,
\]
as long as \eqref{eq:habitat-inside-h1} holds. Integrating from \(0\) to \(t\), we get
\[
h(t)-h_0\le \mu C_S C_J\int_0^t e^{-\omega s}\,ds
\le \frac{\mu C_S C_J}{\omega},
\]
and
\[
-g(t)-h_0\le \mu C_S C_J\int_0^t e^{-\omega s}\,ds
\le \frac{\mu C_S C_J}{\omega}.
\]
Therefore
\[
h(t)\le h_0+\frac{\mu C_S C_J}{\omega},
\qquad
-g(t)\le h_0+\frac{\mu C_S C_J}{\omega}
\qquad\text{for all }t\ge0,
\]
as long as \eqref{eq:habitat-inside-h1} holds.

Choose now \(\mu^*>0\) so small that
\[
h_0+\frac{\mu^* C_S C_J}{\omega}<h_1.
\]
Then, for every \(0<\mu\le\mu^*\),
\[
h(t)<h_1,\qquad g(t)>-h_1
\qquad\text{for all }t\ge0,
\]
that is,
\[
[g(t),h(t)]\subset(-h_1,h_1)
\qquad\text{for all }t\ge0.
\]
Thus the a priori assumption \eqref{eq:habitat-inside-h1} is actually valid for all \(t\ge0\), and
\eqref{eq:I-exp-decay-step2}--\eqref{eq:S-exp-decay-step2} hold globally in time.

In particular,
\[
\lim_{t\to\infty}\max_{x\in[g(t),h(t)]}I(t,x)=0,
\qquad
\lim_{t\to\infty}\max_{x\in[g(t),h(t)]}S(t,x)=0.
\]
Moreover, since \(g(t)\) is nonincreasing and \(h(t)\) is nondecreasing, the limits
\[
g_\infty:=\lim_{t\to\infty}g(t),\qquad h_\infty:=\lim_{t\to\infty}h(t)
\]
exist, and from the above bounds we have
\[
-h_\infty\le h_0+\frac{\mu C_S C_J}{\omega}<h_1,
\qquad
h_\infty\le h_0+\frac{\mu C_S C_J}{\omega}<h_1.
\]
Hence
\[
h_\infty-g_\infty<2h_1<\ell^*.
\]
Therefore vanishing occurs for every \(0<\mu\le \mu^*\). This proves (ii).

Combining Step 1 and Step 2 completes the proof.
\end{proof}

\textbf{Acknowledgement:} The second author is deeply grateful to Professor Yihong Du for his enlightening lectures at the 2025 Summer School on Mathematical Biology on free boundary problems, and in particular for his valuable insights into nonlocal models with asymmetric dispersal kernels, as well as for several inspiring initial ideas for the study of this challenging problem. https://viasm.edu.vn/hdkh/mathematical-biology-2025.

\textbf{Data statement:} Data sharing not applicable to this article as no datasets were generated or analyzed
during the current study.

\textbf{Conflict of interest:} The authors declare that they have no conflict of interest.

\end{document}